\documentclass[11pt,leqno,a4paper]{amsart}

\long\def\forget#1{}

\usepackage{geometry}
\usepackage{pst-plot, pstricks}
\usepackage{fancybox,amssymb,color}
\usepackage{hyperref}
\usepackage[T1]{fontenc}
\usepackage[latin1,utf8]{inputenc}
\usepackage[english]{babel}

\def\theenumi{(\alph{enumi})}

\def\p@enumii{\theenumi}

\usepackage{amsmath}
\usepackage{amssymb}
\usepackage{amscd}
\usepackage{textcomp}
\usepackage{ifthen}

\newcommand{\DS}{\displaystyle}
\newcommand{\TS}{\textstyle}
\newcommand{\SC}{\scriptstyle}
\newcommand{\SSC}{\scriptscriptstyle}

\usepackage{amsfonts}

\newcommand{\BA}{{\mathbb{A}}}
\newcommand{\BB}{{\mathbb{B}}}

\newcommand{\BF}{{\mathbb{F}}}
\newcommand{\BG}{{\mathbb{G}}}

\newcommand{\BN}{{\mathbb{N}}}
\newcommand{\BO}{{\mathbb{O}}}

\newcommand{\BQ}{{\mathbb{Q}}}

\newcommand{\BU}{{\mathbb{U}}}

\newcommand{\BZ}{{\mathbb{Z}}}

\newcommand{\sA}{{\mathscr{A}}}
\newcommand{\sB}{{\mathscr{B}}}

\newcommand{\sD}{{\mathscr{D}}}

\newcommand{\sG}{{\mathscr{G}}}
\newcommand{\sH}{{\mathscr{H}}}

\newcommand{\CB}{{\mathcal{B}}}
\newcommand{\CC}{{\mathcal{C}}}
\newcommand{\CalD}{{\mathcal{D}}}
\newcommand{\CE}{{\mathcal{E}}}
\newcommand{\CF}{{\mathcal{F}}}
\newcommand{\CG}{{\mathcal{G}}}

\newcommand{\CM}{{\mathcal{M}}}
\newcommand{\CN}{{\mathcal{N}}}
\newcommand{\CO}{{\mathcal{O}}}

\newcommand{\CX}{{\mathcal{X}}}

\newcommand{\FS}{{\mathfrak{S}}}

\newcommand{\Fm}{{\mathfrak{m}}}
\newcommand{\Fn}{{\mathfrak{n}}}

\newcommand{\Fp}{{\mathfrak{p}}}
\newcommand{\Fq}{{\mathfrak{q}}}

\newcommand{\q}{\mathfrak{q}}

\DeclareMathOperator{\Ad}{Ad}

\DeclareMathOperator{\End}{End}

\DeclareMathOperator{\Fil}{Fil}
\DeclareMathOperator{\Flag}{Flag}
\DeclareMathOperator{\Frac}{Frac}
\DeclareMathOperator{\Frob}{Frob}
\DeclareMathOperator{\Gal}{Gal}
\DeclareMathOperator{\GL}{GL}
\DeclareMathOperator{\Gr}{Gr}

\DeclareMathOperator{\Hom}{Hom}

\DeclareMathOperator{\Id}{Id}

\DeclareMathOperator{\Isom}{Isom}

\DeclareMathOperator{\Quot}{Quot}

\DeclareMathOperator{\Res}{Res}

\DeclareMathOperator{\Spa}{Spa}
\DeclareMathOperator{\Spec}{Spec}
\DeclareMathOperator{\Spf}{Spf}

\DeclareMathOperator{\Sym}{Sym}
\DeclareMathOperator{\Tor}{Tor}

\DeclareMathOperator{\Var}{V}

\newcommand{\ad}{{\rm ad}}
\newcommand{\alg}{{\rm alg}}

\DeclareMathOperator{\coker}{coker}

\newcommand{\dom}{{\rm dom}}
\newcommand{\dR}{{\rm dR}}

\DeclareMathOperator{\gr}{gr}

\DeclareMathOperator{\id}{\,id}

\newcommand{\lft}{{\rm lft}}
\renewcommand{\mod}{\;{\rm mod}\;}

\DeclareMathOperator{\pr}{pr}

\DeclareMathOperator{\Rad}{Rad}

\DeclareMathOperator{\rk}{rk}
\newcommand{\sep}{{\rm sep}}

\DeclareMathOperator{\trdeg}{trdeg}

\renewcommand{\phi}{\varphi}
\renewcommand{\epsilon}{\varepsilon}

\def\longto{\longrightarrow}
\def\into{\hookrightarrow}
\let\onto\twoheadrightarrow
\def\isoto{\stackrel{}{\mbox{\hspace{1mm}\raisebox{+1.4mm}{$\SC\sim$}\hspace{-3.5mm}$\longrightarrow$}}}

\usepackage[curve,matrix,arrow,cmtip]{xy}

\newdir^{ (}{{}*!/-3pt/\dir^{(}}    
\newdir^{  }{{}*!/-3pt/\dir^{}}    
\newdir_{ (}{{}*!/-3pt/\dir_{(}}    
\newdir_{  }{{}*!/-3pt/\dir_{}}    

\let\setminus\smallsetminus

\newcommand{\es}{\enspace}
\newcommand{\open}{^\circ}

\newcommand{\mal}{^{^\times}}
\newcommand{\fdot}{{}_{\bullet}}
\newcommand{\dbl}{{\mathchoice{\mbox{\rm [\hspace{-0.15em}[}}
                              {\mbox{\rm [\hspace{-0.15em}[}}
                              {\mbox{\scriptsize\rm [\hspace{-0.15em}[}}
                              {\mbox{\tiny\rm [\hspace{-0.15em}[}}}}
\newcommand{\dbr}{{\mathchoice{\mbox{\rm ]\hspace{-0.15em}]}}
                              {\mbox{\rm ]\hspace{-0.15em}]}}
                              {\mbox{\scriptsize\rm ]\hspace{-0.15em}]}}
                              {\mbox{\tiny\rm ]\hspace{-0.15em}]}}}}
\newcommand{\dpl}{{\mathchoice{\mbox{\rm (\hspace{-0.15em}(}}
                              {\mbox{\rm (\hspace{-0.15em}(}}
                              {\mbox{\scriptsize\rm (\hspace{-0.15em}(}}
                              {\mbox{\tiny\rm (\hspace{-0.15em}(}}}}
\newcommand{\dpr}{{\mathchoice{\mbox{\rm )\hspace{-0.15em})}}
                              {\mbox{\rm )\hspace{-0.15em})}}
                              {\mbox{\scriptsize\rm )\hspace{-0.15em})}}
                              {\mbox{\tiny\rm )\hspace{-0.15em})}}}}
\newcommand{\invlim}[1][]{\ifthenelse{\equal{#1}{}}
{\DS \lim_{\longleftarrow}}
{\DS \lim_{\underset{#1}{\longleftarrow}}}
}
\newcommand{\dirlim}[1][]{\ifthenelse{\equal{#1}{}}
{\DS \lim_{\longrightarrow}}
{\DS \lim_{\underset{#1}{\longrightarrow}}}
}
\newcommand{\ul}[1]{{\underline{#1}}}
\newcommand{\ol}[1]{{\overline{#1}}}
\newcommand{\wt}[1]{{\widetilde{#1}}}
\newcommand{\wh}[1]{{\widehat{#1}}}

\usepackage{amsthm}
\theoremstyle{plain}
\newtheorem{theorem}{Theorem}[section]

\newtheorem{lemma}[theorem]{Lemma}

\newtheorem{corollary}[theorem]{Corollary}
\newtheorem{proposition}[theorem]{Proposition}

\newtheorem{theo}[theorem]{Theorem}
\newtheorem{lem}[theorem]{Lemma}
\newtheorem{prop}[theorem]{Proposition}
\newtheorem{cor}[theorem]{Corollary}

\theoremstyle{definition}
\newtheorem{definition}[theorem]{Definition}

\newtheorem{remark}[theorem]{Remark}
\newtheorem{example}[theorem]{Example}

\newtheorem{construction}[theorem]{Construction}

\newtheorem{rem}[theorem]{Remark}

\newtheorem{defn}[theorem]{Definition}

\usepackage{color}
\newcounter{commentcounter}

\def\?{\ 
{\bf\color{red}???}\ 
\immediate\write16{}
\immediate\write16{Warning: There was still a question mark . . . }
\immediate\write16{}}

\def\olK{\,{\overline{\!K}}{}}

\def\wtM{{\,\wt{\!M}}}

\def\ulD{{\underline{D\!}\,}{}}

\def\wtsA{{\,\,\wt{\!\!\sA}}}

\newcommand{\BOne}{\hbox{\rm1\kern-2.5pt l\kern.9pt}}
\newcommand{\OOne}{\BO\kern-7.7pt\raisebox{1.3pt}{\underline{\phantom{{$\BO$}}}}}
\def\UOne{\underline{\hbox{\rm1\kern-2.5pt l\kern.9pt}}}

\newcommand{\BdRplus}[1]{\BB^+_{#1}}
\newcommand{\BdR}[1]{\BB_{#1}}

\usepackage{graphicx}    
\usepackage{mathrsfs}
\usepackage{epic,eepic}
\numberwithin{equation}{section}

\newcommand{\Acal}{\mathscr{A}}
\newcommand{\Bcal}{\mathscr{B}}
\newcommand{\Ccal}{\mathcal{C}}

\newcommand{\Ecal}{\mathcal{E}}
\newcommand{\Fcal}{\mathcal{F}}
\newcommand{\Gcal}{\mathcal{G}}

\newcommand{\Mcal}{\mathcal{M}}
\newcommand{\Ncal}{\mathcal{N}}
\newcommand{\Ocal}{\mathcal{O}}

\newcommand{\Ucal}{\mathcal{U}}
\newcommand{\Xcal}{\mathcal{X}}

\newcommand{\Ycal}{\mathcal{Y}}

\newcommand{\Q}{\mathbb{Q}}

\newcommand{\Z}{\mathbb{Z}}
\newcommand{\C}{\mathbb{C}}

\newcommand{\Abb}{\mathbb{A}}
\newcommand{\boldB}{\mathbb{B}}
\newcommand{\bA}{\mathbf{A}}
\newcommand{\bB}{\mathbf{B}}
\newcommand{\Fbb}{\mathbb{F}}
\newcommand{\Gbb}{\mathbb{G}}

\newcommand{\Ubb}{\mathbb{U}}

\newcommand{\Mfrak}{\mathfrak{M}}

\newcommand{\mfrak}{\mathfrak{m}}

\newcommand{\slope}{\lambda}
\newcommand{\norm}[1]{\wt{#1}}

\begin{document}
\title{The universal family of semi-stable $p$-adic Galois representations 
}
\author[U. Hartl, E. Hellmann]{Urs Hartl \es and \es Eugen Hellmann}

\maketitle

\begin{abstract}
Let $K$ be a finite field extension of $\BQ_p$ and let $\mathscr{G}_K$ be its absolute Galois group. We construct the universal family of filtered $(\phi,N)$-modules, or (more generally) the universal family of $(\phi,N)$-modules with a Hodge-Pink lattice, and study its geometric properties. Building on this, we construct the universal family of semi-stable $\mathscr{G}_K$-representations in $\BQ_p$-algebras. All these universal families are parametrized by moduli spaces which are Artin stacks in schemes or in adic spaces locally of finite type over $\BQ_p$ in the sense of Huber. This has conjectural applications to the $p$-adic local Langlands program.

\end{abstract}
\maketitle

\tableofcontents

\section{Introduction}

\bigskip\noindent
Let $K$ be a finite field extension of $\BQ_p$. The emerging $p$-adic local Langlands program wants to relate on the one hand certain continuous representations of the absolute Galois group $\mathscr{G}_K={\rm Gal}(\olK/K)$ of $K$ on $n$-dimensional $L$-vector spaces for another $p$-adic field $L$, and on the other hand topologically irreducible admissible representations of ${\rm GL}_n(K)$ on finite dimensional $L$-Banach spaces in the sense of \cite{SchneiderTeitelbaum06}. One fundamental difference to the case where $L$ is an $\ell$-adic field with $\ell\ne p$ is that the $\ell$-adic local Langlands correspondence is a bijection of merely discrete sets. In the $p$-adic case the representations vary in families. So one may even speculate about a continuous or analytic correspondence.
At present not even a conjectural formulation of the $p$-adic local Langlands correspondence purely in local terms is known. One of the main tools in the $p$-adic Langlands program is to consider families of representations that admit a dense set of points, where the representations ``come from a global set-up'', as in \cite{CEGGPS} for example. Hence a good understanding of these arithmetic families of $p$-adic Galois representations of $\mathscr{G}_K$ seems to be crucial. This understanding is our aim in the present article: we develop notions of $p$-adic families of $p$-adic Hodge structures (such as filtered $(\phi,N)$-modules) and $p$-adic Galois representations and study the relation between these two. 

The study of such families was begun by Kisin, Pappas and Rapoport in \cite{crysrep,Kisindeform,phimod} and by one of us in \cite{families}, where a universal family of filtered $\phi$-modules was constructed and, building on this, a universal family of crystalline representations with Hodge-Tate weights in $\{0,1\}$. The approach is based on Kisin's integral $p$-adic Hodge theory cf.~\cite{crysrep}. 

In the present article we generalize these results in two directions: 
First we consider more general families of $p$-adic Hodge-structure, namely families of $(\phi,N)$-modules together with a so called \emph{Hodge-Pink lattice}. The inspiration to work with Hodge-Pink lattices instead of filtrations is taken from the analogous theory over function fields; see \cite{Pink,GL11,HartlPSp}. It was already applied to Kisin's integral $p$-adic Hodge theory by Genestier and Lafforgue \cite{GL12} in the absolute case for $\phi$-modules over $\BQ_p$.

Second we generalize \cite{families} to the case of semi-stable representations. In doing so we correct some mistakes made in loc.~cit. The generalization to families with more general Hodge-Tate weights than those in \cite{families} (where the weights are assumed to be in $\{0,1\}$) gives another good reason to work with families of Hodge-Pink lattices: Kisin's theory does not describe $\mathscr{G}_K$-stable $\Z_p$-lattices in a crystalline resp.~semi-stable $\mathscr{G}_K$-representation but all $\mathscr{G}_{K_\infty}$-stable $\Z_p$-lattices, where $K_\infty$ is a certain Kummer extension of $K$ appearing in \cite{crysrep}. 
The $(\phi,N)$-modules with a Hodge-Pink lattice correspond to certain $\mathscr{G}_{K_\infty}$ representation and we describe their moduli space (or stack). This stack turns out to be a vector bundle over a space of filtered $(\phi,N)$-modules. The original space of filtered $(\phi,N)$-modules (corresponding to $\mathscr{G}_K$ rather than $\mathscr{G}_{K_\infty}$-representations) can be recovered as a section defined by a certain transversality condition in this vector bundle. 
Moreover, we consider (following \cite{phimod}) a stack of integral $p$-adic Hodge-structures and a period morphism to the moduli stack of $(\phi,N)$-modules with a Hodge-Pink lattice and describe its image. Once again, this only works using the more general framework of Hodge-Pink lattices. 
\smallskip

The introduction of Hodge-Pink lattices rather than filtrations shows new and interesting phenomena: similarly to the case of filtrations one can define weights of a Hodge-Pink lattice. However, these weights can jump within a family! Whereas for families of $\sG_K$-representations the Hodge-Tate weights should vary continuously. 
On the Galois side there is an explanation of this behavior as follows: there is no notion of Hodge-Tate weights for representations of $\sG_{K_\infty}$, but only for representations of $\sG_K$. 
\smallskip

We can define a notion of weak admissibility for $(\phi,N)$-modules with Hodge-Pink lattice and show that being weakly admissible is an open condition in the set up of adic spaces generalizing the corresponding result for filtered $\phi$-modules in \cite{families}. 
Following the method of \cite{crysrep} and \cite{families} we further cut out an open subspace over which an integral structure for the $(\phi,N)$-modules with Hodge-Pink lattice exists and an open subspace over which a family of $\mathscr{G}_{K_\infty}={\rm Gal}(\olK/K_\infty)$-representation exists. If we restrict ourselves to the subspace of filtered $(\phi,N)$-modules one can promote this family of $\mathscr{G}_{K_\infty}$-representations to the universal family of semi-stable $\mathscr{G}_K$-representations.

\smallskip

We describe our results in more detail. 
Let $K$ be a finite extension of $\Q_p$ with absolute Galois group $\mathscr{G}_K$ and maximal unramified sub-extension $K_0$. Let $\Frob_p$ be the $p$-Frobenius on $K_0$. We consider \emph{families of $(\phi,N)$-modules} over $\Q_p$-schemes $X$, that is finite locally free $\Ocal_X\otimes_{\Q_p}K_0$-modules $D$ together with a $\phi:=\id\otimes\Frob_p$-linear automorphism $\Phi$ and a linear monodromy operator $N\colon D\rightarrow D$ satisfying the usual relation $N\Phi=p\Phi N$. Choosing locally on $X$ a basis of $D$
 and considering $\Phi\in\GL_d\bigl(\Ocal_X\otimes_{\Q_p}K_0\bigr)$ and $N\in{\rm Mat}_{d\times d}(\Ocal_X\otimes_{\Q_p}K_0)$ as matrices, the condition $N\Phi=p\Phi N$ cuts out a closed subscheme $P_{K_0,d}\subset\Res_{K_0/\BQ_p}\GL_{d}\,\times_{\BQ_p}\,\Res_{K_0/\Q_p}{\rm Mat}_{d\times d}$. We can describe the geometry of this scheme as follows.

\medskip\noindent
{\bfseries Theorem~\ref{ThmReduced}.}
{\it
The scheme $P_{K_0,d}$ is equi-dimensional of dimension $[K_0:\Q_p]\,d^2$. It is reduced, Cohen-Macaulay and generically smooth over $\BQ_p$. Its irreducible components are indexed by the possible Jordan types of the (necessarily nilpotent) monodromy operator $N$. 
}

\medskip

Further we consider families of $(\phi,N)$-modules $D$ with a filtration $\Fcal^\bullet$ on $D\otimes_{K_0}K$ and more generally families of $(\phi,N)$-modules with a Hodge-Pink lattice $\Fq$, see Definition~\ref{DefHPLattice} for the precise definitions. Given a cocharacter $\mu$ of the algebraic group $\Res_{K/\Q_p}\GL_{d,K}$ (or more precisely a cocharacter of the Weil restriction of the diagonal torus which is dominant with respect to the Weil restriction of the upper triangular matrices) we define a notion of a filtration $\Fcal^\bullet$ resp.\ a Hodge-Pink lattice with \emph{constant Hodge polygon equal to $\mu$}, and a notion of \emph{boundedness by $\mu$} for a Hodge-Pink lattice $\Fq$. With $\mu$ is associated a reflex field $E_\mu$ which is a finite extension of $\BQ_p$.

\medskip\noindent
{\bfseries Theorem~\ref{ThmModuliStacks}.}
{\it
\begin{enumerate}
\item
The stack $\sH_{\phi,N,\preceq\mu}$ parametrizing rank $d$ families of $(\phi,N)$-modules with Hodge-Pink lattice bounded by $\mu$ on the category of $E_\mu$-schemes is an Artin stack. It is equi-dimensional and generically smooth. Its dimension can be explicitly described in terms of the cocharacter $\mu$ and its irreducible components are indexed by the possible Jordan types of the (nilpotent) monodromy operator.
\item The stack $\sH_{\phi,N,\mu}$  parametrizing rank $d$ families of $(\phi,N)$-modules with Hodge-Pink lattice with constant Hodge polygon equal to $\mu$, is an open and dense substack of $\sH_{\phi,N,\preceq\mu}$. Further it is reduced and Cohen-Macaulay. It admits a canonical map to the stack $\sD_{\phi,N,\mu}$ of filtered $(\phi,N)$-modules with filtration of type $\mu$. This map is representable by a vector bundle.
\end{enumerate}
}

\medskip

If we restrict ourselves to the case of vanishing monodromy, i.e.~the case $N=0$, we cut out a single irreducible component $\sH_{\phi,\preceq\mu}\subset \sH_{\phi,N,\preceq\mu}$ and similarly for the other stacks in the theorem. 
Following \cite{families} we consider the above stacks also as stacks on the category of adic spaces locally of finite type over $\Q_p$, i.e.~we consider the adification $\sH_{\phi,N,\preceq\mu}^{\rm ad}$, etc. Passing from $\Q_p$-schemes to adic spaces allows us to generalize Kisin's comparison between filtered $(\phi,N)$-modules and vector bundles on the open unit disc (together with certain additional structures). To do so we need to fix a uniformizer $\pi$ of $K$ as well as its minimal polynomial $E(u)$ over $K_0$.

\medskip\noindent
{\bfseries Theorem~\ref{ThmEquivDandM}.}
{\it
For every adic space $X$ locally of finite type over $\Q_p$ there is a natural equivalence of categories between the category of $(\phi,N)$-modules with Hodge-Pink lattice over $X$ and the category of $(\phi,N_\nabla)$-modules over $X$, i.e.~the category of vector bundles $\Mcal$ on the product of $X$ with the open unit disc $\Ubb$ over $K_0$ together with a semi-linear map $\Phi_\Mcal\colon \Mcal\rightarrow \Mcal$ that is an isomorphism away from $X\times\{E(u)=0\}\subset X\times\Ubb$ and a differential operator $N_\nabla^\Mcal$ satisfying \[N_\nabla^\CM\circ\Phi_\CM\circ\phi=p\,\tfrac{E(u)}{E(0)}\cdot\Phi_\CM\circ\phi\circ N_\nabla^\CM.\]
}

\smallskip\noindent
{\bfseries Theorem~\ref{ThmHnabla}.}
{\it
The differential operator $N_\nabla^\Mcal$ defines a canonical meromorphic connection on the vector bundle $\Mcal$. The closed substack $\sH_{\phi,N,\mu}^\nabla\subset\sH_{\phi,N,\mu}^\ad$ where this connection is holomorphic coincides with the zero section of the vector bundle $\sH_{\phi,N,\mu}^\ad\rightarrow \sD_{\phi,N,\mu}^\ad$.
}

It should be mentioned that the results of Kisin \cite{crysrep} have a parallel story, earlier developed by Berger \cite{Berger}, using the cyclotomic extension $K(\epsilon_n, n\geq 1)$ for a compatible system $(\epsilon_n)$ of $p^n$-th root of unity, instead of the Kummer extension $K_\infty$.  The above results are very much inspired by loc.~cit.  
\medskip

Similarly to the case of filtered $\phi$-modules in \cite{families} there is a notion of weak admissibility for families of $(\phi,N)$-modules  with Hodge-Pink lattice over an adic space. We show that weak admissibility is an open condition.

\medskip\noindent
{\bfseries Theorem~\ref{ThmWAOpen}.}
{\it
Let $\mu$ be a cocharacter as above, then the groupoid 
\[X\mapsto \{(D,\Phi,N,\q)\in \sH_{\phi,N, \preceq \mu}(X)\mid D\otimes \kappa(x)\ \text{is weakly admissible for all}\ x\in X\}\]
is an open substack $\sH^{\rm ad, wa}_{\phi,N, \preceq \mu}$ of $\sH^\ad_{\phi,N, \preceq \mu}$. 
}

\medskip

Following the construction in \cite{families} we construct an open substack $\sH_{\phi,N,\preceq\mu}^{\rm ad, int}\subset\sH_{\phi,N,\preceq\mu}^{\rm ad, wa}$ where an integral model for the $(\phi,N_{\nabla})$-module over the open unit disc exists. Here integral means with respect to the ring of integers $W$ in $K_0$. Dealing with Hodge-Pink lattices instead of filtrations makes is possible to generalize the period morphism of \cite[\S\,5]{phimod} beyond the miniscule case. That is, we consider a stack $\hat\Ccal_{\preceq\mu,N,K}$ in the category of formal schemes over $\Spf\CO_{E_\mu}$ whose $R$-valued points parameterize tuples $(\Mfrak,\Phi,N)$ where $\Mfrak$ is a finite locally free $(R\otimes_{\Z_p}W)\dbl u\dbr$ module, $\Phi$ is a semi-linear morphism $\Phi\colon \Mfrak\rightarrow \Mfrak$ which is an isomorphism away from $E(u)=0$, whose behavior at $E(u)$ is controlled in terms of $\mu$, and $N$ is an endomorphism of $\Mfrak/u\Mfrak$ satisfying $N\Phi=p\Phi N$; see after Remark~\ref{periodandN} for the precise definition.

Given a $p$-adic formal scheme $\Xcal$ over $\Spf\CO_{E_\mu}$ we construct a period morphism
\[\Pi(\Xcal)\colon \hat\Ccal_{\preceq \mu,N,K}(\Xcal)\longrightarrow \sH^{\rm ad}_{\phi,N,\preceq \mu}(\Xcal^{\rm ad})\]
and the substack $\sH^{\rm ad, int}_{\phi,N,\preceq \mu}$ will serve as the image of this morphism in the following sense. 

\medskip\noindent
{\bfseries Corollary~\ref{CorImagePeriodMap}.}
{\it
Let $X$ be an adic space locally of finite type over the reflex field $E_\mu$ of $\mu$ and let $f\colon X\rightarrow \sH_{\phi,N,\preceq\mu}^{\ad}$ be a morphism defined by $(D,\Phi,N,\q)$. Then $f$ factors over $\sH^{\rm ad, int}_{\phi,N,\preceq\mu}$ if and only if there exists an fpqc-covering $(U_i\rightarrow X)_{i\in I}$  and formal models $\Ucal_i$ of $U_i$ together with $(\Mfrak_i,\Phi_i)\in \hat\Ccal_{\preceq\mu,N,K}(\Ucal_i)$ such that $\Pi(\Ucal_i)(\Mfrak_i,\Phi_i)=(D,\Phi,N,\q)|_{U_i}$. 
}

\medskip

Finally we go back to Galois representations. We prove that there is a canonical open subspace $\sH^{\rm red, ad, adm}_{\phi,N,\preceq\mu}$ of the reduced space underlying $\sH^{\rm ad, int}_{\phi,N,\preceq\mu}$ which carries a family of $\mathscr{G}_{K_\infty}$-representations. This family is universal in a sense made precise in the body of the article.
Roughly this means that a morphism $f\colon X\rightarrow \sH^{\rm ad, int}_{\phi,N,\preceq\mu}$ defined by some $(\Mfrak,\Phi,N)$ over a formal model $\Xcal$ of $X$ factors over $\sH^{\rm ad, adm}_{\phi,N,\preceq\mu}$ if and only there exists a family of $\mathscr{G}_{K_\infty}$-representations $\Ecal$ on $X$ such that the $\phi$-module of $\Ecal$, in the sense of Fontaine, is (up to inverting $p$) given by the $p$-adic completion of $(\Mfrak,\Phi)[1/u]$.
For a finite extension $L$ of $E_\mu$, Kisin's theory implies that we have an equality 
\[\sH^{\rm red, ad,adm}_{\phi,N,\preceq\mu}(L)=\sH^{\rm ad, int}_{\phi,N,\preceq\mu}(L)=\sH^{\rm ad,wa}_{\phi,N,\preceq\mu}(L)\]
of $L$-valued points. 

If we want to promote our family of $\sG_{K_\infty}$-representations to a family of $\sG_K$-representations, we have to restrict ourselves to filtrations rather than Hodge-Pink lattices. The reason for that is that the meromorphic connection $\nabla$ in Theorem~\ref{ThmHnabla} above must be holomorphic in this case. 
In the framework of Berger's work \cite{Berger} with the cyclotomic tower, this is in some sense even more apparent: the connection $\nabla$ comes from the derivation of the $\Gamma$-action. 

\medskip\noindent
{\bfseries Theorem~\ref{ThmUnivFamily}.}
{\it
There is an open substack $\sD_{\phi,N,\mu}^{\rm ad,adm}\subset \sD_{\phi,N,\mu}^{\rm ad}$ over which there exists a family $\Ecal$ of semi-stable $\mathscr{G}_K$-representations such that $D_{\rm st}(\Ecal)=(D,\Phi,N,\Fcal^\bullet)$ is the restriction of the universal family of filtered $(\phi,N)$-modules on $\sD_{\phi,N,\mu}^{\rm ad}$ to $\sD_{\phi,N,\mu}^{\rm ad, adm}$.\\
This family is universal in the following sense: Let $X$ be an adic space locally of finite type over the reflex field $E_\mu$ of $\mu$, and let $\Ecal'$ be a family of semi-stable $\mathscr{G}_K$-representations on $X$ with constant Hodge polygon equal to $\mu$. Then there is a unique morphism $f\colon X\rightarrow \sD_{\phi,N,\mu}^{\rm ad,adm}$ such that $\Ecal'\cong f^\ast\Ecal$ as families of $\mathscr{G}_K$-representations.
}

\medskip

The corresponding result for crystalline $\sG_K$-representations with constant Hodge polygon equal to $\mu$, whose moduli space is $\sD_{\phi,\mu}^{\rm ad,adm}$, is formulated and proved in Corollary~\ref{CorUnivCrystFamily}.
We finally briefly discuss how these results relate to Kisin's construction of potentially semi-stable deformation rings \cite{Kisindeform}. There is a precise relation between our universal family and Kisin's construction discussed in Proposition \ref{comparewKisin} in the body of the paper. It should be mentioned however, that the spirit of our approach differs from Kisin's: first of all our point of view is the study of families of $p$-adic Hodge-structures (i.e.~the side of semi-linear algebra data) in the first place - and then we turn to cut out a subspace defining a Galois representation. Kisin starts with families of Galois representations (provided by deformation rings) and then cuts out a crystalline locus. Moreover, his definition of a crystalline family differs from ours: Kisin defines a family to be semi-stable if its base change to all finite dimensional $\Q_p$-algebras is semi-stable. In contrast we aim at giving a definition that is more in the spirit of Fontaine's definition using period rings. In fact, as we needed to correct some mistakes from the last section of \cite{families}, we also changed the definition of crystalline representations from loc.cit: it seems to be a bit messy to deal with the filtration on a sheafified version of $B_{\rm cris}$, hence we rather use the $\phi$-modules on the open unit disc as our $p$-adic Hodge structures and define the notion of a semi-stable representation using the comparison of a vector bundle on the open unit disc and a Galois representation after tensoring with (a relative version of) $B^+_{\rm cris}$, see Definition \ref{DefSemiStRepNeu}.

\bigskip\noindent
{\bf Notations:} Let $K$ be a finite field extension of the $p$-adic numbers $\BQ_p$ and fix an algebraic closure $\olK$ of $K$. We write $\C_p$ for the $p$-adic completion of $\olK$ and let $\mathscr{G}_K=\Gal(\olK/K)$ be the absolute Galois group of $K$. Let $\norm{K}$ be the Galois closure of $K$ inside $\olK$. Let $K_0$ be the maximal unramified subfield of $K$ and $W$ its ring of integers. Set $f:=[K_0:\BQ_p]$, and let $\Frob_p$ be the Frobenius automorphism of $K_0$ which induces the $p$-power map on the residue field of $K_0$. We fix once and for all a uniformizer $\pi$ of $K$ and its minimal polynomial $E(u)={\rm mipo}_{\pi/K_0}(u)\in W[u]$ over $K_0$. It is an Eisenstein polynomial, and $K=K_0[u]/(E(u))$. We choose a compatible system $\pi_n$ of $p^n$-th roots of $\pi$ in $\olK$ and write $K_\infty$ for the field obtained from $K$ by adjoining all $\pi_n$. 

\bigskip\noindent
{\bf Acknowledgements.}
The second author acknowledges support of the DFG (German Research Foundation) in form of SFB/TR 45 ``Periods, Moduli Spaces and Arithmetic of Algebraic Varieties'' and in form of a Forschungsstipendium He 6753/1-1. Both authors were also supported by SFB 878 ``Groups, Geometry \& Actions'' of the DFG and Germany's Excellence Strategy EXC 2044--390685587 ``Mathematics Münster: Dynamics-Geometry-Structure''. We would further like to thank A. M\'ezard, M. Rapoport, T. Richarz and P. Scholze for helpful discussions.

\section{Families of $(\phi,N)$-modules with Hodge-Pink lattice}\label{SectFamilies}

Let $R$ be a $\BQ_p$-algebra and consider the endomorphism $\phi:=\id_R\otimes\Frob_p$ of $R\otimes_{\BQ_p}K_0$. For an $R\otimes_{\BQ_p}K_0$-module $M$ we set $\phi^\ast M:=M\otimes_{R\otimes_{\BQ_p}K_0,\,\phi}R\otimes_{\BQ_p}K_0$. Similar notation is applied to morphisms between $R\otimes_{\BQ_p}K_0$-modules. We let $\phi^\ast\colon M\to\phi^\ast M$ be the $\phi$-semi-linear map with $\phi^\ast(m)= m\otimes1$.

We introduce the rings \[\BdRplus{R}:=\invlim[i](R\otimes_{\BQ_p}K_0[u])/(E(u)^i)\ \text{and}\ \BdR{R}:=\BdRplus{R}[\frac{1}{E(u)}].\] In a certain sense $\BdRplus{\BQ_p}$ and $\BdR{\BQ_p}$ are the analogs of Fontaine's rings $\mathbf{B}_\dR^+$ and $\mathbf{B}_\dR$ in Kisin's theory \cite{crysrep} of $p$-adic Galois representation. By Cohen's structure theorem \cite[Theorem II.4.2]{SerreLF} the ring $\BdRplus{\BQ_p}=\invlim K_0[u]/(E(u)^i)$ is isomorphic to $K\dbl t\dbr$ under a map sending $t$ to $\tfrac{E(u)}{E(0)}$ (and by Hensel's Lemma the lift of the residue field $K$ to a subring of $\BdRplus{\BQ_p}$ is unique). The rings $\BdRplus{R}$ and $\BdR{R}$ are relative versions over $R$, and are isomorphic to $(R\otimes_{\BQ_p}K)\dbl t\dbr$, respectively $(R\otimes_{\BQ_p}K)\dbl t\dbr[\frac{1}{t}]$. We extend $\phi$ to $R\otimes_{\BQ_p}K_0[u]$ by requiring $\phi(u)=u^p$ and we define $\phi^n(\BdRplus{R}):=\lim\limits_{\longleftarrow\ i}(R\otimes_{\BQ_p}K_0[u])/(\phi^n(E(u))^i)$. Note that we may also identify $\phi^n(\BdRplus{R})$ with $\invlim(R\otimes_{\BQ_p}K(\pi_n)[u])/(1-\tfrac{u}{\pi_n})^i\,=\,(R\otimes_{\BQ_p}K(\pi_n))\dbl 1-\tfrac{u}{\pi_n}\dbr$ under the assignment $\tfrac{E(u)}{E(0)}\mapsto 1-\tfrac{u}{\pi_n}$; compare \cite[(1.1.1)]{crysrep}. 
We extend these rings to sheaves of rings $\phi^n(\BdRplus{X}):=\phi^n(\BdRplus{\CO_X})$ on $\BQ_p$-schemes $X$ or adic spaces $X\in\Ad_{\BQ_p}^\lft$. Here $\Ad_{\BQ_p}^{\lft}$ denotes the category of adic spaces locally of finite type, see \cite{Hu2} for example.
\begin{rem}
Note that $\phi^n(\BdRplus{R})$ is not a subring of $\BdRplus{R}$. If $X=\Spa(R,R^\circ)$ is an affinoid adic space of finite type over $\Q_p$ one should think of $\phi^n(\BdRplus{R})$ as the completion of the structure sheaf on $X\times \Ubb$ along the section defined by $\phi^n(E(u))\in \Ubb$. Here $\Ubb$ denotes the open unit disc over $K_0$.
\end{rem}

\begin{definition}\label{DefPhiModule}
\begin{enumerate}
\item 
A \emph{$\phi$-module $(D,\Phi)$ over $R$} consists of a locally free $R\otimes_{\BQ_p}K_0$-module $D$ of finite rank, and an $R\otimes_{\BQ_p}K_0$-linear isomorphism $\Phi\colon \phi^\ast D\isoto D$. A \emph{morphism $\alpha\colon (D,\Phi)\to(\wt D,\wt\Phi)$ of $\phi$-modules} is an $R\otimes_{\BQ_p}K_0$-homomorphism $\alpha\colon D\to\wt D$ with $\alpha\circ\Phi=\wt\Phi\circ\phi^*\alpha$.
\item 
A \emph{$(\phi,N)$-module $(D,\Phi,N)$ over $R$} consists of a $\phi$-module $(D,\Phi)$ over $R$ and an $R\otimes_{\BQ_p}K_0$-linear endomorphism $N\colon D\to D$ satisfying $N\circ\Phi = p\cdot \Phi\circ\phi^\ast N$. A \emph{morphism $\alpha\colon (D,\Phi,N)\to(\wt D,\wt\Phi,\wt N)$ of $(\phi,N)$-modules} is morphism of $\phi$-modules with $\alpha\circ N=\wt N\circ\alpha$.
\end{enumerate}
The rank of $D$ over $R\otimes_{\BQ_p}K_0$ is called the \emph{rank} of $(D,\Phi)$, resp.\ $(D,\Phi,N)$.
\end{definition}

Every $\phi$-module over $R$ can be viewed as a $(\phi,N)$-module with $N=0$.

\begin{lemma}\label{LemmaNnilpot}
\begin{enumerate}
\item \label{LemmaNnilpotA}
Every $\phi$-module $(D,\Phi)$ over $R$ is Zariski locally on $\Spec R$ free over $R\otimes_{\BQ_p}K_0$.
\item \label{LemmaNnilpotB}
The endomorphism $N$ of a $(\phi,N)$-module over $R$ is automatically nilpotent.
\end{enumerate}
\end{lemma}

\begin{proof}
\ref{LemmaNnilpotA} Let $\Fm\subset R$ be a maximal ideal. Then $R/\Fm\otimes_{\BQ_p}K_0$ is a direct product of fields which are transitively permuted by $\Gal(K_0/\BQ_p)$. The existence of the isomorphism $\Phi$ implies that $D\otimes_RR/\Fm$ is free over $R/\Fm\otimes_{\BQ_p}K_0$. Now the assertion follows by Nakayama's lemma.
\smallskip

\noindent
\ref{LemmaNnilpotB}
By \ref{LemmaNnilpotA} we may locally on $R$ write $N$ as a matrix with entries in $R\otimes_{\BQ_p}K_0$. Set $d:=\rk D$. If the entries of the $d$-th power $N^d$ lie in $\Rad(0)\otimes_{\BQ_p}K_0$, where $\Rad(0)=\bigcap_{\Fp\subset R\text{ prime}}\Fp$ is the nil-radical, then $N$ is nilpotent. Thus we may check the assertion in $L=\Frac(R/\Fp)^\alg$ for all primes $\Fp\subset R$. We replace $R$ by $L$. Then $D=\prod V_{\psi}$ splits up into a direct product of $d$-dimensional $L$-vector spaces indexed by the embeddings $\psi\colon K_0\hookrightarrow L$. For every fixed embedding $\psi$ the $f$-th power $\Phi^f$ restricts to an endomorphism $\Phi_\psi$ of $V_\psi$ satisfying $N\Phi_\psi=p^f\Phi_\psi N$. If $V(\lambda,\Phi_\psi)$ denotes the generalized eigenspace for some $\lambda\in L^\times$, then $N$ maps $V(\lambda,\Phi_\psi)$ to $V(p^f\lambda,\Phi_\psi)$ and hence $N$ is nilpotent, as there are only finitely many non-zero eigenspaces. This implies that $N^d=0$.
\end{proof}

\begin{remark}\label{RemDecompOfD}
If $R$ is 
even a $K_0$-algebra, we can decompose $R\otimes_{\BQ_p}K_0\cong\prod_{i\in\BZ/f\BZ}R$ where the $i$-th factor is given by the map $R\otimes_{\BQ_p}K_0\to R,\,a\otimes b\mapsto a\Frob_p^{-i}(b)$ for $a\in R, b\in K_0$. For a $(\phi,N)$-module over $R$ we obtain corresponding decompositions $D=\prod_i D_i$ and $\phi^*D=\prod_i(\phi^*D)_i$ with $(\phi^*D)_i=D_{i-1}$, and hence also $\Phi=(\Phi_i\colon D_{i-1}\isoto D_i)_i$ and $N=(N_i\colon D_i\to D_i)_i$ with $p\,\Phi_i\circ N_{i-1}=N_i\circ\Phi_i$, because $(\phi^*N)_i=N_{i-1}$. If we set $\Psi_i:=\Phi_i\circ\ldots\circ\Phi_1=(\Phi\circ\phi^*\Phi\circ\ldots\circ\phi^{(i-1)*}\Phi)_i\colon D_0=(\phi^{i*}D)_i\isoto D_i$ then $p^i\,\Psi_i\circ N_0=N_i\circ\Psi_i$ for all $i$, and $\Psi_f=(\Phi^f)_0$. There is an isomorphism of $(\phi,N)$-modules over $R$
\begin{equation}\label{EqDecomp}
(\id_{D_0},\Psi_1,\ldots,\Psi_{f-1})\colon \Bigl(\prod_i D_0,\,\bigl((\Phi^f)_0,\id_{D_0},\ldots,\id_{D_0}\bigr),\,
(p^i N_0)_i\Bigr)\es\isoto\es\Bigl(\prod_iD_i,\,(\Phi_i)_i,\,(N_i)_i\Bigr)\,.
\end{equation}
Thus $(D,\Phi,N)$ is uniquely determined by $(D_0,(\Phi^f)_0,N_0)$ satisfying $p^f(\Phi^f)_0\circ N_0=N_0\circ(\Phi^f)_0$. Further note that under this isomorphism $\Phi^f$ on $(D,\Phi,N)$ corresponds to $\bigl((\Phi^f)_0,\ldots,(\Phi^f)_0\bigr)$ on the left hand side.
\end{remark}

\begin{definition}\label{DefHPLattice}
\begin{enumerate}
\item 
A \emph{$K$-filtered $(\phi,N)$-module $(D,\Phi,N,\CF^\bullet)$ over $R$} consists of a $(\phi,N)$-module $(D,\Phi,N)$ over $R$ together with a decreasing separated and exhaustive $\BZ$-filtration $\CF^\bullet$ on $D_K:=D\otimes_{K_0}K$ by $R\otimes_{\BQ_p}K$-submodules such that $\gr_\CF^i D_K:=\CF^iD_K/\CF^{i+1}D_K$ is locally free as an $R$-module for all $i$. A \emph{morphism $\alpha\colon (D,\Phi,N,\CF^\bullet)\to(\wt D,\wt\Phi,\wt N,\wt\CF^\bullet)$} is a morphism of $(\phi,N)$-modules with $(\alpha\otimes\id)(\CF^iD_K)\subset \wt\CF^i\wt D_K$.
\item 
A \emph{$(\phi,N)$-module with Hodge-Pink lattice $(D,\Phi,N,\Fq)$ over $R$} consists of a $(\phi,N)$-module $(D,\Phi,N)$ over $R$ together with a \emph{$\BdRplus{R}$-lattice} $\Fq\subset D\otimes_{R\otimes K_0}\BdR{R}$. This means that $\Fq$ is a finitely generated $\BdRplus{R}$-submodule, which is a direct summand as $R$-module satisfying $\BdR{R}\cdot\Fq= D\otimes_{R\otimes K_0}\BdR{R}$. We call $\Fq$ the \emph{Hodge-Pink lattice} of $(D,\Phi,N,\Fq)$. A \emph{morphism} $\alpha\colon (D,\Phi,N,\Fq)\to(\wt D,\wt\Phi,\wt N,\tilde\Fq)$ is a morphism of $(\phi,N)$-modules with $(\alpha\otimes\id)(\Fq)\subset\tilde\Fq$.
\end{enumerate}
\end{definition}
\begin{rem}
Note that the graded pieces $\gr_\CF^i D_K$ in (i) are $R\otimes_{\Q_p}K$-modules that are locally on $\Spec (R\otimes_{\Q_p} K)$ free, but not necessarily of the same rank. Hence they are not locally on $\Spec R$ free as $R\otimes_{\Q_p}K$-modules. However, they are locally on $\Spec R$ free as $R$-modules. 
\end{rem}

For every $(\phi,N)$-module with Hodge-Pink lattice $(D,\Phi,N,\Fq)$ over $R$ we also consider the tautological $\BdRplus{R}$-lattice $\Fp:=D\otimes_{R\otimes K_0}\BdRplus{R}$.

\begin{lemma}\label{LemmaHPLattice}
Let $\Fq\subset D\otimes_{R\otimes K_0}\BdR{R}$ be a $\BdRplus{R}$-submodule. Then $\Fq$ is a $\BdRplus{R}$-lattice if and only if $E(u)^n\Fp\subset\Fq\subset E(u)^{-m}\Fp$ for all $n,m\gg0$ and for any (some) such $n,m$ the quotients $E(u)^{-m}\Fp/\Fq$ and $\Fq/E(u)^n\Fp$ are finite locally free $R$-modules. 

If this is the case then \'etale locally on $\Spec R$ the $\BdRplus{R}$-module $\Fq$ is free of the same rank as $\Fp$.
\end{lemma}

\begin{proof}
The assertion $E(u)^n\Fp\subset\Fq\subset E(u)^{-m}\Fp$ for all $n,m\gg0$ is equivalent to $\BdR{R}\cdot\Fq= D\otimes_{R\otimes K_0}\BdR{R}$ when $\Fq$ is finitely generated. Consider such $n,m$. If $\Fq$ is a $\BdRplus{R}$-lattice, hence a direct summand of $D\otimes_{R\otimes K_0}\BdR{R}$ there is an $R$-linear section $s$ of the projection ${\rm pr}:D\otimes_{R\otimes K_0}\BdR{R} \onto (D\otimes_{R\otimes K_0}\BdR{R})/\Fq$. The composition of this section with the inclusion $E(u)^{-m}\Fp/\Fq\into (D\otimes_{R\otimes K_0}\BdR{R})/\Fq$ factors through $E(u)^{-m}\Fp$: Indeed,  for $x\in E(u)^{-m}\Fp/\Fq$ the condition ${\rm pr}(s(x))=x$ means that that there exists $x'\in \Fq$ such that $s(x)=x+x'\in E(u)^{-m}\Fp+\Fq=E(u)^{-m}\Fp$.
  
Hence we see that the inclusion $E(u)^{-m}\Fp/\Fq\into (D\otimes_{R\otimes K_0}\BdR{R})/\Fq$ realizes $E(u)^{-m}\Fp/\Fq$ as a direct summand of the $R$-module $E(u)^{-m}\Fp/E(u)^n\Fp$ which is locally free by Lemma~\ref{LemmaNnilpot}\ref{LemmaNnilpotA}. This shows that $E(u)^{-m}\Fp/E(u)^n\Fp\cong (E(u)^{-m}\Fp/\Fq)\oplus(\Fq/E(u)^n\Fp)$ and both $E(u)^{-m}\Fp/\Fq$ and $\Fq/E(u)^n\Fp$ are finite locally free $R$-modules.

Conversely any isomorphism $E(u)^{-m}\Fp/E(u)^n\Fp\cong (E(u)^{-m}\Fp/\Fq)\oplus(\Fq/E(u)^n\Fp)$ together with the decomposition $D\otimes_{R\otimes K_0}\BdR{R}\;\cong\;(E(u)^n\Fp) \,\oplus\,(E(u)^{-m}\Fp/E(u)^n\Fp)\,\oplus\,(D\otimes_{R\otimes K_0}\BdR{R})/E(u)^{-m}\Fp$ realizes $\Fq$ as a direct summand of $D\otimes_{R\otimes K_0}\BdR{R}$: Indeed, we have the following direct sum decompositions of $R$-modules:
\begin{align*}
\Fq&\cong (E(u)^n\Fp)\oplus (\Fq/E(u)^n\Fp)\\
D\otimes_{R\otimes K_0}\BdR{R}&\cong (E(u)^n\Fp)\oplus (\Fq/E(u)^n\Fp)\oplus (E(u)^{-m}\Fp/\Fq)\oplus (D\otimes_{R\otimes K_0}\BdR{R})/E(u)^{-m}\Fp.
\end{align*} 
Since $E(u)^n\Fp$ is finitely generated over $\BdRplus{R}$ and $\Fq/E(u)^n\Fp$ is finitely generated over $R$, also $\Fq$ is finitely generated over $\BdRplus{R}$, hence a $\BdRplus{R}$-lattice.

To prove the local freeness of $\Fq$ we may work locally on $R$ and assume by Lemma~\ref{LemmaNnilpot}\ref{LemmaNnilpotA} that $\Fp$ is free over $\BdRplus{R}$, say of rank $d$, and $\Fq/E(u)^n\Fp$ and $E(u)^{-m}\Fp/\Fq$ are free over $R$. There is a noetherian subring $\wt R$ of $R$ and a short exact sequence 
\begin{equation}\label{EqTildeSeq1}
\xymatrix { 0 \ar[r] & \wt Q \ar[r] & \wt P \ar[r] & \wt N \ar[r] & 0 }
\end{equation}
of $\BdRplus{\wt R}$-modules which are free $\wt R$-modules, such that the tensor product of \eqref{EqTildeSeq1} with $R$ over $\wt R$ is isomorphic to 
\begin{equation}\label{EqSeq1Fq}
\xymatrix { 0 \ar[r] & \Fq/E(u)^n\Fp \ar[r] & E(u)^{-m}\Fp/E(u)^n\Fp \ar[r] & E(u)^{-m}\Fp/\Fq \ar[r] & 0 \,.}
\end{equation}
Indeed, we can take $\wt R$ as the finitely generated $\BQ_p$-algebra containing all the coefficients appearing in matrix representations of the maps in \eqref{EqSeq1Fq} and the action of $K_0[u]/(E(u))^{m+n}$.
Note that, since $\wt P$ is free over $\wt R$, it is contained in $\wt P\otimes_{\wt R}R\cong E(u)^{-m}\Fp/E(u)^n\Fp$, and since the latter is annihilated by $E(u)^{m+n}$, the same is true for $\wt P,\wt Q$ and $\wt N$.  
Let $\wt\Fp$ be a free $\BdRplus{\wt R}$-module of rank $d$ and fix an isomorphism $\wt\Fp\otimes _{\BdRplus{\wt R}}\BdRplus{R}\cong \Fp$. This isomorphism obviously induces an isomorphism 
\[
E(u)^{-m}\wt\Fp\otimes_{\BdRplus{\wt R}}\bigl(\BdRplus{\wt R}/(E(u))^{m+n}\bigr)\isoto\wt P.
\] 
Let the $\BdRplus{\wt R}$-module $\wt\Fq$ be defined by the exact sequence
\begin{equation}\label{EqTildeSeq2}
\xymatrix { 0 \ar[r] & \wt\Fq \ar[r] & E(u)^{-m}\wt\Fp \ar[r] & \wt N \ar[r] & 0 \,.}
\end{equation}
Since $\BdRplus{\wt R}\cong(\wt R\otimes_{\BQ_p}K)\dbl t\dbr$ is noetherian, $\wt\Fq$ is finitely generated. Consider a maximal ideal $\Fm\subset\BdRplus{\wt R}$. Since $t\in\Fm$, it maps to a maximal ideal $\Fn$ of $\wt R$. Since $\Fn$ is finitely generated, $\BdRplus{\wt R}\otimes_{\wt R}\wt R/\Fn\,\cong\,(\wt R/\Fn\otimes_{\BQ_p}K)\dbl t\dbr$ and this is a direct product of discrete valuation rings. Thus $\wt\Fq\otimes_{\wt R}\wt R/\Fn$ is locally free of rank $d$ by the elementary divisor theorem. Since this holds for all $\Fm$, \cite[IV$_3$, Theorem 11.3.10]{EGA} implies that $\wt\Fq$ is a projective $\BdRplus{\wt R}$-module and by \cite[I$_{\rm new}$, Proposition 10.10.8.6]{EGA} it is locally on $\Spec\wt R\otimes_{\BQ_p}K$ free over $\BdRplus{\wt R}$. Let $\{\psi\colon K\into\ol\BQ_p\}$ be the set of all $\BQ_p$-homomorphisms and let $\norm{K}$ be the compositum of all $\psi(K)$ inside $\ol\BQ_p$. Then $\wt R\to\wt R\otimes_{\BQ_p}\norm{K}$ is finite \'etale and the pullback of $\wt\Fq$ under this base change is locally on $\Spec\wt R\otimes_{\BQ_p}\norm{K}\otimes_{\BQ_p}K$ free over $\BdRplus{\wt R\otimes_{\BQ_p}\norm{K}}$. Since $\Spec\wt R\otimes_{\BQ_p}\norm{K}\otimes_{\BQ_p}K=\coprod_\psi\Spec\wt R\otimes_{\BQ_p}\norm{K}$ it follows that the pullback of $\wt\Fq$ is already locally on $\Spec\wt R\otimes_{\BQ_p}\norm{K}$ free over $\BdRplus{\wt R\otimes_{\BQ_p}\norm{K}}$.

To finish the proof it remains to show that $\wt\Fq\otimes_{\BdRplus{\wt R}}\BdRplus{R}\cong\Fq$. Tensoring \eqref{EqTildeSeq2} with $\BdRplus{R}$ over $\BdRplus{\wt R}$ we obtain the top row in the diagram
\[
\xymatrix { 0 \ar[r] & \Tor_1^{\BdRplus{\wt R}}(\wt N,\BdRplus{R}) \ar[r] & \wt\Fq\otimes_{\BdRplus{\wt R}}\BdRplus{R} \ar[r]\ar@{->>}[d] & E(u)^{-m}\Fp \ar[r]\ar@{=}[d] & \wt N\otimes_{\BdRplus{\wt R}}\BdRplus{R} \ar[r]\ar[d]^\cong & 0 \\
 & 0 \ar[r] & \Fq \ar[r] & E(u)^{-m} \Fp \ar[r] & E(u)^{-m}\Fp/\Fq \ar[r] & 0 \,.
}
\]
Abbreviate $\ell:=m+n$. Since the functor $\wt N\otimes_{\BdRplus{\wt R}}\fdot$ equals the composition of the functors $(\BdRplus{\wt R}/t^{\ell})\otimes_{\BdRplus{\wt R}}\fdot$ followed by $\wt N\otimes_{\BdRplus{\wt R}/t^{\ell}}\fdot$, the $\Tor_1$-module on the left can be computed from a change of rings spectral sequence \cite[Theorem 10.71]{Rotman} and its associated 5-term sequence of low degrees, see \cite[Theorem 10.31]{Rotman},
\[
\ldots\;\longto\; \Tor_1^{\BdRplus{\wt R}}(\BdRplus{\wt R}/t^{\ell},\BdRplus{R})\otimes_{\BdRplus{\wt R}/t^{\ell}}\wt N \;\longto\;\Tor_1^{\BdRplus{\wt R}}(\wt N,\BdRplus{R}) \;\longto\;\Tor_1^{\BdRplus{\wt R}/t^{\ell}}(\wt N,\BdRplus{R}/t^{\ell}) \;\longto\;0\,.
\]
The right term in this sequence is zero because $\Tor_1^{\BdRplus{\wt R}/t^{\ell}}(\wt N,\BdRplus{R}/t^{\ell})\;=\;\Tor_1^{\wt R}(\wt N,R)$ and $\wt N$ is flat over $\wt R$. The left term is zero because $t^{\ell}$ is a non-zero-divisor both in $\BdRplus{\wt R}$ and $\BdRplus{R}$. This shows that $\Tor_1^{\BdRplus{\wt R}}(\wt N,\BdRplus{R})\;=\;0$ and proves the lemma.
\end{proof}

\begin{remark}\label{Rem2.4}
(1) Let $R=L$ be a field and let $(D,\Phi,N,\Fq)$ be a $(\phi,N)$-module with Hodge-Pink lattice over $L$. The Hodge-Pink lattice $\Fq$ gives rise to a $K$-filtration $\CF_\Fq^\bullet$ as follows. Consider the natural projection
\[
\Fp\;\onto\;\Fp/E(u)\Fp\;=\;D\otimes_{R\otimes K_0}\BdRplus{R}/(E(u))\;=\;D\otimes_{R\otimes K_0}R\otimes_{\BQ_p}K\;=\;D_K
\]
and let $\CF_\Fq^i D_K$ be the image of $\Fp\cap E(u)^i\Fq$ in $D_K$ for all $i\in\BZ$, that is
\[
\CF_\Fq^i D_K:=\bigl(\Fp\cap E(u)^i\Fq\bigr)\big/\bigl(E(u)\Fp\cap E(u)^i\Fq\bigr)\,.
\]
Since $L$ is a field $(D,\Phi,N,\CF_\Fq^\bullet)$ is a $K$-filtered $(\phi,N)$-module over $L$. Note that this functor does not exist for general $R$, because $\gr_{\CF_\Fq}^i\!D_K$ will not be locally free over $R$ in general. This is related to the fact that the Hodge polygon of $\CF_\Fq^\bullet$ is locally constant on $R$ whereas the Hodge polygon of $\Fq$ is only semi-continuous; see Remark~\ref{RemMuLocConst} below.

\medskip\noindent
(2) However, for general $R$ consider the category of $(\phi,N)$-modules with Hodge-Pink lattice $(D,\Phi,N,\Fq)$ over $R$, such that $\Fp\subset\Fq\subset E(u)^{-1}\Fp$. This category is equivalent to the category of $K$-filtered $(\phi,N)$-modules $(D,\Phi,N,\CF^\bullet)$ over $R$ with $\CF^0D_K=D_K$ and $\CF^2=0$. Namely, defining $\CF_\Fq^\bullet$ as in (1) we obtain 
\[
\gr_{\CF_\Fq}^i D_K\;\cong\;\left\{\begin{array}{c@{\quad\text{for }}l} E(u)^{-1}\Fp/\Fq & i=0\,, \\[1mm]
\Fq/\Fp & i=1\,, \\[1.5mm] 
0 & i\ne0,1\,,  \end{array}\right.
\]
and so $(D,\Phi,N,\CF_\Fq^\bullet)$ is a $K$-filtered $(\phi,N)$-module by Lemma~\ref{LemmaHPLattice}. Conversely, $\Fq$ equals the preimage of $\CF_\Fq^1D_K$ under the morphism $E(u)^{-1}\Fp\xrightarrow{\es\cdot E(u)\;}\Fp\onto D_K$ and this defines the inverse functor.

\medskip\noindent
(3) Now let $(D,\Phi,N,\CF^\bullet)$ be a $K$-filtered $(\phi,N)$-module over $R$. Using that $\BdRplus{R}=(R\otimes_{\BQ_p}K)\dbl t\dbr$ is an $R\otimes_{\BQ_p}K$-algebra, we can define the Hodge-Pink lattice
\[
\Fq\;:=\;\Fq(\Fcal^\bullet)\;:=\;\sum_{i\in \Z} E(u)^{-i}(\CF^iD_K)\otimes_{R\otimes K}\BdRplus{R}.
\]
It satisfies $\CF_\Fq^\bullet=\CF^\bullet$. Using Lemma \ref{LemmaHPLattice} one easily finds that $\Fq(\Fcal^\bullet)$ is indeed a $\BdRplus{R}$-lattice. 
\end{remark}

\begin{example}\label{ExCyclot1}
The $K$-filtered $(\phi,N)$-modules over $R=\BQ_p$ which correspond to the cyclotomic character $\chi_{\rm cyc}\colon\sG_K\to\BZ_p\mal$ are $D_{\rm st}(\chi_{\rm cyc})=(K_0,\Phi=p^{-1},N=0,\CF^\bullet)$ with $\CF^{-1}=K\supsetneq\CF^0=(0)$ and its dual $D^*_{\rm st}(\chi_{\rm cyc})=(K_0,\Phi=p,N=0,\CF^\bullet)$ with $\CF^1=K\supsetneq\CF^2=(0)$. For both there exists a unique Hodge-Pink lattice which induces the filtration. On $D_{\rm st}(\chi_{\rm cyc})$ it is $\Fq=E(u)\Fp$ and on $D^*_{\rm st}(\chi_{\rm cyc})$ it is $\Fq=E(u)^{-1}\Fp$.
\end{example}

\bigskip

We want to introduce Hodge weights and Hodge polygons. Let $d>0$, let $B\subset\GL_d$ be the Borel subgroup of upper triangular matrices and let $T\subset B$ be the maximal torus consisting of the diagonal matrices. Let $\wt G:=\Res_{K/\BQ_p}\GL_{d,K}$, $\wt B=\Res_{K/\BQ_p}B$ and $\wt T:=\Res_{K/\BQ_p}T$ be the Weil restrictions. We consider cocharacters
\begin{equation}\label{mu}
\mu\colon \Gbb_{m,\ol\BQ_p}\longrightarrow \wt T_{\ol\BQ_p}
\end{equation}
which are dominant with respect to the Borel $\wt B$ of $\wt G$. In other words on $\ol\BQ_p$-valued points the cocharacter 
\[\mu\colon \ol\Q_p^\times\longrightarrow \prod_{\psi\colon K\rightarrow \ol\Q_p}T(\ol\Q_p)\,,\]
where $\psi$ runs over all $\BQ_p$-homomorphisms $\psi\colon K\to\ol\BQ_p$, is given by cocharacters
\[\mu_\psi\colon  x\mapsto {\rm diag}(x^{\mu_{\psi,1}},\dots,x^{\mu_{\psi,d}})\]
for some integers $\mu_{\psi,j}\in\Z$ with $\mu_{\psi,j}\ge\mu_{\psi,j+1}$. We define the \emph{reflex field $E_\mu$ of $\mu$} as the fixed field in $\ol\BQ_p$ of $\{\,\sigma\in\sG_{\BQ_p}\colon \mu_{\sigma\psi,j}=\mu_{\psi,j}\es\forall\;j,\psi\,\}$. It is a finite extension of $\BQ_p$ which is contained in the compositum $\norm{K}$ of all $\psi(K)$ inside $\ol\BQ_p$. For each $j$ the locally constant function $\psi\mapsto\mu_{\psi,j}$ on $\Spec \norm{K}\otimes_{\BQ_p}K\cong\coprod_{\psi\colon K\to\norm{K}}\Spec\norm{K}$ descends to a $\BZ$-valued function $\mu_j$ on $\Spec E_\mu\otimes_{\BQ_p}K$, because $\mu_j$ is constant on the fibers of $\Spec\norm{K}\otimes_{\BQ_p}K\to\Spec E_\mu\otimes_{\BQ_p}K$. In particular, the cocharacter $\mu$ is defined over $E_\mu$. If $R$ is an $E_\mu$-algebra we also view $\mu_j$ as a locally constant $\BZ$-valued function on $\Spec R\otimes_{\BQ_p}K$.

\begin{construction}\label{ConstrHodgeWts}
Let $\ulD=(D,\Phi,N,\Fq)$ be a $(\phi,N)$-module with Hodge-Pink lattice of rank $d$ over a field extension $L$ of $\BQ_p$. By Lemma~\ref{LemmaNnilpot}\ref{LemmaNnilpotA} the $L\otimes_{\BQ_p}K_0$-module $D$ is free. Since $L\otimes_{\BQ_p}K$ is a product of fields, $\BdRplus{L}=(L\otimes_{\BQ_p}K)\dbl t\dbr$ is a product of discrete valuation rings and $\Fq$ is a free $\BdRplus{L}$-module of rank $d$. We choose bases of $D$ and $\Fq$. Then the inclusion $\Fq\subset D\otimes_{L\otimes K_0}\BdR{L}$ is given by an element $\gamma$ of $\GL_d(\BdR{L})=\wt G\bigl(L\dpl t\dpr\bigr)$. By the Cartan decomposition for $\wt G$ there is a uniquely determined dominant cocharacter $\mu_L\colon \BG_{m,L}\to\wt T_L$ over $L$ with $\gamma\in\wt G\bigl(L\dbl t\dbr\bigr)\mu_L(t)^{-1}\wt G\bigl(L\dbl t\dbr\bigr)$. This cocharacter is independent of the chosen bases. If $L$ contains $\norm{K}$, it is defined over $\norm{K}$ because $\wt T$ splits over $\norm{K}$. In this case we view it as an element of $X_*(T_{\norm{K}})_\dom$ and denote it by $\mu_\ulD(\Spec L)$. It has the following explicit description. Under the decomposition $L\otimes_{\BQ_p}K=\prod_{\psi\colon K\to \norm{K}}L$ we have $\wt G\bigl(L\dpl t\dpr\bigr)=\prod_\psi\GL_d\bigl(L\dpl t\dpr\bigr)$, $\mu_L=(\mu_\psi)_\psi$, and $\gamma\in\prod_\psi\GL_d\bigl(L\dbl t\dbr\bigr)\mu_\psi(t)^{-1}\GL_d\bigl(L\dbl t\dbr\bigr)$. The $t^{-\mu_{\psi,1}},\ldots,t^{-\mu_{\psi,d}}$ are the elementary divisors of the $\psi$-component $\Fq_\psi$ of $\Fq$ with respect to $\Fp$. That is, there is an $L\dbl t\dbr$-basis $(v_{\psi,1},\ldots,v_{\psi,d})$ of the $\psi$-component $\Fp_\psi$ of $\Fp$ such that $(t^{-\mu_{\psi,1}}\,v_{\psi,1},\ldots,t^{-\mu_{\psi,d}}\,v_{\psi,d})$ is an $L\dbl t\dbr$-basis of $\Fq_\psi$.

Let $(D,\Phi,N,\CF_\Fq^\bullet)$ be the $K$-filtered $(\phi,N)$-module associated with $\ulD$ by Remark~\ref{Rem2.4}(1). Then $\CF_\Fq^iD_{K,\psi}=\langle v_{\psi,j}\colon i-\mu_{\psi,j}\le0\rangle_L$ and
\[
\dim_L\gr_{\CF_\Fq}^i\!D_{K,\psi}\;=\;\#\{j\colon i-\mu_{\psi,j}=0\}\,. 
\]
More generally, for a $K$-filtered $(\phi,N)$-module $(D,\Phi,N,\CF^\bullet)$ over a field extension $L$ of $\norm{K}$ we consider the decomposition $D_K=\prod_\psi D_{K,\psi}$ and define the integers $\mu_{\psi,1}\ge\ldots\ge\mu_{\psi,d}$ by the formula
\[
\dim_L\gr_\CF^iD_{K,\psi}\;=\;\#\{j\colon \mu_{\psi,j}=i\}\,. 
\]
We define the cocharacter $\mu_{(D,\Phi,N,\CF^\bullet)}(\Spec L):=(\mu_\psi)_\psi$ and view it as an element of $X_*(T_{\norm{K}})_\dom$.
\end{construction}

\begin{defn}\label{Defnbndmu}
\begin{enumerate}
\item[(a)]  Let $R$ be a $\norm{K}$-algebra and consider the decomposition $R\otimes_{\BQ_p}K=\prod_{\psi:K\into\norm{K}}R$. Let $\ulD$ be a $(\phi,N)$-module with Hodge-Pink lattice (respectively a $K$-filtered $(\phi,N)$-module) of rank $d$ over $R$. For every point $s\in\Spec R$ we consider the base change $s^*\ulD$ of $\ulD$ to $\kappa(s)$. We call the cocharacter $\mu_\ulD(s):=\mu_{s^*\ulD}(\Spec \kappa(s))$ from Construction~\ref{ConstrHodgeWts} the \emph{Hodge polygon of $\ulD$ at $s$} and we consider $\mu_\ulD$ as a function $\mu_\ulD:\Spec R\to X_*(T_{\norm{K}})_\dom$. The integers $-\mu_{\psi,j}(s)$ are called the \emph{Hodge weights of $\ulD$ at $s$}.
\end{enumerate}
\noindent
Now let $\mu: \ol\BQ_p^\times\rightarrow \wt T(\ol\BQ_p)$ be a dominant cocharacter as in \eqref{mu}, let $E_\mu$ denote the reflex field of $\mu$, and let $R$ be an $E_\mu$-algebra.
\begin{enumerate}
\item[(b)]
Let $\ulD$ be a $(\phi,N)$-module with Hodge-Pink lattice (respectively a $K$-filtered $(\phi,N)$-module) of rank $d$ over $R$. We say that \emph{$\ulD$ has constant Hodge polygon equal to $\mu$} if $\mu_\ulD(s)=\mu$ for every point $s\in\Spec(R\otimes_{E_\mu}\norm{K})$.
\item[(c)]
Let $\ulD=(D,\Phi,N,\Fq)$ be a $(\phi,N)$-module with Hodge-Pink lattice over $\Spec R$. We say that $\ulD$ has \emph{Hodge polygon bounded by $\mu$} if 
\[
\bigwedge^j_{\BdRplus{R}}\Fq\;\subset\; E(u)^{-\mu_{1}-\ldots-\mu_{j}}\cdot\bigwedge^j_{\BdRplus{R}}\Fp
\]
for all $j=1,\ldots,d$ with equality for $j=d$, where the $\mu_i$ are the $\Z$-valued functions on $\Spec R\otimes_{\Q_p}K$ determined by $\mu$; see the discussion before Construction~\ref{ConstrHodgeWts}.
\end{enumerate}
\end{defn}
Equivalently the condition of being bounded by $\mu$ can be described as follows:  Over $\norm{K}$ the cocharacter $\mu$ is described by a decreasing sequence of integers $\mu_{\psi,1}\geq\dots\geq \mu_{\psi,d}$ for every $\BQ_p$-embedding $\psi:K\hookrightarrow \ol\BQ_p$. Let $R'=R\otimes_{E_\mu} \norm{K}$, then $R'\otimes_{\Q_p}K\cong \prod_{\psi:K\rightarrow \norm{K}} R'_\psi$ with each $R'_\psi=R'$ under the isomorphism $a\otimes b\mapsto (a\psi(b))_\psi$, where $\psi:K\rightarrow R'$ is given via the embedding into the second factor of $R'=R\otimes_{E_\mu}\norm{K}$. Especially we view $R'_\psi$ as a $K$-algebra via $\psi$.
Under this isomorphism $D\otimes_{R\otimes K_0}\boldB_{R'}=:\Fp_{R'}[\tfrac{1}{t}]$ decomposes into a product $\prod_\psi \Fp_{R'}[\tfrac{1}{t}]_\psi$, where $\Fp_{R'}[\tfrac{1}{t}]_\psi$ is a free $R'_\psi\dbl t\dbr[\tfrac{1}{t}]$-module and the $\boldB_{R'}^+$-lattice $\Fp_{R'}\subset \Fp_{R'}[\tfrac{1}{t}]$ decomposes into a product of $R'_\psi\dbl t\dbr$-lattices $\Fp_{R',\psi}\subset \Fp_{R'}[\tfrac{1}{t}]_{\psi}$.

Further, under the isomorphism  $D\otimes_{R\otimes K_0}\boldB_{R'}\cong \prod_\psi \Fp_{R'}[\tfrac{1}{t}]_\psi$ the Hodge-Pink lattice $\Fq_{R'}=\Fq\otimes_RR'$ decomposes into a product $\Fq_{R'}=\prod_\psi \Fq_{R',\psi}$, where $\Fq_{R',\psi}$ is an $R'_\psi\dbl t\dbr$-lattice in $\Fp_{R'}[\tfrac{1}{t}]_\psi$.
Then the condition of being bounded by $\mu$ is equivalent to
\begin{equation}\label{EqDefHodgeWts}
\bigwedge^j_{\BdRplus{R'}}\Fq_{R',\psi}\es\subset\es E(u)^{-\mu_{\psi,1}-\ldots-\mu_{\psi,j}}\cdot\bigwedge^j_{\BdRplus{R'}}\Fp_{R',\psi}
\end{equation}
for all $\psi$ and all $j=1,\ldots,d$ with equality for $j=d$. 

Note that by Cramer's rule (e.g.~\cite[III.8.6, Formulas (21) and (22)]{BourbakiAlgebra}) the condition of Definition~\ref{Defnbndmu}(c), respectively \eqref{EqDefHodgeWts} is equivalent to
\[
\bigwedge^j_{\BdRplus{R}}\Fp\;\subset\; E(u)^{\mu_{d-j+1}+\ldots+\mu_{d}}\cdot\bigwedge^j_{\BdRplus{R}}\Fq\,,
\]
respectively
\begin{equation}\label{EqDefHodgeWtsB}
\bigwedge^j_{\BdRplus{R'}}\Fp_{R',\psi}\es\subset\es E(u)^{\mu_{\psi,d-j+1}+\ldots+\mu_{\psi,d}}\cdot\bigwedge^j_{\BdRplus{R'}}\Fq_{R',\psi}
\end{equation}
for all $j=1,\ldots,d$ with equality for $j=d$.

\begin{remark}\label{RemMuLocConst}
The Hodge polygon of a $K$-filtered $(\phi,N)$-module $(D,\Phi,N,\CF^\bullet)$ is locally constant on $R$, because $\gr^i_\CF D_{K,\psi}$ is locally free over $R$ as a direct summand of the locally free $R$-module $\gr^i_\CF D_K$.

In contrast, the Hodge polygon of a $(\phi,N)$-module $\ulD$ with Hodge-Pink lattice over $R$ is not locally constant in general. Nevertheless, for any cocharacter $\mu$ as in \eqref{mu} the set of points $s\in\Spec R$ such that $\mu_\ulD(s)\preceq\mu$ in the Bruhat order, is closed in $\Spec R$. This is a consequence of the next
\end{remark}

\begin{proposition}\label{PropHWts}
Let $\mu\in X_*(T_{\norm{K}})_\dom$ be a dominant cocharacter with reflex field $E_\mu$ and let $R$ be an $E_\mu$-algebra. Let $\ulD=(D,\Phi,N,\Fq)$ be a $(\phi,N)$-module with Hodge-Pink lattice of rank $d$ over $R$.
\begin{enumerate}
\item \label{PropHWts_A}
The condition that $\ulD$ has Hodge polygon bounded by $\mu$ is representable by a finitely presented closed immersion  $(\Spec R)_{\preceq\mu}\hookrightarrow \Spec R$.
\item \label{PropHWts_B}
If $R$ is reduced then $\ulD$ has Hodge polygon bounded by $\mu$ if and only if for all points $s\in\Spec R\otimes_{E_\mu}\norm{K}$ we have $\mu_\ulD(s)\preceq\mu$ in the Bruhat order, that is, for all $\psi$ the vector $\mu_\psi-\mu_\ulD(s)_\psi\in\BZ^d$ is a non-negative linear combination of the positive coroots $\check{\alpha}_j=(\ldots,0,1,-1,0,\ldots)$ having the ``\:$1$'' as $j$-th entry.
\item \label{PropHWts_C}
Let $\mu'$ be another dominant cocharacter such that $\mu'\preceq\mu$ in the Bruhat order. Let $E_{\mu'}$ denote its reflex field and let $E=E_\mu E_{\mu'}\subset \norm{K}$ be the composite field. Assume that $R$ is an $E$-algebra, then $(\Spec R)_{\preceq\mu'}\hookrightarrow (\Spec R)_{\preceq\mu}$ as closed subschemes of $\Spec R$. 
\end{enumerate}
\end{proposition}

\begin{proof}
\ref{PropHWts_A} By Lemma~\ref{LemmaHPLattice} we find a large positive integer $n$ such that $E(u)^n\Fp\subset\Fq\subset E(u)^{-n}\Fp$. This implies $\wedge^j\Fq\subset E(u)^{-jn}\wedge^j \Fp$ for all $j$ and $\wedge^d \Fp\subset E(u)^{-dn}\wedge^d\Fq$. Viewing $\mu_j\colon\Spec R\otimes_{\BQ_p}K\to\BZ$ as locally constant function as in the discussion before Construction~\ref{ConstrHodgeWts}, we consider the modules over $\BdRplus{R}\cong(R\otimes_{\BQ_p}K)\dbl t\dbr$
\begin{eqnarray}\label{EqM_j}
M_0&:=& E(u)^{-dn}\wedge^d\Fq\,\big/\,E(u)^{\mu_1+\ldots+\mu_d}\cdot\wedge^d\Fq\qquad\text{and}\\[1mm]
M_j&:=& E(u)^{-jn}\wedge^j \Fp\,\big/\,E(u)^{-\mu_1-\ldots-\mu_j}\cdot\wedge^j\Fp\qquad\text{for }1\le j\le d\,.\nonumber
\end{eqnarray}
As $R$-modules they are finite locally free. Then $\ulD$ has Hodge polygon bounded by $\mu$ if and only if for all $j=1,\ldots,d$ all generators of $\wedge^j\Fq$ are mapped to zero in $M_j$ and all generators of $\wedge^d\Fp$ are mapped to zero in $M_0$. Since $M:=M_0\oplus\ldots\oplus M_d$ is finite locally free over $R$, this condition is represented by a finitely presented closed immersion into $\Spec R$ by \cite[I$_{\rm new}$, Lemma~9.7.9.1]{EGA}. 

\medskip\noindent
\ref{PropHWts_B} If $R$ is reduced then also the \'etale $R$-algebra $R':=R\otimes_{E_\mu}\norm{K}$ is reduced and $R\hookrightarrow R'\hookrightarrow \prod_{s\in\Spec R'}\kappa(s)$ is injective. Therefore also $M\hookrightarrow M\otimes_R\bigl(\prod_{s\in\Spec R'}\kappa(s)\bigr)$ is injective. So $\ulD$ has Hodge polygon bounded by $\mu$ if and only if this holds for the pullbacks $s^*\ulD$ to $\Spec\kappa(s)$ at all points $s\in\Spec R'$. By definition of $\mu':=\mu_\ulD(s)$ there is a $\kappa(s)\dbl t\dbr$-basis $(v_{\psi,1},\ldots,v_{\psi,d})$ of the $\psi$-component $(s^*\Fp)_\psi$ of $s^*\Fp$ such that $(t^{-\mu'_{\psi,1}}\,v_{\psi,1},\ldots,t^{-\mu'_{\psi,d}}\,v_{\psi,d})$ is a $\kappa(s)\dbl t\dbr$-basis of $(s^*\Fq)_\psi$. Therefore condition~\eqref{EqDefHodgeWts} holds if and only if $\mu_{\psi,1}+\ldots+\mu_{\psi,j}\ge\mu'_{\psi,1}+\ldots+\mu'_{\psi,j}$ for all $\psi$ and $j$ with equality for $j=d$. One easily checks that this is equivalent to $\mu'\preceq\mu$.

\medskip\noindent
\ref{PropHWts_C} Again $\mu'\preceq\mu$ implies $\mu_{\psi,1}+\ldots+\mu_{\psi,j}\ge\mu'_{\psi,1}+\ldots+\mu'_{\psi,j}$ for all $\psi$ and $j$ with equality for $j=d$. We view $\mu_j,\mu'_j$ as locally constant $\BZ$-valued functions on $\Spec E\otimes_{\BQ_p}K$. Then $\mu_1+\ldots+\mu_j\ge\mu'_1+\ldots+\mu'_j$ for all $j$ with equality for $j=d$. In terms of \eqref{EqM_j} the $R$-modules $M_j$ for $\mu$ are quotients of the $R$-modules $M'_j$ for $\mu'$ with $M'_0=M_0$. Therefore $(\Spec R)_{\preceq\mu'}\into\Spec R$ factors through $(\Spec R)_{\preceq\mu}$.
\end{proof}

\begin{rem}
The reader should note that $\mu'\preceq\mu$ does not imply a relation between $E_{\mu'}$ and $E_\mu$ as can be seen from the following example.
Let $d=2$ and $[K:\BQ_p]=2$ and $\{\psi:K\into\norm{K}\}=\Gal(K/\BQ_p)=\{\psi_1,\psi_2\}$. Consider the three cocharacters $\mu,\mu',\mu''$ given by $\mu_{\psi_1}=(2,0),\,\mu_{\psi_2}=(2,0)$ and $\mu'_{\psi_1}=(2,0),\,\mu'_{\psi_2}=(1,1)$ and $\mu''_{\psi_1}=(1,1),\,\mu''_{\psi_2}=(1,1)$. Then $\mu''\preceq\mu'\preceq\mu$. On the other hand we find $E_\mu=E_{\mu''}=\BQ_p$ and $E_{\mu'}=\norm{K}=K$.
\end{rem}

\begin{rem}
In Definition~\ref{Defnbndmu}(a) we assumed that $R$ is a $\norm{K}$-algebra to obtain a well defined Hodge polygon $\mu_\ulD(s)\in X_*(\wt T_\norm{K})$. In Definition~\ref{Defnbndmu}(b) we can lower the ground field over which $R$ is defined to $E_\mu$ because $\Gal(\norm{K}/E_\mu)$ fixes $\mu$. The ground field cannot be lowered further, as one sees from the following
\end{rem}

\begin{proposition}\label{PropReflexField}
Let $\ulD$ be a $(\phi,N)$-module with Hodge-Pink lattice (respectively a $K$-filtered $(\phi,N)$-module) of rank $d$ over a field $L$ such that $\mu_\ulD(s)=\mu$ for all points $s\in\Spec L\otimes_{\BQ_p}\norm{K}$. Then there is a canonical inclusion of the reflex field $E_\mu\into L$.
\end{proposition}

\begin{proof}
Since every $K$-filtered $(\phi,N)$-module arises from a $(\phi,N)$-module with Hodge-Pink lattice as in Remark~\ref{Rem2.4}(3), it suffices to treat the case where $\ulD$ is a $(\phi,N)$-module with Hodge-Pink lattice. We consider the decomposition $\norm{L}:=L\otimes_{\BQ_p}\norm{K}=\prod_{s\in\Spec\norm{L}}\kappa(s)$ and for each $s$ we denote by $\alpha_s:L\into\kappa(s)$ and $\beta_s:\norm{K}\into\kappa(s)$ the induced inclusions. Let $\mu_L:\BG_{m,L}\to\wt T_L$ be the cocharacter over $L$ associated with $\ulD$ in Construction~\ref{ConstrHodgeWts}. The assumption of the proposition means that $\alpha_s(\mu_L)=\beta_s(\mu)$ for all $s$. The Galois group $\CG:=\Gal(\norm{K},\BQ_p)$ acts on $\norm{L}$. The Galois group $\Gal(\kappa(s)/\alpha_s(L))$ can be identified with the decomposition group $\CG_s:=\{\sigma\in\CG:\sigma(s)=s\}$ under the monomorphism $\Gal(\kappa(s)/\alpha_s(L))\into\CG,\,\tau\mapsto\beta_s^{-1}\circ\tau|_{\beta_s(\norm{K})}\circ\beta_s$. Since $\mu_L$ is defined over $L$, each $\tau\in\Gal(\kappa(s)/\alpha_s(L))$ satisfies $\tau(\alpha_s(\mu_L))=\alpha_s(\mu_L)$, and hence $(\beta_s^{-1}\circ\tau|_{\beta_s(\norm{K})}\circ\beta_s)(\mu)=\mu$. By definition of the reflex field $E_\mu$ this implies that $\beta_s^{-1}\circ\tau|_{\beta_s(\norm{K})}\circ\beta_s\in\Gal(\norm{K}/E_\mu)$ and $\tau|_{\beta_s(E_\mu)}=\id$. So $\beta_s(E_\mu)\subset\alpha_s(L)$ and we get an inclusion $\alpha_s^{-1}\beta_s:E_\mu\into L$. To see that this is independent of $s$ choose a $\sigma\in\CG$ with $\sigma(s)=\tilde s$. Then $\alpha_{\tilde s}=\sigma\circ\alpha_s$ and $\beta_{\tilde s}=\sigma\circ\beta_s$.
\end{proof}

\section{Moduli spaces for $(\phi,N)$-modules with Hodge-Pink lattice}

We will introduce and study moduli spaces for the objects introduced in Chapter~\ref{SectFamilies}. Proposition~\ref{PropReflexField} suggests to work over the reflex field.

\begin{definition}\label{DefStacks}
Let $\mu$ be a cocharacter as in \eqref{mu} and let $E_\mu$ be its reflex field. We define fpqc-stacks $\sD_{\phi,N,\mu}$, resp.\ $\sH_{\phi,N,\preceq\mu}$, resp.\ $\sH_{\phi,N,\mu}$ on the category of $E_\mu$-schemes. For an affine $E_\mu$-scheme $\Spec R$
\begin{enumerate}
\item 
the groupoid $\sD_{\phi,N,\mu}(\Spec R)$ consists of $K$-filtered $(\phi,N)$-modules $(D,\Phi,N,\CF^\bullet)$ over $R$ of rank $d$ with constant Hodge polygon equal to $\mu$.
\item 
the groupoid $\sH_{\phi,N,\preceq\mu}(\Spec R)$ consists of $(\phi,N)$-modules with Hodge-Pink lattice $(D,\Phi,N,\Fq)$ over $R$ of rank $d$ with Hodge polygon bounded by $\mu$.
\item 
the groupoid $\sH_{\phi,N,\mu}(\Spec R)$ consists of $(\phi,N)$-modules with Hodge-Pink lattice $(D,\Phi,N,\Fq)$ over $R$ of rank $d$ with Hodge polygon bounded by $\mu$ and constant equal to $\mu$.
\end{enumerate}
Let $\sD_{\phi,\mu}\subset\sD_{\phi,N,\mu}$, resp.\ $\sH_{\phi,\preceq\mu}\subset\sH_{\phi,N,\preceq\mu}$, resp.\ $\sH_{\phi,\mu}\subset\sH_{\phi,N,\mu}$ be the closed substacks on which $N$ is zero. They classify $\phi$-modules with $K$-filtration, resp.\ Hodge-Pink lattice and the corresponding condition on the Hodge polygon.
\end{definition}

We are going to show that these stacks are Artin stacks of finite type over $E_\mu$. 

Locally on $\Spec R$ we may choose an isomorphism $D\cong(R\otimes_{\BQ_p}K_0)^d$ by Lemma~\ref{LemmaNnilpot}\ref{LemmaNnilpotA}. Then $\Phi$ and $N$ correspond to matrices $\Phi\in\GL_d(R\otimes_{\BQ_p}K_0)=(\Res_{K_0/\BQ_p}\GL_{d,K_0})(R)$ and $N\in {\rm Mat}_{d\times d}(R\otimes_{\BQ_p}K_0)=(\Res_{K_0/\BQ_p}{\rm Mat}_{d\times d})(R)$. The relation $\Phi\,\phi^\ast N=p\,N\,\Phi$ is represented by a closed subscheme
\[
P_{K_0,d}\;\subset\;(\Res_{K_0/\BQ_p}\GL_{d,K_0})\times_{\Spec\BQ_p}(\Res_{K_0/\BQ_p}{\rm Mat}_{d\times d})\,.
\]

\begin{theorem}\label{ThmReduced}
\begin{enumerate}
\item[(a)] The $\BQ_p$-scheme $P_{K_0,d}$ is reduced, Cohen-Macaulay, generically smooth and equidimensional of dimension $fd^2$. In the notation of Remark~\ref{RemDecompOfD} the matrix $(\Phi^f)_0$ has no multiple eigenvalues at the generic points of the irreducible components of $P_{K_0,d}$. 
\item[(b)]  The generic points of $P_{K_0,d}$ are in bijection with the partitions $d=k_1+\ldots+k_m$ for integers $m$ and $1\le k_1\le\ldots\le k_m$. To such a partition corresponds the generic point at which the suitably ordered eigenvalues $\lambda_1,\ldots,\lambda_d$ of $(\Phi^f)_0$ satisfy $p^f\lambda_i=\lambda_j$ if and only if $j=i+1$ and $i\notin\{k_1,\,k_1+k_2,\,\ldots\,,\,k_1+\ldots+k_m\}$. Equivalently to such a partition corresponds the generic point at which the nilpotent endomorphism $N_0$, in the notation of \ref{RemDecompOfD}, has Jordan canonical form with $m$ Jordan blocks of size $k_1,\dots, k_m$. 
\end{enumerate}
\end{theorem}

For the proof we will need the following lemma.

\begin{lemma}\label{LemmaCombinatorics}
Let $r_1,\ldots,r_n$ be integers with $r_1+\ldots+r_n\ge n$. Then $\sum_{i=1}^n r_i^2 \;-\; \sum_{i=1}^{n-1} r_i\,r_{i+1} \,>\, 1$, except for the case when $r_1=\ldots =r_n=1$. 
\end{lemma}

\begin{proof}
We multiply the claimed inequality with $2$ and write it as $r_1^2+\sum_{i=1}^{n-1} (r_i-r_{i+1})^2 +r_n^2 >2$. There are the following three critical cases
\begin{enumerate}
\item \label{LemmaCombinatorics_A}
$\sum_i (r_i-r_{i+1})^2=0$,
\item \label{LemmaCombinatorics_B}
$\sum_i (r_i-r_{i+1})^2=1$,
\item \label{LemmaCombinatorics_C}
$\sum_i (r_i-r_{i+1})^2=2$.
\end{enumerate}

In case \ref{LemmaCombinatorics_A} we have $r_1=\ldots=r_n$. Since $r_1=\ldots=r_n=1$ was excluded and $r_1\le0$ contradicts $r_1+\ldots+r_n\ge n$, we have $r_1^2+r_n^2>2$.

In case \ref{LemmaCombinatorics_B} there is exactly one index $1\le i<n$ with $r_1=\ldots =r_i\ne r_{i+1}=\ldots =r_n$ and $|r_i-r_{i+1}|=1$. If $r_1\ne0\ne r_n$ then $r_1^2+\sum_{i=1}^{n-1} (r_i-r_{i+1})^2 +r_n^2 >2$. On the other hand, if $r_1=\pm 1$ and $r_n=0$, then $\sum_\nu r_\nu = \pm i<n$. And if $r_1=0$ and $r_n=\pm 1$, then $\sum_\nu r_\nu=\pm(n-i)<n$. Both are contradictions.

In case \ref{LemmaCombinatorics_C} there are exactly two indices $1\le i<j<n$ with $r_1=\ldots =r_i$ and $r_{i+1}=\ldots =r_j$ and $r_{j+1}=\ldots =r_n$, as well as $|r_i-r_{i+1}|=1=|r_j-r_{j+1}|$. If in addition $r_1=r_n=0$ then $\sum_i r_i = \pm(j-i) <n$, which is a contradiction. Therefore $r_1^2+r_n^2>0$ and $r_1^2+\sum_{i=1}^{n+1} (r_i-r_{i+1})^2 +r_n^2 >2$ as desired.
\end{proof}

\begin{proof}[Proof of Theorem~\ref{ThmReduced}]
We break the proof into several steps.

1. By \cite[IV$_2$, Proposition~6.5.3, Corollaires~6.3.5(ii), 6.1.2, and IV$_4$, Proposition~17.7.1]{EGA} the statement may be checked after the finite \'etale base change $\Spec K_0\to\Spec\BQ_p$. We will use throughout that after this base change, Remark~\ref{RemDecompOfD} allows to decompose $\Phi=(\Phi_i)_i$ and $N=(N_i)_i$ such that $p\,\Phi_i\circ N_i=N_{i+1}\circ\Phi_i$. 

\medskip\noindent
2. We first prove that all irreducible components of $P_{K_0,d}$ have dimension greater or equal to $fd^2$. Sending $(\Phi,N)$ to the entries of the matrices $\Phi_i,N_i$ embeds $P_{K_0,d}\times_{\BQ_p} K_0$ into affine space $\BA_{K_0}^{2fd^2}$ as a locally closed subscheme cut out by the $fd^2$ equations $p\,\Phi_i\circ N_i=N_{i+1}\circ\Phi_i$ for $i=0,\ldots,f-1$. Therefore the codimension of $P_{K_0,d}\times_{\BQ_p} K_0$ in $\BA_{K_0}^{2fd^2}$ is less or equal to $fd^2$ by Krull's principal ideal theorem~\cite[Theorem~10.2]{Eisenbud}, and all irreducible components of $P_{K_0,d}$ have dimension greater or equal to $fd^2$ by \cite[Corollary~13.4]{Eisenbud}. 

\medskip\noindent
3. We next prove the assertion on the generic points. Let $y=(\Phi,N)$ be the generic point of an irreducible component $Y$ of $P_{K_0,d}$. After passing to an algebraic closure $L$ of $\kappa(y)$ we may use Remark~\ref{RemDecompOfD} to find a base change matrix $S\in\GL_d(L\otimes_{\BQ_p}K_0)$ such that $S^{-1}\Phi\,\phi(S)=\bigl((\Phi^f)_0,\Id_d,\ldots,\Id_d)$ and $(\Phi^f)_0$ is a block diagonal matrix in Jordan canonical form
\[
(\Phi^f)_0 \,=\, \left(\raisebox{3ex}{$
\xymatrix @C=0.3pc @R=0.3pc {
J_1 \ar@{.}[drdr] & & \\
& & \\
& & J_r
}$}
\right)
\es\text{with}\es
J_i \,=\, \left(\raisebox{3.5ex}{$
\xymatrix @C=0.3pc @R=0.3pc {
\rho_i \ar@{.}[drdr] & 1 \ar@{.}[dr] & \\
& & 1\\
& & \rho_i
}$}\right)
\quad\text{and}\quad
N_0 \,=\, \left(\raisebox{3ex}{$
\xymatrix @C=0.3pc @R=0.3pc {
N_{11} \ar@{.}[rr]\ar@{.}[dd] & & N_{1r}\ar@{.}[dd]\\
& & \\
N_{r1} \ar@{.}[rr] & & N_{rr}
}$}
\right).
\]
Note that a priori some of the $\rho_i$ can be equal. Let $s_i$ be the size of the Jordan block $J_i$. Then $N_{ij}$ is an $s_i\times s_j$-matrix. The condition $p^f(\Phi^f)_0\circ N_0=N_0\circ(\Phi^f)_0$ is equivalent to $p^f\!J_i\, N_{ij}=N_{ij}J_j$ for all $i,j$. It yields $N_{ij}=(0)$ for $p^f\rho_i\ne\rho_j$. By renumbering the $J_i$ we may assume that $N_{ij}\ne(0)$ implies $i<j$. We set $N_{ij}=\bigl(n^{(ij)}_{\mu,\nu}\bigr)_{\mu=1\ldots s_i,\,\nu=1\ldots s_j}$. When $p^f\rho_i=\rho_j$ it follows from
\[
\left(\raisebox{5.2ex}{$
\xymatrix @C=0.3pc @R=0.3pc {
p^f n_{2,1} \ar@{.}[rr]\ar@{.}[dd] & & p^f n_{2,s_j}\ar@{.}[dd]\\
& & \\
p^f n_{s_i,1} \ar@{.}[rr] & & p^f n_{s_i,s_j}\\
0 \ar@{.}[rr] & & 0
}$}
\right)
\;=\;
p^f( J_i-\rho_i) N_{ij} \;=\;  N_{ij}( J_j-\rho_j) \; = \;
\left(\raisebox{4.5ex}{$
\xymatrix @C=0.3pc @R=0.3pc {
0 \ar@{.}[ddd] &  n_{1,1} \ar@{.}[rr]\ar@{.}[ddd] & &  n_{1,s_j-1}\ar@{.}[ddd]  \\
& & & & \\
& & & & \\
0_{_{}} &  n_{s_i,1} \ar@{.}[rr] & &  n_{s_i,s_j-1} 
}$}
\right) 
\]
that $p^f n^{(ij)}_{\mu,\nu}=n^{(ij)}_{\mu-1,\nu-1}$ for all $\mu,\nu\ge2$ and $n^{(ij)}_{\mu,\nu}=0$ whenever $\mu-\nu>\min\{0,s_i-s_j\}$. We set $s:=\max\{s_i\}$. The assertion of the theorem says that $s=1$ and that all $\rho_i$ are pairwise different.

First assume that $s>1$. We exhibit a morphism $\Spec L[z,z^{-1}]\to P_{K_0,d}$ which sends the point $\{z=1\}$ to $y$ and the generic point $\Spec L(z)$ to a point at which the maximal size of the Jordan blocks is strictly less than $s$. Since $y$ was a generic point of $P_{K_0,d}$ this is impossible. The morphism $\Spec L[z,z^{-1}]\to P_{K_0,d}$ is given by matrices $\wt S$, $(\wt\Phi^f)_0$ and $\wt N_0$ as follows. We set $\wt S:=S$. For all $i$ with $s_i=s$ we set 
\[
\wt J_i \,:=\, \left(\raisebox{5.2ex}{$
\xymatrix @C=0.3pc @R=0.3pc {
\rho_i \ar@{.}[drdr] & 1 \ar@{.}[drdr] & & \\
& & & \\
& & \rho_i & 1\\
& & & z\rho_i
}$}\right)
\]
and for all $i$ with $s_i<s$ we set $\wt J_i:=J_i$. When $p^f\rho_i\ne\rho_j$ we set $\wt N_{ij}:=(0)$. To define $\wt N_{ij}$ when $p^f\rho_i=\rho_j$, and hence $i<j$, we distinguish the following cases
\begin{enumerate}
\item \label{JN_A}
If $s_i,s_j<s$ we set $\wt N_{ij}=N_{ij}$.
\item \label{JN_B}
If $s_i=s>s_j$ we set $\wt N_{ij}=N_{ij}$.
\item \label{JN_C}
If $s_i<s=s_j$ we set $\wt N_{ij}=\bigl(\tilde n^{(ij)}_{\mu,\nu}\bigr)_{\mu,\nu}$ with $\tilde n^{(ij)}_{\mu,s_j}:=n^{(ij)}_{\mu,s_j}$ for all $\mu$, with $\tilde n^{(ij)}_{\mu,\nu}:=0$ whenever $\mu>\nu+s_i-s_j+1$, and with $\tilde n^{(ij)}_{\mu,\nu}:=n^{(ij)}_{\mu,\nu}+(1-z)p^{(s_j-1-\nu)f}\cdot\rho_j\cdot n^{(ij)}_{\mu-\nu+s_j-1,\,s_j}$ for $\nu<s_j$ and $\mu\le\nu+s_i-s_j+1$.
\item \label{JN_D}
If $s_i=s_j=s$ we set $\wt N_{ij}=\bigl(\tilde n^{(ij)}_{\mu,\nu}\bigr)_{\mu,\nu}$ with $\tilde n^{(ij)}_{\mu,s}:= n^{(ij)}_{\mu,s}$ for all $\mu$, with $\tilde n^{(ij)}_{\mu,\nu}:=0$ whenever $\mu>\nu$, and with $\tilde n^{(ij)}_{\mu,\nu}:= n^{(ij)}_{\mu,\nu} +(1-z)p^{(s-1-\nu)f}\cdot\rho_j\cdot n^{(ij)}_{\mu-\nu+s-1,\,s}$ for all $\mu\le\nu<s$.
\end{enumerate}
We have to check that $p^f\!\wt J_i\,\wt N_{ij}=\wt N_{ij}\wt J_j$ for all $i,j$ with $p^f\rho_i=\rho_j$. In case \ref{JN_A} this is obvious and in case \ref{JN_B} it follows from the fact that the bottom row of $N_{ij}$ is zero. For case \ref{JN_C} we compute
\begin{eqnarray*}
p^f(\wt J_i-\rho_i)\wt N_{ij} & = &
\left(\raisebox{5.2ex}{$
\xymatrix @C=0.3pc @R=0.3pc {
p^f\tilde n_{2,1} \ar@{.}[rr]\ar@{.}[dd] & & p^f\tilde n_{2,s_j}\ar@{.}[dd]\\
& & \\
p^f\tilde n_{s_i,1} \ar@{.}[rr] & & p^f\tilde n_{s_i,s_j}\\
0 \ar@{.}[rr] & & 0
}$}
\right)
\es = \\[2mm]
& = &
\left(\raisebox{4.5ex}{$
\xymatrix @C=0.3pc @R=0.3pc {
0 \ar@{.}[ddd] & \tilde n_{1,1} \ar@{.}[rr]\ar@{.}[ddd] & & \tilde n_{1,s_j-2}\ar@{.}[ddd] & \tilde n_{1,s_j-1}+(z-1)\rho_j\,\tilde n_{1,s_j} \ar@{.}[ddd] \\
& & & & \\
& & & & \\
0 & \tilde n_{s_i,1} \ar@{.}[rr] & & \tilde n_{s_i,s_j-2} & \tilde n_{s_i,s_j-1}+(z-1)\rho_j\,\tilde n_{s_i,s_j}
}$}
\right) 
\es=\es
\wt N_{ij}(\wt J_j-\rho_j)\,.
\end{eqnarray*}
Finally for case \ref{JN_D} we compute
\begin{eqnarray*}
p^f(\wt J_i-\rho_i)\wt N_{ij} & = &
\left(\raisebox{5.2ex}{$
\xymatrix @C=0.3pc @R=0.3pc {
p^f\tilde n_{2,1} \ar@{.}[rr]\ar@{.}[dd] & & p^f\tilde n_{2,s-1}\ar@{.}[dd] & p^f\tilde n_{2,s}\ar@{.}[dd]\\
& & & \\
p^f\tilde n_{s,1} \ar@{.}[rr] & & p^f\tilde n_{s,s-1} & p^f\tilde n_{s,s}\\
0 \ar@{.}[rr] & & 0 & (z-1)p^f\!\rho_i\,\tilde n_{s,s}
}$}
\right)\es=\\[2mm]
& = & 
\left(\raisebox{4.5ex}{$
\xymatrix @C=0.3pc @R=0.3pc {
0 \ar@{.}[ddd] & \tilde n_{1,1} \ar@{.}[rr]\ar@{.}[ddd] & & \tilde n_{1,s-2}\ar@{.}[ddd] & \tilde n_{1,s-1}+(z-1)\rho_j\,\tilde n_{1,s} \ar@{.}[ddd] \\
& & & & \\
& & & & \\
0 & \tilde n_{s,1} \ar@{.}[rr] & & \tilde n_{s,s-2} & \tilde n_{s,s-1}+(z-1)\rho_j\,\tilde n_{s,s}
}$}
\right)
\es=\es \wt N_{ij}(\wt J_j-\rho_j)\,.
\end{eqnarray*}
Altogether this defines the desired morphism $\Spec L[z,z^{-1}]\to P_{K_0,d}$.

So we have shown that $s=1$ at the generic point $y$ and that $(\Phi^f)_0$ is a diagonal matrix. We still have to show that all diagonal entries are pairwise different. For this purpose we rewrite $(\Phi^f)_0$ and $N_0$ as 
\[
(\Phi^f)_0 \,=\, \left(\raisebox{3ex}{$
\xymatrix @C=0.3pc @R=0.3pc {
\lambda_1\Id_{r_1} \ar@{.}[drdr] & & \\
& & \\
& & \lambda_n\Id_{r_n}
}$}
\right)
\qquad\text{and}\qquad
N_0 \,=\, \left(\raisebox{3ex}{$
\xymatrix @C=0.3pc @R=0.3pc {
M_{11} \ar@{.}[rr]\ar@{.}[dd] & & M_{1n}\ar@{.}[dd]\\
& & \\
M_{n1} \ar@{.}[rr] & & M_{nn}
}$}
\right).
\]
We denote the multiplicity of the eigenvalue $\lambda_i$ by $r_i\ge1$. Then $M_{ij}$ is an $r_i\times r_j$-matrix. By renumbering the $\lambda_i$ we may assume that there are indices $0=l_0<l_1<\ldots<l_m=d$ such that $p^f\!\lambda_i = \lambda_j$ if and only if $j=i+1$ and $i\notin\{l_1,\ldots,l_m\}$. 

We compute $\dim Y=\trdeg_{\BQ_p}\kappa(y)=\trdeg_{\BQ_p}L$ as follows. The eigenvalues $\lambda_{l_1},\ldots,\lambda_{l_m}$ contribute at most the summand $m$ to $\trdeg_{\BQ_p}L$. 

The matrix $S\in\GL_d(L\otimes_{\BQ_p}K_0)$ is determined only up to multiplication on the right with an element of the $\phi$-centralizer $\CC(L):=\{\,S\in\GL_d(L\otimes_{\BQ_p}K_0)\colon S\bigl((\Phi^f)_0,\Id_d,\ldots,\Id_d\bigr)=\bigl((\Phi^f)_0,\Id_d,\ldots,\Id_d\bigr)\phi(S)\,\}$ of $\bigl((\Phi^f)_0,\Id_d,\ldots,\Id_d\bigr)$. Writing $S=(S_0,\ldots,S_{f-1})$ this condition implies that $S_i\stackrel{!}{=}(\phi(S))_i:=S_{i-1}$ for $i=1,\ldots,f-1$ and $S_0(\Phi^f)_0\stackrel{!}{=}(\Phi^f)_0(\phi(S))_0:=(\Phi^f)_0S_{f-1}=(\Phi^f)_0S_0$. Therefore $\CC$ has dimension $\sum_i r_i^2$ and the entries of $S\in(\Res_{K_0/\BQ_p}\GL_{d,K_0})/\CC$ contribute at most the summand $fd^2-\sum_i r_i^2$ to $\trdeg_{\BQ_p}L$.

The condition $p^f(\Phi^f)_0\circ N_0=N_0\circ(\Phi^f)_0$ is equivalent to $p^f\!\lambda_i\, M_{ij}=\lambda_jM_{ij}$ for all $i,j$. This implies that there is no condition on $M_{ij}$ when $j=i+1$ and $i\notin\{l_1,\ldots,l_m\}$, and that all other $M_{ij}$ are zero. So the entries of the $M_{ij}$ contribute at most the summand $\sum_{i\notin\{l_1,\ldots,l_m\}}r_i\,r_{i+1}$ to $\trdeg_{\BQ_p}L$.

Adding all summands and comparing with our estimate in part 2 above, we obtain
\begin{eqnarray*}
fd^2 \es \le \es \dim Y \es = \es \trdeg_{\BQ_p}\kappa(y) & \le & m + fd^2 \;-\;\sum_{i=1}^n r_i^2 \;+\; \sum_{i\notin\{l_1,\ldots,l_m\}}\!\!\!\!\!\!r_i\,r_{i+1}\\
& = & fd^2+\sum_{\nu=0}^{m-1}\Bigl(1 -\sum_{i=1+l_\nu}^{l_{\nu+1}} r_i^2 \;+ \sum_{i=1+l_\nu}^{l_{\nu+1}-1}r_i\,r_{i+1}\Bigr)\,.
\end{eqnarray*}
By Lemma~\ref{LemmaCombinatorics} the parentheses are zero when all $r_i=1$, and negative otherwise. So we have proved that $r_1=\ldots=r_n=1$. In other words, all diagonal entries of $(\Phi^f)_0$ are pairwise different. Let $k_\nu:=l_\nu-l_{\nu-1}$ for $\nu=1,\ldots,m$. Then the generic point $y$ corresponds to the partition $d=k_1+\ldots+k_m$ under the description of the generic points in the theorem. 
As we have noticed above the $1\times 1$ matrices $M_{ij}$ vanish at $y$ unless $j=i+1$ and $i\notin \{l_1,\dots, l_m\}$ and in the latter case we must have $M_{ij}(y)\neq 0$. This implies the claim on the Jordan type of $N_0$ at the generic points of the irreducible components. 

Moreover, it follows that $\dim Y=fd^2$ for all irreducible components $Y$ of $P_{K_0,d}$. By \cite[Proposition~18.13]{Eisenbud} this also implies that $P_{K_0,d}\times_{\BQ_p} K_0$ is Cohen-Macaulay.

\medskip\noindent
4. It remains to show that $P_{K_0,d}$ is generically smooth over $\BQ_p$. From this it follows that it is reduced, because it is Cohen-Macaulay. Let again $y$ be the generic point of an irreducible component of $P_{K_0,d}\times_{\BQ_p}K_0$ and let $L$ be an algebraic closure of $\kappa(y)$. As above, Remark~\ref{RemDecompOfD} allows us to change the basis over $L$ and assume that $\Phi=\bigl((\Phi^f_0),\id,\ldots,\id\bigr)$ and $N=(p^iN_0)_i$ with $(\Phi^f)_0={\rm diag}(\lambda_1,\ldots,\lambda_d)$ and $\lambda_i\ne\lambda_j$ for all $i\ne j$. We write $F^{(0)}:=(\Phi^f)_0$ and $N_0^{(0)}:=N_0=(n_{ij})_{ij}$. The condition $N_0^{(0)}F^{(0)}=p^fF^{(0)}N_0^{(0)}$ implies that $n_{ij}=0$ if $p^f\lambda_i\ne\lambda_j$. And conversely $n_{ij}\ne0$ if $p^f\lambda_i=\lambda_j$ by our explicit description of $N_0$ at $y$ above.

We claim that for every $n\ge1$, any deformation $(F^{(n-1)},N_0^{(n-1)})\in P_{K_0,d}(L[\epsilon]/\epsilon^n)$ of $(F^{(0)},N_0^{(0)})$ can be lifted further to $(F^{(n)},N_0^{(n)})\in P_{K_0,d}(L[\epsilon]/\epsilon^{n+1})$.  This implies that $P_{K_0,d}$ is smooth at $y$, as it follows that any tangent vector $\Ocal_{P_{K_0,d},y}\rightarrow L[\epsilon]/\epsilon^2$ comes from a map $\Ocal_{P_{K_0,d},y}\rightarrow L\dbl \epsilon\dbr$ and hence the image of 
\[\Spec\big(\bigoplus\nolimits_{i\geq 0} (\mfrak_{P_{K_0,d},y}^i/\mfrak_{P_{K_0,d},y}^{i+1})\big)\rightarrow \Spec\big(\Sym^\bullet (\mfrak_{P_{K_0,d},y}/\mfrak_{P_{K_0,d},y}^2)\big) \]
contains any tangent vector. This means that the tangent cone at $y$ equals the tangent space, and hence by \cite[III, \S 4 Definition 2 and Corollary 1]{Mumford} $P_{K_0,d}$ is smooth at $y$. 
Let us take any deformation $(\wt F,\wt N_0)\in \GL_d(L[\epsilon]/\epsilon^{n+1})\times {\rm Mat}_{d\times d}(L[\epsilon]/\epsilon^{n+1})$ of $(F^{(n-1)},N_0^{(n-1)})$. Then we have $\wt N_0\wt F-p^f\wt F \wt N_0\in \epsilon^n{\rm Mat}_{d\times d}(L)$. Changing $(\wt F,\wt N_0)$ to $(F^{(n)},N_0^{(n)})=(\wt F+\epsilon^nF',\wt N_0+\epsilon^nN'_0)$ with $F',N'_0\in {\rm Mat}_{d\times d}(L)$ we find 
\[N_0^{(n)}F^{(n)} - p^fF^{(n)} N_0^{(n)}= (\wt N_0\wt F-p^f\wt F \wt N_0) +\epsilon^n(N_0^{(0)}F'+N'_0F^{(0)}-p^f(F' N_0^{(0)}+F^{(0)} N'_0))\]
and hence it suffices to show that the map $h:{\rm Mat}_{d\times d}(L)\times {\rm Mat}_{d\times d}(L)\rightarrow {\rm Mat}_{d\times d}(L)$ given by
\[h:(F',N'_0)\longmapsto (N'_0F^{(0)}-p^fF^{(0)} N'_0)+(N_0^{(0)}F'-p^fF' N_0^{(0)})\]
is surjective. For this purpose let $N'_0=(b_{ij})_{ij}$ and $F'={\rm diag}(a_1,\ldots,a_d)$. Then we find that 
\[
h_{ij}(F',N'_0)\;:=\;h(F',N'_0)_{ij}\;=\;(\lambda_j-p^f\lambda_i)b_{ij} + a_j\,n_{ij}-p^fa_i\,n_{ij}\,.
\]
Whenever $\lambda_j-p^f\lambda_i\in L^\times$ and hence $n_{ij}=0$ we obtain the surjectivity of $h_{ij}$.
By permuting the indices we may assume that $n_{ij}\ne0$ implies $j=i+1$. Treating every Jordan block of $N_0^{(0)}$ separately we may further assume that $n_{i,i+1}\ne0$ for all $i$. It then follows that we have 
\[h_{i,i+1}(F',N_0')=(a_{i+1}-p^fa_i)n_{i,i+1}\]
which suffices to see that the map $h$ is surjective. 
\end{proof}

\begin{remark}
The scheme $P_{K_0,d}$ is in general not normal. For example if $K=\BQ_p$ and $d=2$ then $P_{\BQ_p,2}$ has two generic points. 
This was already proven in \cite[Lemma A.3]{KisinFM}.
In one of the generic points $\Phi$ has eigenvalues $\lambda,p\lambda$ and $N\ne0$. In the other $\Phi$ has eigenvalues $\lambda_1,\lambda_2$ with $\lambda_j\ne p\lambda_i$ for all $i,j$ and $N=0$. Both irreducible components meet in the codimension one point where $\lambda_2=p\lambda_1$ and $N=0$.
\end{remark}

Let $\Delta$ denote the set of simple roots (defined over $\ol\Q_p$) of $\wt G:=\Res_{K/\Q_p}\GL_{d,K}$ with respect to the Borel subgroup $\wt B$ and denote by $\Delta_\mu\subset \Delta$ the set of all simple roots $\alpha$ such that $\langle \alpha,\mu\rangle=0$. Here $\langle -,-\rangle$ is the canonical pairing between characters and cocharacters.
We write $P_\mu$ for the parabolic subgroup of $\wt G$ containing $\wt B$ and corresponding to $\Delta_\mu\subset \Delta$. This parabolic subgroup is defined over $E_\mu$, and the quotient by this parabolic is a projective $E_\mu$-variety
\begin{equation}\label{EqFlag}
\Flag_{K,d,\mu}=\wt G_{E_\mu}/P_\mu
\end{equation}
representing the functor 
\[
R\mapsto\{\text{filtrations}\ \Fcal^\bullet\ \text{of}\ R\otimes_{\Q_p}K^{\oplus d}\ \text{with constant Hodge polygon equal to}\ \mu\}
\]
Thus $\sD_{\phi,N,\mu}$ and $\sD_{\phi,\mu}$ are isomorphic to the stack quotients
\begin{equation}\label{sDpresentation}
\begin{aligned}
\sD_{\phi,N,\mu}&\cong&(P_{K_0,d}\,\times_{\Spec\BQ_p}\,\Flag_{K,d,\mu})\;\big/\;(\Res_{K_0/\BQ_p}\GL_{d,K_0})_{E_\mu}\qquad\text{and}\\[2mm]
\sD_{\phi,\mu}&\cong&(\Res_{K_0/\BQ_p}\GL_{d,K_0}\,\times_{\Spec\BQ_p}\,\Flag_{K,d,\mu})\;\big/\;(\Res_{K_0/\BQ_p}\GL_{d,K_0})_{E_\mu}
\end{aligned}
\end{equation}
where $g\in(\Res_{K_0/\BQ_p}\GL_{d,K_0})_{E_\mu}$ acts on $(\Phi,N,\CF^\bullet)\in P_{K_0,d}\,\times_{\Spec\BQ_p}\,\Flag_{K,d,\mu}$ by 
\[
(\Phi,N,\CF^\bullet)\;\longmapsto\;\bigl(g^{-1}\Phi\,\phi(g)\,,\,g^{-1}Ng\,,\,g^{-1}\CF^\bullet\bigr)\,.
\]

\bigskip

We next describe the moduli space for the Hodge-Pink lattice $\Fq$. Fix the integers $m=\max\{\,\mu_{\psi,1}\,|\,\psi:K\to\norm{K}\,\}$ and $n=\max\{\,-\mu_{\psi,d}\,|\,\psi:K\to\norm{K}\,\}$. Then by Cramer's rule $E(u)^n\Fp\subset\Fq\subset E(u)^{-m}\Fp$. So $\Fq$ is determined by the epimorphism
\begin{equation}\label{Eqpr}
pr\colon R\otimes_{\BQ_p}(K[t]/t^{m+n})^{\oplus d} \;\isoto\; E(u)^{-m}\Fp/E(u)^n\Fp \;\onto\; E(u)^{-m}\Fp/\Fq
\end{equation}
which is induced by choosing an isomorphism $D\cong(R\otimes_{\BQ_p}K_0)^d$ locally on $R$. The quotient $E(u)^{-m}\Fp/\Fq$ is a finite locally free $R$-module and of finite presentation over $R\otimes_{\BQ_p}K[t]/t^{m+n}$ by Lemma~\ref{LemmaHPLattice}. Therefore it is an $R$-valued point of Grothendieck's Quot-scheme $\Quot_{\CO^d\,|\,K[t]/t^{m+n}\,|\,\BQ_p}$; see \cite[n$\open$221, Theorem 3.1]{FGA} or \cite[Theorem 2.6]{AltmanKleiman}. This Quot-scheme is projective over $\BQ_p$. The boundedness by $\mu$ is represented by a closed subscheme $Q_{K,d,\preceq\mu}$ of $\Quot_{\CO^d\,|\,K[t]/t^{m+n}\,|\,\BQ_p}\times_{\Spec\BQ_p}\Spec E_\mu$ according to Proposition~\ref{PropHWts}\ref{PropHWts_A}. Thus $\sH_{\phi,N,\preceq\mu}$ and $\sH_{\phi,\preceq\mu}$ are isomorphic to the stack quotients
\begin{eqnarray*}
\sH_{\phi,N,\preceq\mu}&\cong&(P_{K_0,d}\,\times_{\Spec\BQ_p}\,Q_{K,d,\preceq\mu})\;\big/\;(\Res_{K_0/\BQ_p}\GL_{d,K_0})_{E_\mu}\qquad\text{and}\\[2mm]
\sH_{\phi,\preceq\mu}&\cong&(\Res_{K_0/\BQ_p}\GL_{d,K_0}\,\times_{\Spec\BQ_p}\,Q_{K,d,\preceq\mu})\;\big/\;(\Res_{K_0/\BQ_p}\GL_{d,K_0})_{E_\mu}
\end{eqnarray*}
where $g\in(\Res_{K_0/\BQ_p}\GL_{d,K_0})_{E_\mu}$ acts on $(\Phi,N,pr)\in P_{K_0,d}\,\times_{\Spec\BQ_p}\,Q_{K,d,\preceq\mu}$ with $pr$ from \eqref{Eqpr} by 
\[
(\Phi,N,pr)\;\longmapsto\;\bigl(g^{-1}\Phi\,\phi(g)\,,\,g^{-1}Ng\,,\,pr\circ(g\otimes_{\BQ_p}\id_{\BQ_p[t]/t^{m+n}})\bigr)\,.
\]

Let $Q_{K,d,\mu}$ be the complement in $Q_{K,d,\preceq\mu}$ of the image of $\bigcup_{\mu'\prec\mu}Q_{K,d,\preceq\mu'}\times_{\Spec E_{\mu'}}\Spec \norm{K}$ under the finite \'etale projection $Q_{K,d,\preceq\mu}\times_{\Spec E_\mu}\Spec \norm{K}\to Q_{K,d,\preceq\mu}$. Here the union is taken over all dominant cocharacters $\mu':\Gbb_{m, \ol\Q_p}\to \wt T_{\ol\Q_p}$ which are strictly less than $\mu$ in the Bruhat order; see Proposition~\ref{PropHWts}\ref{PropHWts_B}. Since there are only finitely many such $\mu'$ the scheme $Q_{K,d,\mu}$ is an open subscheme of $Q_{K,d,\preceq\mu}$ and quasi-projective over $E_\mu$.
By Proposition~\ref{PropHWts}\ref{PropHWts_A} the stacks $\sH_{\phi,N,\mu}\subset\sH_{\phi,N,\preceq\mu}$ and $\sH_{\phi,\mu}\subset\sH_{\phi,\preceq\mu}$ are therefore open substacks and isomorphic to the stack quotients
\begin{eqnarray*}
\sH_{\phi,N,\mu}&\cong&(P_{K_0,d}\,\times_{\Spec\BQ_p}\,Q_{K,d,\mu})\;\big/\;(\Res_{K_0/\BQ_p}\GL_{d,K_0})_{E_\mu}\qquad\text{and}\\[2mm]
\sH_{\phi,\mu}&\cong&(\Res_{K_0/\BQ_p}\GL_{d,K_0}\,\times_{\Spec\BQ_p}\,Q_{K,d,\mu})\;\big/\;(\Res_{K_0/\BQ_p}\GL_{d,K_0})_{E_\mu}\,.
\end{eqnarray*}

\medskip

There is another description of $Q_{K,d,\preceq\mu}$ in terms of the affine Grassmannian. Consider the infinite dimensional affine group schemes $L^+\!\GL_d$ and $L^+\wt G$ over $\BQ_p$, and the sheaves $L\!\GL_d$ and $L\wt G$ for the fpqc-topology on $\BQ_p$ whose sections over a $\BQ_p$-algebra $R$ are given by
\begin{eqnarray*}
L^+\!\GL_d(R)&=&\GL_d(R\dbl t\dbr)\,, \\[2mm]
L^+\wt G(R)&=&\wt G(R\dbl t\dbr)\es=\es\GL_d(R\otimes_{\BQ_p}K\dbl t\dbr)\es=\es\GL_d(\BdRplus{R})\,,\\[2mm]
L\!\GL_d(R)&=&\GL_d(R\dbl t\dbr[\tfrac{1}{t}])\,, \\[2mm]
L\wt G(R)&=&\wt G(R\dbl t\dbr[\tfrac{1}{t}])\es=\es\GL_d(R\otimes_{\BQ_p}K\dbl t\dbr[\tfrac{1}{t}])\es=\es\GL_d(\BdR{R})\,.
\end{eqnarray*}
$L^+\!\GL_d$ and $L^+\wt G$ are called the \emph{group of positive loops}, and $L\!\GL_d$ and $L\wt G$ are called the \emph{loop group} of $\GL_d$, resp.\ $\wt G$. The \emph{affine Grassmannian} of $\GL_d$, resp.\ $\wt G$ is the quotient sheaf for the fppf-topology on $\BQ_p$
\begin{eqnarray*}
\Gr_{\GL_d}\es :=\es L\!\GL_d/L^+\!\GL_d\,,\qquad\text{resp.}\qquad \Gr_{\wt G}\es :=\es L\wt G/L^+\wt G\,.
\end{eqnarray*}
They are ind-schemes over $\BQ_p$ which are ind-projective; see \cite[\S4.5]{BeilinsonDrinfeld}, \cite{BeauvilleLaszlo}, \cite{LaszloSorger}, \cite{HartlViehmann1}.

We set $\Gr_{\GL_d,\norm{K}}\;:=\;\Gr_{\GL_d}\,\times_{\Spec\BQ_p}\,\Spec \norm{K}$. Then there are morphisms
\begin{eqnarray}\label{EqMapsQ}
Q_{K,d,\preceq\mu}&\longto&\Gr_{\wt G}\,\times_{\Spec\BQ_p}\,\Spec E_\mu\es=:\es \Gr_{\wt G,E_\mu}\qquad\text{and}\\[2mm]
Q_{K,d,\preceq\mu}\,\times_{\Spec E_\mu}\,\Spec \norm{K}&\longto&\prod_{\psi:K\to\norm{K}}\Gr_{\GL_d,\norm{K}}\,,\nonumber
\end{eqnarray}
which are defined as follows. Let $\Fq\subset(\BdR{Q_{K,d,\preceq\mu}})^{\oplus d}$ be the universal Hodge-Pink lattice over $Q_{K,d,\preceq\mu}$. Then by Lemma~\ref{LemmaHPLattice} there is an \'etale covering $f: \Spec R\to Q_{K,d,\preceq\mu}$ such that $f^\ast\Fq$ is free over $\BdRplus{R}$. With respect to a basis of $f^\ast\Fq$ the equality $\BdR{R}\cdot f^\ast\Fq=D\otimes_{R\otimes K_0}\BdR{R}$ corresponds to a matrix $A\in\GL_d(\BdR{R})=L\wt G(R)$. The image of $A$ in $\Gr_{\wt G}(R)$ is independent of the basis and by \'etale descend defines the first factor of the map $Q_{K,d,\preceq\mu}\to\Gr_{\wt G}\,\times_{\Spec\BQ_p}\,\Spec E_\mu$. The base change of this map along the finite \'etale morphism $\Spec \norm{K}\to\Spec E_\mu$ defines the second map in \eqref{EqMapsQ}, using the splitting $\wt G\times_{\BQ_p}\norm{K}=\prod_\psi\GL_{d,\norm{K}}$ which induces similar splittings for $L^+\wt G$, $L\wt G$, and $\Gr_{\wt G}$.

The boundedness by $\mu$ is represented by closed ind-subschemes 
\[
\Gr_{\wt G,E_\mu}^{\preceq\mu}\qquad \text{and}\qquad \Gr_{\wt G,E_\mu}^{\preceq\mu}\,\times_{\Spec E_\mu}\,\Spec \norm{K}\;=\;\prod_\psi\Gr_{\GL_d,\norm{K}}^{\preceq\mu_\psi}
\]
of $\Gr_{\wt G,E_\mu}$, resp.\ $\prod_\psi\Gr_{\GL_d,\norm{K}}$ through which the maps \eqref{EqMapsQ} factor. Conversely the universal matrix $A$ over $L\wt G$ defines a $\BdRplus{L\wt G}$-lattice $\Fq=A\cdot(\BdRplus{L\wt G})^d$. Its restriction to $\Gr_{\wt G,E_\mu}^{\preceq\mu}$ has Hodge polygon bounded by $\mu$ and corresponds to the inverses of the maps \eqref{EqMapsQ}. This yields canonical isomorphisms $Q_{K,d,\preceq\mu}\;\cong\;\Gr_{\wt G,E_\mu}^{\preceq\mu}$ and $Q_{K,d,\preceq\mu}\,\times_{\Spec E_\mu}\,\Spec \norm{K}\;\cong\;\prod_\psi\Gr_{\GL_d,\norm{K}}^{\preceq\mu_\psi}$. These isomorphisms restrict to isomorphisms of open subschemes $Q_{K,d,\mu}\;\cong\;\Gr_{\wt G,E_\mu}^{\mu}$ and $Q_{K,d,\mu}\,\times_{\Spec E_\mu}\,\Spec \norm{K}\;\cong\;\prod_\psi\Gr_{\GL_d,\norm{K}}^{\mu_\psi}$.

In view of \cite[\S4]{HartlViehmann1}, especially Lemma~4.3, the boundedness by $\mu$ on $\prod_\psi\Gr_{\GL_d,\norm{K}}^{\preceq\mu_\psi}$ can be phrased in terms of Weyl module representations of $\GL_{d,\norm{K}}$. In this formulation it was proved by Varshavsky~\cite[Proposition A.9]{Varshavsky} that $\Gr_{\GL_d,\norm{K}}^{\mu_\psi}$ is reduced. Therefore this locally closed subscheme is determined by its underlying set of points. Reasoning with the elementary divisor theorem as in Construction~\ref{ConstrHodgeWts} shows that $\Gr_{\GL_d,\norm{K}}^{\mu_\psi}$ is equal to the locally closed Schubert cell $L^+\!\GL_{d,\norm{K}}\cdot \,\mu_\psi(t)^{-1}\cdot\! L^+\!\GL_{d,\norm{K}}\,\big/\,L^+\!\GL_{d,\norm{K}}$ and is a homogeneous space under $L^+\!\GL_{d,\norm{K}}$. This description descends to $Q_{K,d,\mu}$ and shows that the latter is reduced and isomorphic to the locally closed Schubert cell $L^+\wt G_{E_\mu}\cdot \mu(t)^{-1}\cdot L^+\wt G_{E_\mu}\,\big/\,L^+\wt G_{E_\mu}$ which is a homogeneous space under $L^+\wt G_{E_\mu}:=L^+\wt G\times_{\Spec\BQ_p}\Spec E_\mu$. 

These homogeneous spaces can be described more explicitly. Set 
\begin{eqnarray*}
S_{\GL_d,\mu_\psi}&:=&L^+\!\GL_{d,\norm{K}}\es \cap\es \mu_\psi(t)^{-1}\cdot L^+\!\GL_{d,\norm{K}}\cdot \;\mu_\psi(t)\es\subset\es L^+\!\GL_{d,\norm{K}}\qquad\text{and}\\[2mm]
S_{\wt G,\mu}&:=&L^+\wt G_{E_\mu}\es \cap\es  \mu(t)^{-1}\cdot L^+\wt G_{E_\mu}\cdot \mu(t)\es\subset\es L^+\wt G_{E_\mu}\,.
\end{eqnarray*}
These are closed subgroup schemes and the homogeneous spaces are isomorphic to the quotients
\begin{eqnarray*}
L^+\!\GL_{d,\norm{K}}/S_{\GL_d,\mu_\psi}&\isoto&L^+\!\GL_{d,\norm{K}}\cdot \,\mu_\psi(t)^{-1}\cdot\! L^+\!\GL_{d,\norm{K}}\,\big/\,L^+\!\GL_{d,\norm{K}}\qquad\text{and}\\[2mm]
L^+\wt G_{E_\mu}/S_{\wt G,\mu}&\isoto&L^+\wt G_{E_\mu}\cdot \mu(t)^{-1}\cdot L^+\wt G_{E_\mu}\,\big/\,L^+\wt G_{E_\mu}\es\cong\es Q_{K,d,\mu}\,.
\end{eqnarray*}
Consider the closed normal subgroup $L^{++}\wt G_{E_\mu}(R)\;:=\;\{\,A\in L^+\wt G_{E_\mu}(R): A\equiv 1\mod t\,\}$. Then the parabolic subgroup $P_\mu$ from \eqref{EqFlag} equals
\[
P_\mu\es=\es S_{\wt G,\mu}\cdot L^{++}\wt G_{E_\mu}\,\big/\,L^{++}\wt G_{E_\mu}\es\subset\es L^+\wt G_{E_\mu}\,\big/\,L^{++}\wt G_{E_\mu} \es=\es \wt G_{E_\mu}
\]
and this yields a morphism
\begin{equation}\label{EqMorphQFlag}
Q_{K,d,\mu}\;=\; L^+\wt G_{E_\mu}/S_{\wt G,\mu}\;\onto\; L^+\wt G_{E_\mu}/S_{\wt G,\mu}\cdot L^{++}\wt G_{E_\mu}\;=\;\wt G_{E_\mu}/P_\mu\;=\; \Flag_{K,d,\mu}\,,
\end{equation}
with fibers isomorphic to $S_{\wt G,\mu}\cdot L^{++}\wt G_{E_\mu}/S_{\wt G,\mu}$. The latter is an affine space because we may consider the base change from $E_\mu$ to $\norm{K}$ and the decomposition 
\[
\bigl(S_{\wt G,\mu}\cdot L^{++}\wt G_{E_\mu}/S_{\wt G,\mu}\bigr)\times_{\Spec E_\mu}\Spec\norm{K}\;=\;\prod_\psi\bigl(S_{\GL_d,\mu_\psi}\cdot L^{++}\GL_{d,\norm{K}}/S_{\GL_d,\mu_\psi}\bigr).
\]
Each component is an affine space whose $R$-valued points are in bijection with the matrices
\[
\begin{pmatrix} 1 \\ a_{21} & 1 \\ \vdots & \ddots & \ddots \\ a_{d1} & \ldots & a_{d,d-1} & 1
\end{pmatrix}
\]
where $a_{ij}\in \bigoplus_{k=1}^{\mu_{\psi,j}-\mu_{\psi,i}-1}t^k R$. The Galois group $\Gal(\norm{K}/E_\mu)$ canonically identifies the components with the same values for $\mu_\psi$. Therefore $S_{\wt G,\mu}\cdot L^{++}\wt G_{E_\mu}/S_{\wt G,\mu}$ is an affine space.

We show that $Q_{K,d,\mu}$ is a geometric vector bundle over $\Flag_{K,d,\mu}$ by exhibiting its zero section. The projection $L^+\wt G_{E_\mu}\onto \wt G_{E_\mu}$ has a section given on $R$-valued points by the map $\wt G_{E_\mu}(R)\to L^+\wt G_{E_\mu}(R)=\wt G_{E_\mu}(R\dbl t\dbr)$ induced from the natural inclusion $R\into R\dbl t\dbr$. Since $L^+P_\mu\subset S_{\wt G,\mu}$ by definition of $P_\mu$, this section induces a section 
\[
\Flag_{K,d,\mu}\;\longto\; L^+\wt G_{E_\mu}/L^+P_\mu\;\longto\; L^+\wt G_{E_\mu}/S_{\wt G,\mu}\;=\;Q_{K,d,\mu}\,.
\]
This is the zero section of the geometric vector bundle $Q_{K,d,\mu}$ over $\Flag_{K,d,\mu}$. Using lattices the section coincides (on $L$-valued points for a field $L$) with the map $\CF^\bullet\mapsto\Fq(\CF^\bullet)$ defined in Remark~\ref{Rem2.4}(3) and the projection $Q_{K,d,\mu}\to\Flag_{K,d,\mu}$ coincides with the map $\Fq\mapsto\CF^\bullet_\Fq$ from Remark~\ref{Rem2.4}(1). Let us summarize.

\begin{proposition}\label{PropQ}
\begin{enumerate}
\item 
$Q_{K,d,\preceq\mu}$ is projective over $E_\mu$ of dimension $\sum_{\psi,j} (d+1-2j)\mu_{\psi,j}$ and contains $Q_{K,d,\mu}$ as dense open subscheme. Both schemes are irreducible.
\item 
$Q_{K,d,\mu}$ is smooth over $E_\mu$ and isomorphic to the homogeneous space $L^+\wt G_{E_\mu}/S_{\wt G,\mu}$ which is a geometric vector bundle over $\Flag_{K,d,\mu}$.
\end{enumerate}
\end{proposition}

\begin{proof}
Everything was proved above, except the formula for the dimension and the density of $Q_{K,d,\mu}$ which follow from \cite[4.5.8 and 4.5.12]{BeilinsonDrinfeld}. The irreducibility of $Q_{K,d,\preceq\mu}$ is a consequence of the density statement.
\end{proof}

\begin{theorem}\label{ThmModuliStacks}
\begin{enumerate}
\item \label{ThmModuliStacks_A}
The moduli stacks $\sD_{\phi,N,\mu}$, $\sD_{\phi,\mu}$, $\sH_{\phi,N,\preceq\mu}$, $\sH_{\phi,\preceq\mu}$, $\sH_{\phi,N,\mu}$ and $\sH_{\phi,\mu}$ are noetherian Artin stacks of finite type over $E_\mu$.
\item \label{ThmModuliStacks_B}
The stack $\sH_{\phi,N,\mu}$ is a dense open substack of $\sH_{\phi,N,\preceq\mu}$ and projects onto $\sD_{\phi,N,\mu}$. The morphism $\sH_{\phi,N,\mu}\to\sD_{\phi,N,\mu}$ has a section and is relatively representable by a vector bundle.
\item \label{ThmModuliStacks_C}
The stack $\sH_{\phi,\mu}$ is a dense open substack of $\sH_{\phi,\preceq\mu}$ and projects onto $\sD_{\phi,\mu}$. The morphism $\sH_{\phi,\mu}\to\sD_{\phi,\mu}$ has a section and is relatively representable by a vector bundle.
\item \label{ThmModuliStacks_D}
The stacks $\sH_{\phi,\preceq\mu}$, $\sH_{\phi,\mu}$ are irreducible of dimension $\sum_{\psi,j} (d+1-2j)\mu_{\psi,j}$, and $\sD_{\phi,\mu}$ is irreducible of dimension $\sum_\psi\#\{(i,j):\mu_{\psi,i}>\mu_{\psi,j}\}$. The stacks $\sH_{\phi,\mu}$ and $\sD_{\phi,\mu}$ are smooth over $E_\mu$.
\item \label{ThmModuliStacks_E}
The stacks $\sH_{\phi, N, \preceq\mu}$, $\sH_{\phi, N,\mu}$ are equi-dimensional of dimension $\sum_{\psi,j} (d+1-2j)\mu_{\psi,j}$, and $\sD_{\phi,N, \mu}$ is equi-dimensional of dimension $\sum_\psi\#\{(i,j):\mu_{\psi,i}>\mu_{\psi,j}\}$. The stacks $\sH_{\phi,N,\mu}$ and $\sD_{\phi,N,\mu}$ are reduced, Cohen-Macaulay and generically smooth over $E_\mu$. The irreducible components of $\sH_{\phi, N, \preceq\mu}$, $\sH_{\phi, N,\mu}$ and $\sD_{\phi,N, \mu}$ are indexed by the possible Jordan types of the nilpotent endomorphism $N$. 
\end{enumerate}
\end{theorem}

\begin{proof}
\ref{ThmModuliStacks_A} The stacks are quotients of noetherian schemes of finite type over $E_\mu$ by the action of the smooth group scheme $(\Res_{K_0/\BQ_p}\GL_{d,K_0})_{E_\mu}$ and hence are noetherian Artin stacks of finite type by \cite[4.6.1, 4.7.1, 4.14]{LaumonMB}. 

\smallskip\noindent
\ref{ThmModuliStacks_B} and \ref{ThmModuliStacks_C} follow from the corresponding statements for $Q_{K,d,\mu}$ in Proposition~\ref{PropQ}.

\smallskip\noindent
\ref{ThmModuliStacks_D} The covering spaces 
\begin{align*}
\Res_{K_0/\BQ_p}\GL_{d,K_0}\,&\times_{\Spec\BQ_p}\,Q_{K,d,\preceq\mu},\ \text {resp.}\\ 
\Res_{K_0/\BQ_p}\GL_{d,K_0}\,&\times_{\Spec\BQ_p}\,Q_{K,d,\mu},\  \text{resp.}\\ 
\Res_{K_0/\BQ_p}\GL_{d,K_0}\,&\times_{\Spec\BQ_p}\,\Flag_{K,d,\mu}
\end{align*}
 of these stacks are irreducible because $\Res_{K_0/\BQ_p}\GL_{d,K_0}$ is geometrically irreducible. This implies the irreducibility of the stacks. The formulas for the dimension follow from \cite[pp.~98f]{LaumonMB} and Proposition~\ref{PropQ}, respectively the well known dimension formula for partial flag varieties. The smoothness follows from the smoothness of $\Res_{K_0/\BQ_p}\GL_{d,K_0}\,\times_{\Spec\BQ_p}\,Q_{K,d,\mu}$, resp.\ $\Res_{K_0/\BQ_p}\GL_{d,K_0}\,\times_{\Spec\BQ_p}\,\Flag_{K,d,\mu}$ by \cite[4.14]{LaumonMB}.
 
\smallskip\noindent
\ref{ThmModuliStacks_E}
As in \ref{ThmModuliStacks_D} these results are direct consequences of the corresponding results on the covering spaces, which follow from Theorem~\ref{ThmReduced}. We only need to convince ourselves that the action of $(\Res_{K_0/\BQ_p}\GL_{d,K_0})_{E_\mu}$ does not identify irreducible components of $P_{K_0,d}$. However this follows from the fact that the Jordan canonical forms of the nilpotent endomorphism $N$ at two distinct generic points $y_1$ and $y_2$ of $P_{K_0,d}$ are distinct by the description in Theorem \ref{ThmReduced}. 
\end{proof}

\begin{remark}
These stacks are not separated. Namely, let $\ulD,\ulD'$ be two $(\phi,N)$-modules with Hodge-Pink lattice (respectively two $K$-filtered $(\phi,N)$-modules) over $R$. Then $\Isom(\ulD,\ulD')$ is representable by an algebraic space, separated and of finite type over $R$; see \cite[Lemme~4.2]{LaumonMB}. The above stacks are separated over $E_\mu$ if and only if all these algebraic spaces $\Isom(\ulD,\ulD')$ are proper. This is not the case in general. For example let $R$ be a discrete valuation ring with fraction field $L$, let $D=D'=R\otimes_{\BQ_p}K_0^d$ with $\Phi=\id$ and $N=0$. Then every element $f\in L$ is an automorphism of $D\otimes_RL$, compatible with $\Phi$ and $N$. However, it extends to an automorphism of $D$ only if $f\in R\mal$.
\end{remark}

\section{Vector bundles on the open unit disc} \label{VBonopenunitdisc}

In \cite{crysrep} Kisin related $K$-filtered $(\phi,N)$-modules over $\BQ_p$ to vector bundles on the open unit disc. This was generalized in \cite[\S5]{families} to families of $K$-filtered $\phi$-modules with Hodge-Tate weights $0$ and $-1$. In this section we generalize it to arbitrary families of $(\phi,N)$-modules with Hodge-Pink lattice. For this purpose we work in the category $\Ad_{\BQ_p}^\lft$ of adic spaces locally of finite type over $\Spa(\BQ_p,\BZ_p)$; see \cite{contval,Hu2,Huber} and \cite[\S2.2]{families}. Since the stacks $\sD_{\phi,\mu}$, $\sD_{\phi,N,\mu}$, $\sH_{\phi,\preceq\mu}$, $\sH_{\phi,N,\preceq\mu}$, $\sH_{\phi,\mu}$ and $\sH_{\phi,N,\mu}$ are quotients of quasi-projective schemes over $E_\mu$ they give rise to stacks on $\Ad_{E_\mu}^\lft$ which we denote by $\sH_{\phi,N,\mu}^\ad$, etc.

For $0\leq r<1$ we write $\boldB_{[0,r]}$ for the closed disc of radius $r$ over $K_0$ in the category of adic spaces and denote by 
\[\Ubb=\lim\limits_{\substack{\longrightarrow \\ r\rightarrow 1}} \boldB_{[0,r]}\]
the open unit disc. This is an open subspace of the closed unit disc (which is \emph{not} identified with the set of all points $x$ in the closed unit disc with $|x|<1$ in the adic setting). 
In the following we will always write $u$ for the coordinate function on $\boldB_{[0,r]}$ and $\Ubb$, i.e.\ we view 
\[\bB^{[0,r]}:=\Gamma(\boldB_{[0,r]},\Ocal_{\boldB_{[0,r]}})\]
 and $\bB^{[0,1)}:=\Gamma(\Ubb,\Ocal_{\Ubb})$ as sub-rings of $K_0\dbl u\dbr$.

Let $X\in\Ad_{\BQ_p}^\lft$ be an adic space over $\Q_p$ and write
\begin{align*}
\sB_X^{[0,r]}&=\Ocal_X\widehat{\otimes}_{\Q_p}\bB^{[0,r]}=\pr_{X,\ast}\Ocal_{X\times \boldB_{[0,r]}} \\
\sB_X^{[0,1)}&=\Ocal_X\widehat{\otimes}_{\Q_p}\bB^{[0,1)}=\pr_{X,\ast}\Ocal_{X\times \Ubb}
\end{align*}
for the sheafified versions of the rings $\bB^{[0,r]}$ and $\bB^{[0,1)}$ where $\pr_X$ is the projection onto $X$. These are sheaves of topological $\Ocal_X$-algebras on $X$.

We introduce the function
\begin{equation}\label{lambda}
\lambda\;:=\;\prod_{n\geq 0}\phi^n(E(u)/E(0))\;\in\;\bB^{[0,1)}\,.
\end{equation}
and the differential operator $N_\nabla:=-u\lambda\frac{d}{du}:\bB_{[0,1)}\to\bB_{[0,1)}$. 
For any adic space $X\in\Ad_{\BQ_p}^\lft$ we view $\lambda$ as a section of $\sB_X^{[0,1)}$ and $N_\nabla$ as a differential operator on $\sB_X^{[0,1)}$. The Frobenius $\phi$ on $\CO_X\otimes_{\BQ_p}K_0[u]$ extends to a Frobenius endomorphism of $\sB_X^{[0,1)}$ again denoted by $\phi$ by means of $\phi(u)=u^p$. These operators satisfy the relation
\begin{equation}\label{EqNnabla}
N_\nabla\,\phi=p\,\tfrac{E(u)}{E(0)}\cdot\phi\, N_\nabla\,.
\end{equation}

\begin{definition}\label{DefPhiModOnU}
A \emph{$(\phi,N_\nabla)$-module} $(\CM,\Phi_\CM,N_\nabla^\CM)$ over an adic space $X\in\Ad_{\BQ_p}^\lft$ consists of a locally free sheaf $\CM$ of finite rank on $X\times_{\BQ_p}\BU$, a differential operator $N_\nabla^\CM:\CM\to\CM[\tfrac{1}{\lambda}]$ over $N_\nabla$, that is $N_\nabla^\CM(fm)=-u\lambda\frac{df}{du}\cdot m+f\cdot N_\nabla^\CM(m)$ for all sections $f$ of $\CO_{X\times_{\BQ_p}\BU}$ and $m$ of $\CM$, and an $\CO_{X\times_{\BQ_p}\BU}$-linear isomorphism $\Phi_\CM:(\phi^\ast\CM)[\tfrac{1}{E(u)}]\isoto\CM[\tfrac{1}{E(u)}]$, satisfying $N_\nabla^\CM\circ\Phi_\CM\circ\phi=p\,\tfrac{E(u)}{E(0)}\cdot\Phi_\CM\circ\phi\circ N_\nabla^\CM$.

A \emph{morphism} $\alpha :(\CM,\Phi_\CM,N_\nabla^\CM)\to(\CN,\Phi_\CN,N_\nabla^\CN)$ between $(\phi,N_\nabla)$-modules over $X$ is a morphism $\alpha :\CM\to\CN$ of sheaves satisfying $\alpha \circ\Phi_\CM=\Phi_\CN\circ\phi^\ast(\alpha )$ and $N_\nabla^\CN\circ \alpha =\alpha \circ N_\nabla^\CM$.
\end{definition}

\begin{rem}\label{RemConnection}
(1) 
Note that it is not clear whether a $(\phi,N_\nabla)$-module $\Mcal$ is locally on $X$ free over $X\times \Ubb$ and hence it is not clear whether $\pr_{X,\ast}\Mcal$ is locally on $X$ a free $\sB_X^{[0,1)}$-module. However it follows from \cite[Proposition~2.1.15]{KPX} that $\pr_{X,\ast}\Mcal$ is a finitely presented $\sB_X^{[0,1)}$-module.

\smallskip\noindent
(2) The differential operator $N_\nabla^\CM$ can be equivalently described as a connection $\nabla_\CM:\CM\to\CM\otimes u^{-1}\Omega^1_{X\times\BU/X}[\tfrac{1}{\lambda}]$ when we set $\nabla_{\!\SSC\CM}(m):=-\tfrac{1}{\lambda}N_\nabla^\CM(m)\otimes\tfrac{du}{u}$. Then $N_\nabla^\CM$ is recovered as the composition of $\nabla_{\!\SSC\CM}$ followed by the map $u^{-1}\Omega^1_{X\times\BU/X}[\tfrac{1}{\lambda}]\to\CM,\,du\mapsto-u\lambda$.
\end{rem}

Let $X\in\Ad_{\BQ_p}^\lft$ be an adic space. We will show that the category of $(\phi,N_\nabla)$-modules over $X$ is equivalent to the category of $(\phi,N)$-modules with Hodge-Pink lattice over $X$ by defining two mutually quasi-inverse functors $\ul\CM$ and $\ulD$.

To define $\ul\CM$ let $\ulD=(D,\Phi,N,\Fq)$ be a $(\phi,N)$-module with Hodge-Pink lattice over $X$. We denote by $\pr:X\times_{\BQ_p}\Ubb\rightarrow X\times_{\Q_p}K_0$ the projection and set $(\CalD,\Phi_\CalD):=\pr^\ast(D,\Phi)$. Then 
\[
\pr_{X,\ast}(\CalD,\Phi_\CalD)=(D,\Phi)\otimes_{(\CO_X\otimes K_0)}\sB_X^{[0,1)}.
\] 
We choose a $\BdRplus{\CO_X}$-automorphism $\eta_\ulD$ of $\Fp:=D\otimes_{\CO_X\otimes K_0}\BdRplus{\CO_X}$ and we let $\iota_0:\CalD\into D\otimes_{\CO_X\otimes K_0}\BdRplus{\CO_X}$ be the embedding obtained as the composition of the natural inclusion $D\otimes_{\CO_X\otimes K_0}\sB^{[0,1)}_X\into D\otimes_{\CO_X\otimes K_0}\BdRplus{\CO_X}$ composed with the automorphism $\eta_\ulD$. Here we follow Kisin~\cite[\S\,1.2]{crysrep} and choose
\begin{eqnarray}\label{EqTheta}
\eta_\ulD:\; D\otimes_{\CO_X\otimes K_0}\BdRplus{\CO_X} & \isoto & D\otimes_{\CO_X\otimes K_0}\BdRplus{\CO_X}\,,\nonumber \\
 d_0\otimes f & \longmapsto & \TS\sum_iN^i(d_0)\otimes  \tfrac{(-1)^i}{i!}\log\bigl(1-\tfrac{E(u)}{E(0)}\bigr)^i\cdot f\,.
\end{eqnarray}

\begin{remark}\label{RemOtherTheta1}
(1) Actually, Kisin introduces a formal variable $\ell_u$ over $\sB_X^{[0,1)}$ which formally acts like $\log u$. He extends $\phi$ to $\sB_X^{[0,1)}[\ell_u]$ via $\phi(\ell_u)=p\,\ell_u$, extends $N_\nabla$ to a derivation on $\sB_X^{[0,1)}[\ell_u]$ via $N_\nabla(\ell_u)=-\lambda$, and defines $N$ as the $\sB_X^{[0,1)}$-linear derivation on $\sB_X^{[0,1)}[\ell_u]$ that acts as the differentiation of the formal variable $\ell_u$. Under the $\Phi$-equivariant identification 
\begin{equation*}\label{EqDl_uROT}
\xymatrix @R=0pc {
D[\ell_u]^{N=0} \es:=\es \bigl\{\,\sum_{i=0}^{<\infty}d_i\ell_u^i:\es d_i\in D\text{ with }N(\TS\sum_id_i\ell_u^i)=0\,\bigr\} & D \ar[l]_{\qquad\qquad\qquad\qquad\qquad\qquad\quad\cong}\\
\qquad\qquad\sum_{i=0}^{<\infty}\tfrac{(-1)^i}{i!}N^i(d_0)\ell_u^i & d_0 \ar@{|->}[l]
}
\end{equation*}
Kisin's map $\iota_0:D[\ell_u]^{N=0}\otimes_{\CO_X\otimes K_0}\sB_X^{[0,1)}\into\Fp\,,\es\sum_i d_i\ell_u^i\otimes f \mapsto \sum_i d_i \otimes f\cdot(\log\tfrac{u}{\pi})^i$ corresponds to our $\iota_0$, because we identify $\tfrac{E(u)}{E(0)}$ with $1-\tfrac{u}{\pi}$.

\medskip\noindent
(2) Instead of the above $\eta_\ulD$ one could also choose $\eta_\ulD=\id_\Fp$. This would lead to a few changes which we will comment on in Remark~\ref{RemOtherTheta}. Note that our $\eta_\ulD$ from \eqref{EqTheta} is different from $\id_\Fp$ if $N\ne0$.
\end{remark}

\medskip

For all $n\ge0$ we now consider the map
\[
\pr_{X,\ast}\CalD[\tfrac{1}{\lambda}] \xrightarrow{\;\Phi_\CalD^{-j}}  \pr_{X,\ast}\phi^{j\ast}(\CalD[\tfrac{1}{\lambda}]) =\!=\!= \pr_{X,\ast}\CalD[\tfrac{1}{\lambda}] \otimes_{\sB_X^{[0,1)},\,\phi^j}\sB_X^{[0,1)} \xrightarrow{\;\phi^{j*}\iota_0\,} \Fp[\tfrac{1}{E(u)}]\otimes_{\BdRplus{\CO_X}}\phi^j(\BdRplus{\CO_X})
\]
where we write $\phi^{j*}\iota_0$ for $\iota_0\otimes\id$. We set
\begin{equation}\label{EqCM}
\pr_{X,\ast}\CM\;:=\;\bigl\{\,m\in\pr_{X,\ast}\CalD[\tfrac{1}{\lambda}]:\es \phi^{j*}\iota_0\circ\Phi_\CalD^{-j}(m)\in\Fq\otimes_{\BdRplus{\CO_X}}\phi^j(\BdRplus{\CO_X})\text{ for all }j\ge0\,\bigr\}\,.
\end{equation}
and we let $\CM$ be the induced sheaf on $X\times_{\BQ_p}\BU$. Since $\lambda=\tfrac{E(u)}{E(0)}\phi(\lambda)$ the isomorphism $\Phi_\CalD$ induces an isomorphism $\Phi_\CM:(\phi^\ast\CM)[\tfrac{1}{E(u)}]\isoto\CM[\tfrac{1}{E(u)}]$. 

We want to show that $\CalD$ and $\CM$ are locally free sheaves of finite rank on $X\times_{\BQ_p}\BU$. For $\CalD$ this follows from $\CalD|_{X\times\BB_{[0,r]}}=D\otimes_{(\CO_X\otimes K_0)}\CO_{X\times\BB_{[0,r]}}$. We work on a covering of $X$ by affinoids $Y=\Spa(A,A^+)$. Let $h\in\BZ$ be such that $\Fq\subset E(u)^{-h}\Fp$ on $Y$ and let $n$ be maximal such that $\phi^n(E(u))$ is not a unit in $\bB_{[0,r]}$, that is, such that $r^{p^n}\ge|\pi|$. Then $\CM|_{Y\times\BB_{[0,r]}}$ is defined by the exact sequence
\[
0 \longto \CM|_{Y\times\BB_{[0,r]}} \longto \lambda^{-h}\CalD|_{Y\times\BB_{[0,r]}} \xrightarrow{\;\bigoplus_{j=0}^n\,\phi^{j*}\iota_0\circ\Phi_\CalD^{-j}\,} \DS\bigoplus_{j=0}^n (E(u)^{-h}\Fp/\Fq)\otimes_{\BdRplus{A},\,\phi^j}\phi^j(\BdRplus{A}) \longto 0\,.
\]
The $A\otimes_{\BQ_p}K_0[u]$-module $E(u)^{-h}\Fp/\Fq$ is locally free over $A$, say of rank $k$. The endomorphism $\phi:K_0[u]\to K_0[u]$ makes the target $K_0[u]$ into a free module of rank $p$ over the source $K_0[u]$. Therefore $(E(u)^{-h}\Fp/\Fq)\otimes_{\BdRplus{A},\,\phi^j}\phi^j(\BdRplus{A})$ is locally free over $A$ of rank $p^jk$. Since the affinoid algebra $A$ is noetherian and $\bB_{[0,r]}$ is a principal ideal domain by \cite[Corollary of Proposition~4]{Lazard} also $\Gamma(Y\times\BB_{[0,r]},\CO_{Y\times\BB_{[0,r]}})\,=\,A\wh\otimes_{\BQ_p}\bB_{[0,r]}$ is noetherian. So $\Gamma(Y\times\BB_{[0,r]},\CM)$ is finitely generated over $A\wh\otimes_{\BQ_p}\bB_{[0,r]}$ and flat over $A$. The residue field of each maximal ideal $\Fm\subset A\wh\otimes_{\BQ_p}\bB_{[0,r]}$ is finite over $\BQ_p$ by \cite[Corollary 6.1.2/3]{BoschGR}. Therefore $\Fn=\Fm\cap A$ is a maximal ideal of $A$. By the elementary divisor theorem $A/\Fn\otimes_A\Gamma(Y\times\BB_{[0,r]},\CM)$ is free over the product of principal ideal domains $A/\Fn\otimes_{\BQ_p}\bB_{[0,r]}$. Therefore $\Gamma(Y\times\BB_{[0,r]},\CM)$ is locally free of rank $d$ over $A\wh\otimes_{\BQ_p}\bB_{[0,r]}$ by \cite[IV$_3$, Theorem 11.3.10]{EGA}. This shows that $\CM$ is a locally free sheaf of rank $d$ on $X\times_{\BQ_p}\BU$. 

We equip $\CM$ with a differential operator $N_\nabla^\CM$ over $N_\nabla$. On $\lambda^{-h}\CalD\;=\; D\otimes_{(\CO_X\otimes K_0)}\lambda^{-h}\sB_X^{[0,1)}$ we have the differential operator $N_\nabla^\CalD:=N\otimes\lambda+\id_D\otimes N_\nabla$
\begin{equation}\label{EqDNnabla}
\xymatrix @R=0pc {
\lambda^{-h}\CalD \ar[r]^{N\otimes\lambda+\id_D\otimes N_\nabla} & \lambda^{-h}\CalD\\
d\otimes \lambda^{-h}f \ar@{|->}[r] & N(d)\otimes\lambda^{1-h}f+d\otimes(hu\lambda^{-h}f\tfrac{d\lambda}{du}-u\lambda^{1-h}\tfrac{df}{du})
}
\end{equation}
with $d\in D$ and $f\in\sB_X^{[0,1)}$. Its image lies in $\lambda^{-h}\CalD$. If $E(u)^n\Fp\subset\Fq\subset E(u)^{-h}\Fp$ then $\lambda^n\CalD\subset\CM\subset\lambda^{-h}\CalD$. Thus $N_\nabla^\CalD(\CM)\subset\lambda^{-h}\CalD\subset\lambda^{-h-n}\CM$ and we let $N_\nabla^\CM$ be the restriction of $N_\nabla^\CalD$ to $\CM$. The equation $N_\nabla^\CM\circ\Phi_\CM\circ\phi=p\,\tfrac{E(u)}{E(0)}\cdot\Phi_\CM\circ\phi\circ N_\nabla^\CM$ is satisfied because it is satisfied on $\CalD$ by \eqref{EqNnabla}. Therefore we have constructed a $(\phi,N_\nabla)$-module $\ul\CM(D,\Phi,N,\Fq)\,:=\,(\CM,\Phi_\CM,N_\nabla^\CM)$ over $X$. Note that in terms of Kisin's description of $\CalD\cong D[\ell_u]^{N=0}\otimes_{(\CO_X\otimes K_0)}\sB_X^{[0,1)}$ the differential operator $N_\nabla^\CM$ is given as $\id_D\otimes N_\nabla$.

\begin{example}\label{ExCyclot2}
The $(\phi,N)$-modules with Hodge-Pink lattice from Example~\ref{ExCyclot1}, corresponding to the cyclotomic character, give rise to the following $(\phi,N_\nabla)$-modules of rank $1$ over $X=\Spa(\BQ_p,\BZ_p)$. For $\ulD=(K_0,\Phi=p^{-1},N=0,\Fq=E(u)\Fp)$ we obtain $\CalD=(\sB_X^{[0,1)},\Phi_\CalD=p^{-1},N_\nabla)$ and $\CM=\lambda\sB_X^{[0,1)}$. On the basis vector $\lambda$ of $\CM$ the actions of $\Phi_\CalD$ and $N_\nabla$ are given by $\Phi_\CalD(\phi(\lambda))=p^{-1}\phi(\lambda)=\tfrac{E(0)}{pE(u)}\,\lambda$ and $N_\nabla(\lambda)=-u\tfrac{d\lambda}{du}\,\lambda$. So we find $\ul\CM(\ulD)\cong(\sB_X^{[0,1)},\Phi_\CM=\tfrac{E(0)}{pE(u)},N_\nabla^\CM)$ with $N_\nabla^\CM(f)=N_\nabla(f)-u\tfrac{d\lambda}{du}\,f$. Similarly for $\ulD=(K_0,\Phi=p,N=0,\Fq=E(u)^{-1}\Fp)$ we obtain $\CalD=(\sB_X^{[0,1)},\Phi_\CalD=p,N_\nabla)$ and $\CM=\lambda^{-1}\sB_X^{[0,1)}$ which leads to $\ul\CM(\ulD)\cong(\sB_X^{[0,1)},\Phi_\CM=\tfrac{pE(u)}{E(0)},N_\nabla^\CM)$ with $N_\nabla^\CM(f)=N_\nabla(f)+u\tfrac{d\lambda}{du}\,f$.
\end{example}

\bigskip

To define the quasi-inverse functor $\ulD$ let $(\CM,\Phi_\CM,N_\nabla^\CM)$ be a $(\phi,N_\nabla)$-module over $X$. We denote by $e:X\times_{\Q_p}K_0\rightarrow X\times_{\BQ_p}\Ubb$ the isomorphism $x\mapsto (x,0)$ onto the closed subspace defined by $u=0$. Let $(D,\Phi,N):=e^\ast(\CM,\Phi_\CM,N_\nabla^\CM)$. It is a $(\phi,N)$-module over $X$ because $N$ is clearly $\CO_X\otimes_{\BQ_p}K_0$-linear and $e^\ast\bigl(\tfrac{E(u)}{E(0)}\bigr)=1$ implies $N\circ\Phi=p\cdot \Phi\circ\phi^\ast N$. By \cite[Proposition 5.2]{phimod} there is a unique $\CO_{X\times_{\BQ_p}\BU}$-linear isomorphism
\begin{equation}\label{EqXi}
\xi:\pr^\ast\!D[\tfrac{1}{\lambda}] \isoto \CM[\tfrac{1}{\lambda}]
\end{equation}
satisfying $\xi\circ \pr^\ast\Phi=\Phi_\CM\circ\phi^\ast\xi$ and $e^\ast\xi=\id_D$. In particular the composition $\pr^\ast\Phi\circ(\phi^\ast\xi)^{-1}=\xi^{-1}\circ\Phi_\CM$ induces an isomorphism $\phi^\ast\CM\otimes\BdRplus{\CO_X}\isoto D\otimes_{(\CO_X\otimes K_0)}\BdRplus{\CO_X}=\Fp$ of $\BdRplus{\CO_X}$-modules. We set $\Fq:=\eta_\ulD\circ(\xi\otimes\id_{\BdR{\CO_X}})^{-1}(\CM\otimes\BdRplus{\CO_X})$. Then $\ulD(\CM,\Phi_\CM,N_\nabla^\CM)\,:=\,(D,\Phi,N,\Fq)$ is a $(\phi,N)$-module with Hodge-Pink lattice over $X$ by Lemma~\ref{LemmaHPLattice} and the following lemma.

\begin{lemma}\label{LemmaCokerPhiM}
Locally on a covering of $X$ by affinoids $Y=\Spa(A,A^+)$ there exist integers $h,n$ with $E(u)^n\Phi_\CM(\phi^\ast\CM)\subset\CM\subset E(u)^{-h}\Phi_\CM(\phi^\ast\CM)$ such that the quotients \[E(u)^{-h}\Phi_\CM(\phi^\ast\CM)/\CM\ \text{and}\  \CM/ E(u)^n\Phi_\CM(\phi^\ast\CM)\] are finite locally free over $A$.
\end{lemma}

\begin{proof}
We may assume that $X=Y=\Spa(A,A^+)$ is affinoid. Then the existence of $h$ and $n$ follows from the finiteness of $\CM$ and $\phi^\ast\CM$. Let $\Fm\subset A$ be a maximal ideal and set $L=A/\Fm$. 
Let $|\pi|<r<1$ and set $\wt\CM:=\Gamma(Y\times_{\BQ_p}\BB_{[0,r]},\CM)$ and $\wt{\phi^\ast\CM}:=\Gamma(Y\times_{\BQ_p}\BB_{[0,r]},\phi^\ast\CM)$. Then $\CM/ E(u)^n\Phi_\CM(\phi^\ast\CM)\cong\wt\CM/ E(u)^n\Phi_\CM(\wt{\phi^\ast\CM})$. 
Consider the exact sequence
\[
\xymatrix {0 \ar[r] & E(u)^n\wt{\phi^\ast\CM} \ar[r]^{\qquad\Phi_\CM} & \wt\CM \ar[r] & \CM/ E(u)^n\Phi_\CM(\phi^\ast\CM) \ar[r] & 0
}
\]
in which the first map is injective because $E(u)$ is a non-zero-divisor in $A\wh\otimes_{\BQ_p}\bB_{[0,r]}$. We tensor the sequence with $L$ over $A$ to obtain the exact sequence of $L\otimes_{\BQ_p}\bB_{[0,r]}$-modules
\[0\longrightarrow T\longrightarrow L\otimes_A E(u)^n\wt{\phi^\ast\CM} \xrightarrow{\id_L\otimes\Phi_\CM} L\otimes_A\wt\CM \longrightarrow L\otimes_A\bigl(\CM/ E(u)^n\Phi_\CM(\phi^\ast\CM)\bigr)\longrightarrow 0
\]
with $T= \Tor_1^A\bigl(L,\,\CM/ E(u)^n\Phi_\CM(\phi^\ast\CM)\bigr)$. Since $L\otimes_{\BQ_p}K_0$ is a product of fields, $L\otimes_{\BQ_p}\bB_{[0,r]}$ is a product of principal ideal domains by \cite[Corollary of Proposition~4]{Lazard}. Since $E(u)^{n+h}$ annihilates $L\otimes_A\CM/ E(u)^n\Phi_\CM(\phi^\ast\CM)$ the latter is a torsion module over $L\otimes_{\BQ_p}\bB_{[0,r]}$. It follows that $\id_L\otimes\Phi_\CM$ is a morphism of free modules of the same rank over a product of principle ideal domains whose cokernel is a torsion module. It is a direct consequence of the classification of finitely generated modules over a principle ideal domain that the map $\id_L\otimes\Phi_\CM$ then has to be injective. It follows that 
\[0=T= \Tor_1^A\bigl(L,\,\CM/ E(u)^n\Phi_\CM(\phi^\ast\CM)\bigr)= \Tor_1^{A_\mfrak}\bigl((A_\mfrak/\mfrak A_\mfrak),\,(\CM/ E(u)^n\Phi_\CM(\phi^\ast\CM))_\mfrak\bigr).\]
Since $(\CM/ E(u)^n\Phi_\CM(\phi^\ast\CM))_\mfrak$ is finite over the noetherian local ring $A_\mfrak$ it is locally free by the local criterion of flatness \cite[Theorem~6.8]{Eisenbud}. It follows that $\CM/ E(u)^n\Phi_\CM(\phi^\ast\CM)$ is locally free as an $A$-module.
Finally, the two last objects in the short exact sequence
\[0\longrightarrow E(u)^{-h}\Phi_\CM(\phi^\ast\CM)/\CM \xrightarrow{\cdot E(u)^{n+h}\es}\CM/E(u)^{n+h}\CM\longrightarrow \CM/ E(u)^n\Phi_\CM(\phi^\ast\CM)\longrightarrow 0 \]
are flat and hence so is the first (all its higher $\Tor$-terms have to vanish). As $E(u)^{-h}\Phi_\CM(\phi^\ast\CM)/\CM$ is also finite as an $A$-module it follows that it is finite and locally free over $A$.
\end{proof}

\begin{theorem}\label{ThmEquivDandM}
For every adic space $X\in\Ad_{\BQ_p}^\lft$ the functors $\ul\CM$ and $\ulD$ constructed above are mutually quasi-inverse equivalences between the category of $(\phi,N)$-modules with Hodge-Pink lattice over $X$ and the category of $(\phi,N_\nabla)$-modules over $X$.
\end{theorem}

\begin{proof}
We must show that the functors are mutually quasi-inverse. To prove one direction let $(D,\Phi,N,\Fq)$ be a $(\phi,N)$-module with Hodge-Pink lattice over $X$ and let $(\CM,\Phi_\CM,N_\nabla^\CM)=\ul\CM(D,\Phi,N,\Fq)$. By construction $e^\ast\CM=D$, and under this equality $e^\ast\Phi_\CM$ corresponds to $\Phi$. Since $e^\ast\lambda=1$, formula \eqref{EqDNnabla} shows that $e^\ast N_\nabla^\CM$ corresponds to $N$ on $D$. By the uniqueness of the map $\xi$ from \eqref{EqXi}, its inverse $\xi^{-1}$ equals the inclusion $\CM\into\CalD[\tfrac{1}{\lambda}]$, by which we defined $\CM$. This shows that $\eta_\ulD\circ(\xi\otimes\id_{\BdR{\CO_X}})^{-1}(\CM\otimes\BdRplus{\CO_X})$ equals $\Fq$ and that $\ulD\circ\ul\CM=\id$.

Conversely let $(\CM,\Phi_\CM,N_\nabla^\CM)$ be a $(\phi,N_\nabla)$-module over $X$ and let $(D,\Phi,N,\Fq)=\ulD(\CM,\Phi_\CM,N_\nabla^\CM)$. Via the isomorphism $\xi$ from \eqref{EqXi}, $\CM$ is a $\phi$-submodule of $\pr^\ast\!D[\tfrac{1}{\lambda}]$. By construction of $\Fq$ and 
\[
\underline{\CM}(\ulD(\CM,\Phi_\CM,N_\nabla^\CM))\subset \pr^\ast\!D[\tfrac{1}{\lambda}],
\] 
the latter submodule coincides with $\CM$ modulo all powers of $E(u)$. Since both submodules have a Frobenius which is an isomorphism outside $\Var(E(u))$ they are equal on all of $X\times_{\BQ_p}\BU$. It remains to show that $N_\nabla^\CM$ is compatible with $N_\nabla^{\pr^\ast\!D}$ under the isomorphism $\xi:\pr^\ast\!D[\tfrac{1}{\lambda}]\isoto\CM[\tfrac{1}{\lambda}]$. We follow \cite[Lemma 1.2.12(3)]{crysrep} and let $\sigma:=\xi\circ N_\nabla^{\pr^\ast\!D}-N_\nabla^\CM\circ\xi$. Then $\sigma:\pr^\ast\!D[\tfrac{1}{\lambda}]\to\CM[\tfrac{1}{\lambda}]$ is $\CO_{X\times_{\BQ_p}\BU}$-linear and it suffices to show that $\sigma(D)=0$. By \eqref{EqDNnabla} both $N_\nabla^{\pr^\ast\!D}$ and $N_\nabla^\CM$ reduce to $N$ modulo $u$. Therefore $\sigma(D)\subset u\CM[\tfrac{1}{\lambda}]$. One checks that $\sigma\circ\Phi_{\pr^\ast\!D}\circ\phi=p\tfrac{E(u)}{E(0)}\cdot\Phi_\CM\circ\phi\circ\sigma$ and this implies 
\[
\sigma(D)\;=\;\sigma\circ\Phi_{\pr^\ast\!\!D}(\phi^\ast D)\;=\;p\tfrac{E(u)}{E(0)}\cdot\Phi_\CM\circ\phi^\ast(u\CM[\tfrac{1}{\lambda}])\;\subset\;u^p\CM[\tfrac{1}{\lambda}]\,.
\]
By induction $\sigma(D)\subset u^{p^i}\CM[\tfrac{1}{\lambda}]$ for all $i$ and hence $\sigma(D)=0$. This shows that also $\ul\CM\circ\ulD$ is isomorphic to the identity and proves the theorem.
\end{proof}

\begin{corollary}\label{CorEquivDandM}
The stack $\sH_{\phi,N,\preceq\mu}^\ad$ is isomorphic to the stack whose groupoid of $X$-valued points for $X\in\Ad_{E_\mu}^\lft$ consists of $(\phi,N_\nabla)$-modules $(\CM,\Phi_\CM,N_\nabla^\CM)$ over $X$ satisfying 
\[
\bigwedge^j_{\CO_{X\times\BU}}\CM\;\subset\; E(u)^{-\mu_1-\ldots-\mu_j}\cdot\bigwedge^j_{\CO_{X\times\BU}}\Phi_\CM(\phi^\ast\CM)
\]
with equality for $j=\rk\CM$. Here $\mu_i$ is viewed as a $\BZ$-valued function on $X\times_{\BQ_p}K$.
\end{corollary}

\begin{proof}
This follows from the definition of the functor $\ulD$, in particular the definition of the Hodge-Pink lattice.
\end{proof}

\begin{definition}\label{DefHnabla}
We define substacks $\sH_{\phi,N,\mu}^\nabla\subset\sH_{\phi,N,\mu}^\ad$, resp.\ $\sH_{\phi,N,\preceq\mu}^\nabla\subset\sH_{\phi,N,\preceq\mu}^\ad$, resp.\ $\sH_{\phi,\mu}^\nabla\subset\sH_{\phi,\mu}^\ad$, resp.\ $\sH_{\phi,\preceq\mu}^\nabla\subset\sH_{\phi,\preceq\mu}^\ad$. For an adic space $X\in\Ad_{E_\mu}^\lft$ the groupoid $\sH_{\phi,N,\mu}^\nabla(X)$ consists of those $(D,\phi,N,\Fq)\in\sH_{\phi,N,\mu}^\ad(X)$ for which the associated $(\phi,N_\nabla)$-module $(\CM,\Phi_\CM,N_\nabla^\CM)$ satisfies $N_\nabla^\CM(\CM)\subset\CM$. The groupoids $\sH_{\phi,N,\preceq\mu}^\nabla(X)$, resp.\ $\sH_{\phi,\mu}^\nabla(X)$, resp.\ $\sH_{\phi,\preceq\mu}^\nabla(X)$ are defined by the same condition. (Note that on the latter two $N=0$, but $N_\nabla^\CM\ne0$.)
\end{definition}

\begin{theorem}\label{ThmHnabla}
The substacks $\sH_{\phi,N,\mu}^\nabla\subset\sH_{\phi,N,\mu}^\ad$, resp.\ $\sH_{\phi,N,\preceq\mu}^\nabla\subset\sH_{\phi,N,\preceq\mu}^\ad$, resp.\ $\sH_{\phi,\mu}^\nabla\subset\sH_{\phi,\mu}^\ad$, resp.\ $\sH_{\phi,\preceq\mu}^\nabla\subset\sH_{\phi,\preceq\mu}^\ad$ are Zariski closed substacks. The substack $\sH_{\phi,N,\mu}^\nabla$ coincides with the image of the zero section of the vector bundle $\sH_{\phi,N,\mu}^\ad\to\sD_{\phi,N,\mu}^\ad$.
\end{theorem}

\begin{rem}\label{familyonDadphiN}
We can consider a family of $(\phi,N_\nabla)$-modules over $\sD_{\phi,N,\mu}^{\rm ad}$ namely we pull back the canonical family of $(\phi,N_\nabla)$-modules on $\sH_{\phi,N,\mu}^{\rm ad}$ along the zero section. 
Then for $x\in \sD_{\phi,N,\mu}^{\rm ad}(\Q_p)$ the fiber of this family at $x$ coincides with the $(\phi,N_\nabla)$-module that Kisin \cite{crysrep} associates with the filtered $(\phi,N)$-module defined by $x$. 
\end{rem}

\begin{remark}\label{RemOtherTheta}
If instead of the isomorphism $\eta_\ulD$ from \eqref{EqTheta} we choose $\eta_\ulD=\id_\Fp$ as in Remark~\ref{RemOtherTheta1}(2), the above results remain valid, except that $\sH_{\phi,N,\mu}^\nabla$ coincides with the image of a different section. This section is obtained by composing the zero section with the inverse $\eta_\ulD^{-1}:d\otimes f\mapsto\TS\sum_iN^i(d)\otimes  \tfrac{1}{i!}\log\bigl(1-\tfrac{E(u)}{E(0)}\bigr)^i\cdot f$ of the automorphism $\eta_\ulD$. It sends a filtration $\CF^\bullet$ to $\eta_\ulD^{-1}\bigl(\sum_{i\in \Z} E(u)^{-i}(\CF^iD_K)\otimes_{R\otimes K}\BdRplus{R}\bigr)$. Note that both sections coincide on the closed substack $\sD_{\phi,\mu}^\ad$ where $N=0$.
\end{remark}

\begin{proof}[Proof of Theorem~\ref{ThmHnabla}]
To prove that the substacks are closed let $\ulD\in\sH_{\phi,N,\preceq\mu}^\ad(X)$ for an adic space $X\in\Ad_{E_\mu}^\lft$ and let $(M,\Phi_\CM,N_\nabla^\CM)=\ul\CM(\ulD)$ be the associated $(\phi,N_\nabla)$-module over $X$. Locally on $X$ there is an integer $h$ with $N_\nabla^\CM(\CM)\subset\lambda^{-h}\CM$ by Lemma~\ref{LemmaHPLattice} and the construction of $N_\nabla^\CM$. The quotients $\bigl(\lambda^{-h}\CM/\CM\bigr)\otimes\bigl(\sB_X^{[0,1)}/(\phi^n(E(u))^h)\bigr)$ are finite locally free as $\CO_X$-modules for all $n\ge0$. Now the condition $N_\nabla^\CM(\CM)\subset\CM$ is equivalent to the vanishing of the images under $N_\nabla^\CM$ of a set of generators of $\CM$ in $\bigl(\lambda^{-h}\CM/\CM\bigr)\otimes\bigl(\sB_X^{[0,1)}/(\phi^n(E(u))^h)\bigr)$ for each $n\ge0$. Due to \cite[I$_{\rm new}$, Lemma~9.7.9.1]{EGA} the latter is represented by a Zariski closed subspace of $X$.

We show that the closed substack $\sH_{\phi,N,\mu}^\nabla$ of $\sH_{\phi,N,\mu}^\ad$ coincides with the image of the zero section. Since $N_\nabla^\CalD$ on $\CalD:=D\otimes_{\CO_X\otimes K_0}\sB_{\CO_X}^{[0,1)}$ induces the differential operator $\id_D\otimes N_\nabla$ on $\Fp:=D\otimes_{\CO_X\otimes K_0}\BdRplus{\CO_X}$ under the map $i_0=\eta_\ulD\circ\text{inclusion}:\CalD\into\Fp$ from \eqref{EqTheta}, it follows directly that the image of the zero section is contained in $\sH_{\phi,N,\mu}^\nabla$. To prove the converse we may work on the coverings $X:=(P_{K_0,d}\,\times_{\BQ_p}\,Q_{K,d,\mu})^\ad$ of $\sH_{\phi,N,\mu}^\ad$ and $(P_{K_0,d}\,\times_{\BQ_p}\,\Flag_{K,d,\mu})^\ad$ of $\sD_{\phi,N,\mu}^\ad$ because the zero section and $\sH_{\phi,N,\mu}^\nabla$ are both invariant under the action of $(\Res_{K_0/\BQ_p}\GL_{d,K_0})_{E_\mu}$. We first claim that both have the same underlying topological space. By \cite[Corollary 6.1.2/3]{BoschGR} this can be checked on $L$-valued points of $X$ for finite extensions $L$ of $E_\mu$. For those it was proved by Kisin~\cite[Lemma~1.2.12(4)]{crysrep} that the universal Hodge-Pink lattice $\Fq$ at $L$ lies in the image of the zero section if the pullback $\ul\CM$ to $L$ of the universal $(\phi,N_\nabla)$-module on $X$ has holomorphic $N_\nabla^\CM$. From this our claim follows.

To prove equality as closed subspaces of $X$ we look at a closed point $x\in X$ and its complete local ring $\wh\CO_{X,x}$. Let $\Fm_x\subset\wh\CO_{X,x}$ be the maximal ideal, let $I\subset\wh\CO_{X,x}$ be the ideal defining $\sH_{\phi,N,\mu}^\nabla$, and set $R_n:=\wh\CO_{X,x}/(\Fm_x^n+I)$. Then $R_n$ is a finite dimensional $\BQ_p$-vector space by \cite[Corollary 6.1.2/3]{BoschGR}. We consider the universal $\ulD_{R_n}=(D,\Phi,N,\Fq)$ over $R_n$ by restriction of scalars from $R_n$ to $\BQ_p$ as a $(\phi,N)$-module $\wt\ulD$ with Hodge-Pink lattice over $\BQ_p$ of rank $(\dim_{\BQ_p}R_n)(\rk_{R_n\otimes K_0}D)=\dim_{K_0}D$. It is equipped with a ring homomorphism $R_n\to\End(\wt\ulD)$. Since $N_\nabla^\CM$ is holomorphic on $\ul\CM(\ulD_{R_n})$, Kisin~\cite[Lemma~1.2.12(4)]{crysrep} tells us again that $\Fq=\Fq(\CF^\bullet)$ for the filtration $\CF^\bullet=\CF^\bullet_\Fq$ from Remark~\ref{Rem2.4}. This shows that the ideal $J$ defining the zero section in $X$ vanishes in $R_n$. Since this holds for all $n$, the ideals $I$ and $J$ are equal in $\wh\CO_{X,x}$. As $x$ was arbitrary, they coincide on all of $X$ and this proves the theorem.
\end{proof}

\section{Weak admissibility}
Similar to the case of filtrations, one can define a notion of weak admissibility for $(\phi,N)$-modules with Hodge-Pink lattice and develop a Harder-Narasimhan formalism. Compare also \cite[\S\,2]{Hellmann} for the following. 
Recall that $f=[K_0:\Q_p]$ and $e=[K:K_0]$.

\begin{defn}
Let $L$ be a field with a valuation $v_L:L\rightarrow \Gamma_L\cup\{0\}$ in the sense of Huber, see \cite[\S\,2, Definition]{contval} and set $\Gamma_L^\Q:=\Gamma_L\otimes_\Z\Q$.\\
\noindent (i) Let $\ulD=(D,\Phi,N)$ be a $(\phi,N)$-module over $L$. Then define 
\[t_N(\ulD)\;:=\;v_L(\det\nolimits_{\SSC L} \Phi^f)^{1/f^2}\;\in\;\Gamma_L^\Q.\]
\quad\es If $L\supset K_0$ we are in the situation of Remark~\ref{RemDecompOfD} and have $t_N(\ulD)=v_L(\det\nolimits_{\SSC L} (\Phi^f)_0)^{1/f}$.\\[2mm]
\noindent (ii) Let $\ulD=(D,\Phi,N,\Fcal^\bullet)$ be a $K$-filtered $(\phi,N)$-module over $L$. Then
\[t_H(\ulD)\;:=\; \tfrac{1}{ef} \sum_{i\in \Z}i\dim_L(\Fcal^iD_K/\Fcal^{i+1}D_K)\;\in\;\Q.\]
\noindent (iii) Let $\ulD=(D,\Phi,N,\Fq)$ be a $(\phi,N)$-module with Hodge-Pink lattice of rank $d$ over $L$. Then we set
\[
t_H(\ulD)\;:=\;\tfrac{1}{ef}\bigl(\dim_L(\Fq/t^n\Fp)-\dim_L(\Fp/t^n\Fp)\bigr)\;=\;\tfrac{1}{ef}\dim_L(\Fq/t^n\Fp)\,-\,n\rk\ulD\;\in\;\Q
\]
for $n\gg0$, which is independent of $n$ whenever $t^n\Fp\subset\Fq$. If $L$ is an extension of $\norm{K}$ and $(\mu_\psi)_\psi=\mu_\ulD(\Spec L)$ is the Hodge polygon of $\ulD$, see Definition~\ref{Defnbndmu}, then $\TS t_H(\ulD):= \tfrac{1}{ef} \sum_\psi\,\mu_{\psi,1}+\ldots+\mu_{\psi,d}$. If the $\psi$-component $\Fq_\psi$ satisfies $\wedge^d\Fq_\psi=t^{-h_\psi}\wedge^d\Fp_\psi$ then $t_H(\ulD)=\tfrac{1}{ef}\sum_\psi h_\psi$. Moreover $t_H(\ulD)=t_H(D,\Phi,N,\CF^\bullet_\Fq)$.

\smallskip\noindent (iv) Let $\ulD$ be a $(\phi,N)$-module with Hodge-Pink lattice (resp.\ a $K$-filtered $(\phi,N)$-module) over $L$. Then its \emph{slope} is defined to be
\[\slope(\ulD)\;:=\; (v_L(p)^{t_H(\ulD)}\cdot t_N(\ulD)^{-1})^{1/d}\;\in\;\Gamma_L^\Q.\]
\end{defn}
\begin{defn}
\noindent (i)
A $(\phi,N)$-module with Hodge-Pink lattice $\ulD=(D,\Phi,N,\Fq)$ over a field $L$ endowed with a valuation is called \emph{semi-stable} if $\slope(\ulD')\geq \slope(\ulD)$ for all $\ulD'=\bigl(D',\Phi|_{\phi^*D'},N|_{D'},\Fq\cap D'\otimes_{L\otimes_{\Q_p}K_0}\BdR{L}\bigr)$ where $D'\subset D$ is a free $L\otimes_{\BQ_p}K_0$-submodule stable under $\Phi$ and $N$. 

\smallskip\noindent (ii) A $K$-filtered $(\phi,N)$-module $\ulD=(D,\Phi,N,\Fcal^\bullet)$ over $L$ is called \emph{semi-stable} if   $\slope(\ulD')\geq \slope(\ulD)$ for all $\ulD'=\bigl(D',\Phi|_{\phi^*D'},N|_{D'},\Fcal^\bullet\cap D'_K\bigr)$ where $D'\subset D$ is a free $L\otimes_{\BQ_p}K_0$-submodule stable under $\Phi$ and $N$. 

\smallskip\noindent (iii) A $(\phi,N)$-module with Hodge-Pink lattice (resp.\ a $K$-filtered $(\phi,N)$-module) is called \emph{weakly admissible} if it is semi-stable of slope $1$.
\end{defn} 

\begin{lem}\label{waundersection}
Let $(D,\Phi,N,\Fcal^\bullet)$ be a $K$-filtered $(\phi,N)$-module over a valued field $L$ and let $(D,\Phi,N,\Fq)$ denote the $(\phi,N)$-module with Hodge-Pink lattice associated to $(D,\Phi,N,\Fcal^\bullet)$ by the zero section $\Fcal^\bullet\mapsto \Fq=\Fq(\Fcal^\bullet)$ of Remark~\ref{Rem2.4}(3). Then $(D,\Phi,N,\Fcal^\bullet)$ is weakly admissible if and only if $(D,\Phi,N,\Fq)$ is.   
\end{lem}
\begin{proof}
It is obvious from the definitions that $t_H(D,\Phi,N,\Fcal^\bullet)=t_H(D,\Phi,N,\Fq(\Fcal^\bullet))$. Further we have to test on the same sub-objects $D'\subset D$. Hence the claim follows from the fact
\[\Fq(\Fcal^\bullet \cap D'_K)=\Fq(\Fcal^\bullet)\cap D'\otimes_{L\otimes_{\Q_p}K_0}\BdR{L},\]
which is obvious from the description of $\Fq(-)$ in Remark~\ref{Rem2.4}(3) by choosing an $L\otimes_{\BQ_p}K$-basis of $D_K$ adapted to the submodules $\CF^iD'_K$ and $\CF^iD_K$.
\end{proof}

\begin{prop}
Let $(D,\Phi,N,\q)$ be a $(\phi,N)$-module with Hodge-Pink lattice defined over some valued field $L$. Then there is a unique Harder-Narasimhan filtration \[0=D_0\subset D_1\subset \dots\subset D_r=D\] of $(D,\Phi,N,\q)$,  by free $L\otimes_{\Q_p}K_0$-submodules stable under $\Phi$ and $N$ such that the subquotients $D_i/D_{i-1}$ with their induced Hodge-Pink lattice are semi-stable of slope $\slope_i\in \Gamma_L\otimes \Q$ and $\slope_1<\slope_2<\dots<\slope_r$.
\end{prop}
\begin{proof}
This is the usual Harder-Narasimhan formalism, see \cite[3]{FargFont} for a fairly general exposition. See also \cite[Proposition~2.19]{Hellmann}.
\end{proof}
\begin{cor}\label{waandBC}
Let $(D,\Phi,N,\q)$ be a $(\phi,N)$-module with Hodge-Pink lattice over $L$ and let $L'$ be an extension of $L$ with valuation $v_{L'}$ extending the valuation $v_L$. Then $(D,\Phi,N,\q)$ is weakly admissible if and only if $(D',\Phi',N',\q')=(D\otimes_LL',\Phi\otimes \id,N\otimes\id,\q\otimes_LL')$ is weakly admissible. 
\end{cor}
\begin{proof}
This is similar to \cite[Corollary~2.22]{Hellmann}. 

If $(D',\Phi',N',\q')$ is weakly admissible, then every $(\Phi,N)$-stable subobject $D_1\subset D$ defines a $(\Phi',N')$-stable subobject $D'_1=D_1\otimes_LL'$ of $D'$ such that 
\[\slope(D_1,\Phi|_{\phi^*D_1},N|_{D_1},\q\cap D_1\otimes_{L\otimes K_0}\BdR{L} )=\slope(D'_1,\Phi'|_{\phi^*D'_1},N'|_{D'_1},\q'\cap D'_1\otimes_{L'\otimes K_0}\BdR{L'} ).\]
It follows that $(D,\Phi,N,\q)$ is weakly admissible, as $(D',\Phi',N',\q')$ is.

Now assume that $(D,\Phi,N,\q)$ is weakly admissible. We may reduce to the case where $L'$ is finitely generated over $L$ and ${\rm Aut}(L'/L)$ is large enough. As in the proof of \cite[Corollary~2.22]{Hellmann} one shows that the action of ${\rm Aut}(L'/L)$ preserves the slope of $\Phi'$-stable sub-objects of $D'$. Hence the Harder-Narasimhan filtration of $D'$ descends to $D$. As $D$ is weakly admissible, this filtration can only have one step. 
\end{proof}

\begin{theo}\label{ThmWAOpen}
Let $\mu$ be a cocharacter as in $(\ref{mu})$ with reflex field $E_\mu$. Then the groupoid 
\[X\mapsto \{(D,\Phi,N,\q)\in \sH_{\phi,N, \preceq \mu}(X)\mid D\otimes \kappa(x)\ \text{is weakly admissible for all}\ x\in X\}\]
is an open substack $\sH^{\rm ad, wa}_{\phi,N, \preceq \mu}$ of $\sH^\ad_{\phi,N, \preceq \mu}$ on the category of adic spaces locally of finite type over $E_\mu$. 
\end{theo}
\begin{proof}
This is similar to the proof of \cite[Theorem 4.1]{families}.

It follows from Corollary \ref{waandBC} that $\sH^{\rm ad, wa}_{\phi,N,\preceq \mu}$ is indeed a stack, i.e.\ weak admissibility may be checked over an fpqc-covering. Hence it suffices to show that the weakly admissible locus is open in
\[X_\mu:= P_{K_0,d} \times_{\BQ_p} Q_{K,d,\preceq \mu}.\]
Let us denote by $Z_i$ the projective $P_{K_0,d}$-scheme whose $S$-valued points are given by pairs $(x,U)$ with $x=(g,N)\in P_{K_0,d}(S)\subset (\Res_{K_0/\BQ_p}\GL_{d,K_0})\times (\Res_{K_0/\Q_p}{\rm Mat}_{d\times d})$ and an $\Ocal_S\otimes K_0$-subspace $U\subset \Ocal_S\otimes K_0^{\oplus d}$ which is locally on $S$ free of rank $i$, a direct summand as $\CO_S$-module, and stable under the action of $\Phi_g=g\cdot\phi$ and $N$. This is a closed subscheme of the product $P_{K_0,d}\times_{\BQ_p}\Quot_{\CO^d\,|\,K_0\,|\,\BQ_p}$ (where $\Quot_{\CO^d\,|\,K_0\,|\,\BQ_p}$ is Grothendieck's Quot-scheme which is projective over $\BQ_p$; see \cite[n$\open$221, Theorem 3.1]{FGA} or \cite[Theorem 2.6]{AltmanKleiman}), cut out by the invariance conditions under $\Phi_g$ and $N$. 
Further write $f_i\in \Gamma(Z_i,\Ocal_{Z_i})$ for the global section defined by
\[f_i(g,U)=\det(g\cdot\phi)^f|_U=\det\bigl(g\cdot\phi(g)\cdot\ldots\cdot\phi^{f-1}(g)\bigr)^f|_U,\]
where $f=[K_0:\Q_p]$, and where the determinant is the determinant as $\CO_{Z_i}$-modules. Write $U$ for the pullback of the universal $(\Phi,N)$-invariant subspace on $Z_i$ to the product $Z_i\times Q_{K,d,\preceq\mu}$, write $\Fq$ for the pullback of the universal $\BdRplus{}$-lattice on $Q_{K,d,\preceq\mu}$ to $Z_i\times Q_{K,d,\preceq\mu}$, and write $\Fp=(\BdRplus{})^{\oplus d}$ for the pullback of the tautological $\BdRplus{}$-lattice $D\otimes\BdRplus{}$ on $P_{K_0,d}$ to $Z_i\times Q_{K,d,\preceq\mu}$. Fix integers $n,h$ with $t^n\Fp\subset\Fq\subset t^{-h}\Fp$ and consider the complex of finite locally free sheaves on $Z_i\times Q_{K,d,\preceq\mu}$
\[
P_\bullet:\quad P_1\,:=\, t^{-h}\Fp/t^n\Fp \;\xrightarrow{\es\delta\;}\;t^{-h}\Fp/\Fq \oplus \bigl(D/U\otimes t^{-h}\BdRplus{}/t^n\BdRplus{}\bigr)\,=:\,P_0
\]
given by the canonical projection $D\onto D/U$ in the second summand. Let $T_1$ be the functor from the category of quasi-coherent sheaves on $Z_i\times Q_{K,d,\preceq\mu}$ to itself defined by
\[
T_1: M\longmapsto\ker\bigl(\delta\otimes\id_M:P_1\otimes M \to P_0\otimes M\bigr)\,.
\]
If $M=\kappa(y)$ for a point $y=(g_y,N_y,U_y,\Fq_y)\in Z_i\times Q_{K,d,\preceq\mu}$ then $T_1(\kappa(y))=(\Fq_y\cap\Fp_{i,y}[\tfrac{1}{t}])/t^n\Fp_{i,y}$ where we write $\Fp_{i,y}:=U_y\otimes_{\kappa(y)\otimes K_0}\BdRplus{\kappa(y)}$. We consider the function 
\begin{eqnarray}\label{Hodgeslope}
h_i:Z_i\times Q_{K,d,\preceq\mu} & \longto & \Q, \nonumber\\
y & \longmapsto & \tfrac{1}{ef}\dim_{\kappa(y)}T_1(\kappa(y))\,-\,ni\;=\;t_H\bigl(U_y,g_y(\id\otimes\phi)|_{U_y},N|_{U_y},\Fq_y\cap\Fp_{i,y}[\tfrac{1}{t}]\bigr)\,.
\end{eqnarray}
We write $Z_i^{\rm ad}$ resp.\ $Q_{K,d,\preceq\mu}^{\rm ad}$ for the adic spaces associated to the varieties $Z_i$ and $Q_{K,d,\preceq\mu}$.  Similarly we write $h_i^{\rm ad}$ for the function on the adic spaces $Z_i^{\ad}\times Q^{\ad}_{K,d,\preceq\mu}$ defined by the same formula as in $(\ref{Hodgeslope})$. By semi-continuity \cite[III$_2$, Th\'eor\`eme~7.6.9]{EGA}, the sets 
\[Y_{i,m}=\{y\in Z_i^\ad\times Q^\ad_{K,d,\preceq\mu}\mid h_i(y)\geq m\}\]
are closed and hence proper over $X_\mu^\ad=P^\ad_{K_0,d}\times_{\BQ_p} Q^\ad_{K,d,\leq\mu}$. We write 
\[\pr_{i,m}:Y_{i,m}\longrightarrow X^\ad_\mu\]
for the canonical, proper projection. 

If we write $X_0\subset X^\ad_\mu$ for the open subset of all $(D,\Phi,N,\q)$ such that $\slope(D,\Phi,N,\q)=1$, then 
\begin{equation}\label{walocus}
X_0\setminus X_\mu^{\rm wa}=X_0\cap \bigcup_{i,m} \pr_{i,m}\big(\big\{y\in Y_{i,m}\mid v_y(f_i)> v_y(p)^{f^2m} \big\}\big),
\end{equation}
where the union runs over $1\leq i\leq d-1$ and $m\in \Z$. 
Indeed: Let $x=(D,\Phi,N,\q)$ be an $L$-valued point of $X_0$, then any proper $(\Phi,N)$-stable subspace of $D'\subset D$ defines (for some $1\leq i\leq d-1$) a point $y=(D',\q)$ of $Z_i\times Q_{K,d,\leq\mu}$ mapping to $x$. This subspace violates the weak admissibility condition if and only if   
\[v_y(f_i)=t_N(D',\Phi|_{\phi^*D'})^{f^2}>v_y(p)^{f^2\,t_H(D',\Phi,\q\cap (D'\otimes\BdR{\kappa(x)}))}=v_y(p)^{f^2h_i(y)},\]
and hence $(\ref{walocus})$ follows. On the other hand the union
\[\bigcup_{i,m} \pr_{i,m}\big(\big\{y\in Y_{i,m}\mid v_y(f_i)> v_y(p)^{f^2m} \big\}\big)\]
is a finite union, because $Y_{i,m}=\emptyset$ for $m>hi$ and $Y_{i,m}= Z_i^\ad\times Q^\ad_{K,d,\preceq\mu}$ for $m\le-ni$. Therefore the union is closed by properness of the map $\pr_{i,m}$ and the definition of the topology on an adic space. The theorem follows from this. 
\end{proof}

We define sub-groupoids  $\sH_{\phi,\preceq\mu}^{\rm ad, wa}\subset \sH_{\phi,\preceq\mu}^{\rm ad}$, resp.\ $\sH_{\phi,N,\mu}^{\rm ad, wa}\subset \sH_{\phi,N,\mu}^{\rm ad}$, resp.\ $\sH_{\phi,\mu}^{\rm ad, wa}\subset \sH_{\phi,\mu}^{\rm ad}$, resp.\ $\sD_{\phi,N,\mu}^{\rm ad,wa}\subset \sD_{\phi,N,\mu}^{\rm ad}$, resp.\ $\sD_{\phi,\mu}^{\rm ad,wa}\subset \sD_{\phi,\mu}^{\rm ad}$ as follows: Given an adic space $X$ and $(D,\Phi,N,\Fq)\in \sH_{\phi,\preceq\mu}^{\rm ad}$, we say that $(D,\Phi,N,\Fq)\in \sH_{\phi,\preceq\mu}^{\rm ad, wa}$ if and only if $(D,\Phi,N,\Fq)\otimes \kappa(x)$ is weakly admissible for all points $x\in X$. The other sub-groupoids are defined in the same manner.
\begin{cor}\label{Cor3.7}
The sub-groupoids $\sH_{\phi,\preceq\mu}^{\rm ad, wa}\subset \sH_{\phi,\preceq\mu}^{\rm ad}$, resp.\ $\sH_{\phi,N,\mu}^{\rm ad, wa}\subset \sH_{\phi,N,\mu}^{\rm ad}$, resp.\ $\sH_{\phi,\mu}^{\rm ad, wa}\subset \sH_{\phi,\mu}^{\rm ad}$, resp.\ $\sD_{\phi,N,\mu}^{\rm ad,wa}\subset \sD_{\phi,N,\mu}^{\rm ad}$, resp.\ $\sD_{\phi,\mu}^{\rm ad,wa}\subset \sD_{\phi,\mu}^{\rm ad}$ are open substacks.
\end{cor}
\begin{proof}
This follows by pulling back $\sH_{\phi,N,\preceq\mu}^{\rm ad,wa}\subset\sH_{\phi,N,\preceq\mu}^{\rm ad}$ along the morphisms $\sH_{\phi,\preceq\mu}^{\rm ad}\rightarrow \sH_{\phi,N,\preceq\mu}^{\rm ad}$, resp.\ $\sH_{\phi,N,\mu}^{\rm ad}\rightarrow \sH_{\phi,N,\preceq\mu}^{\rm ad}$, resp.\ $\sH_{\phi,\mu}^{\rm ad}\rightarrow \sH_{\phi,N,\preceq\mu}^{\rm ad}$, resp.\ $\sD_{\phi,N,\mu}^{\rm ad}\rightarrow \sH_{\phi,N,\preceq\mu}^{\rm ad}$, resp.\ $\sD_{\phi,\mu}^{\rm ad}\rightarrow \sH_{\phi,N,\preceq\mu}^{\rm ad}$. 
Here we use the fact that the zero sections $\sD_{\phi,N,\mu}^{\rm ad}\rightarrow \sH_{\phi,N,\preceq\mu}^{\rm ad}$ and $\sD_{\phi,\mu}^{\rm ad}\rightarrow \sH_{\phi,N,\preceq\mu}^{\rm ad}$ preserve weak admissibility by Lemma \ref{waundersection}. 
\end{proof}

\begin{rem}\label{Rem3.8}
Note that the projection
\[\pr:\sH_{\phi,N,\mu}^{\rm ad}\longrightarrow \sD_{\phi,N,\mu}^{\rm ad}\]
does not preserve weak admissibility. We always have $\pr^{-1}(\sD_{\phi,N,\mu}^{\rm ad, wa})\subset \sH_{\phi,N,\mu}^{\rm ad, wa}$ and hence especially any section of the vector bundle $\sH_{\phi,N,\mu}^{\rm ad}\rightarrow \sD_{\phi,N,\mu}^{\rm ad}$ maps the weakly admissible locus to the weakly admissible locus.

 Indeed, let $\ulD=(D,\Phi,N,\Fq)$ be a point of $\sH_{\phi,N,\mu}^\ad$ over a field $L$ whose image $(D,\Phi,N,\CF^\bullet_\Fq)$ in $\sD_{\phi,N,\mu}^\ad$ is weakly admissible. Let $D'\subset D$ be an $L\otimes_{\BQ_p}K_0$-submodule which is stable under $\Phi$ and $N$, and set $\Fp':=D'\otimes_{L\otimes K_0}\BdRplus{L}$. Then $\Fq':=\Fq\cap\Fp'[\tfrac{1}{t}]$ satisfies $t^i\Fq'\cap\Fp'\subset t^i\Fq\cap\Fp$ and $\CF^i_{\Fq'}D'_K\subset \CF^i_\Fq D_K\cap D'_K$. This implies 
\[
t_H(D',\Phi|_{\phi^*D'},N|_{D'},\Fq')\;=\;t_H(D',\Phi|_{\phi^*D'},N|_{D'},\CF^\bullet_{\Fq'})\;\le\; t_H(D',\Phi|_{\phi^*D'},N|_{D'},\CF^\bullet_\Fq\cap D'_K)
\]
with equality for $D'=D$, and
\[
v_L(p)^{t_H(D',\Phi|_{\phi^*D'},N|_{D'},\Fq')}\;\ge\;v_L(p)^{t_H(D',\Phi|_{\phi^*D'},N|_{D'},\CF^\bullet_\Fq\cap D'_K)}\;\ge\; t_N(D',\Phi|_{\phi^*D'})
\]
with equality for $D'=D$, because $(D,\Phi,N,\CF^\bullet_\Fq)$ is weakly admissible. Therefore also $\ulD$ is weakly admissible. This proves that  $\pr^{-1}(\sD_{\phi,N,\mu}^{\rm ad, wa})\subset \sH_{\phi,N,\mu}^{\rm ad, wa}$. However, in general this inclusion is strict as can be seen from the following example. 
\end{rem}

\begin{example}
Let $K=K_0=\BQ_p$, $d=2$ and $\mu=(2,0)$. We consider points $\ulD=(D,\Phi,N,\CF^\bullet)$ in $\sD_{\phi,N,\mu}^\ad$ over a field $L$ with $\Phi=p\Id_2$ and $N=0$. The filtration is of the form $D=\CF^0D\supset\CF^1D=\CF^2D=\bigl(\begin{smallmatrix}u\\v\end{smallmatrix}\bigr)\cdot L\supset\CF^3D=(0)$ for some $(\begin{smallmatrix}u\\v\end{smallmatrix}\bigr)\in D$. None of these points is weakly admissible, because the subspace $D'=\bigl(\begin{smallmatrix}u\\v\end{smallmatrix}\bigr)\cdot L\subset D$ has $t_N(\ulD')=v_L(p)$ and $\CF^2D'=D'$, whence $t_H(\ulD')=2$ and $\slope(\ulD')=v_L(p)<1$.

The preimage of such a point in $\sH_{\phi,N,\mu}^\ad$ is given by a Hodge-Pink lattice $\Fq$ with $\Fp\subset\Fq\subset t^{-2}\Fp$ with Hodge weights $0$ and $-2$. This means that $\Fq=\Fp+\bigl(\begin{smallmatrix}u+tu'\\v+tv'\end{smallmatrix}\bigr)\cdot t^{-2}\BdRplus{L}$ for some $(\begin{smallmatrix}u'\\v'\end{smallmatrix}\bigr)\in D$. If the vectors $\bigl(\begin{smallmatrix}u\\v\end{smallmatrix}\bigr)$ and $\bigl(\begin{smallmatrix}u'\\v'\end{smallmatrix}\bigr)$ are linearly dependent over $L$ then $\ulD=(D,\Phi,N,\Fq)$ is not weakly admissible, because the subspace $D'=\bigl(\begin{smallmatrix}u\\v\end{smallmatrix}\bigr)\cdot L\subset D$ has $t_N(\ulD')=v_L(p)$ and $\Fq':=\Fq\cap D'\otimes_L\BdR{L}=t^{-2}D'\otimes_L\BdRplus{L}$, whence $t_H(\ulD')=2$ and $\slope(\ulD')=v_L(p)<1$.

On the other hand, if the vectors $\bigl(\begin{smallmatrix}u\\v\end{smallmatrix}\bigr)$ and $\bigl(\begin{smallmatrix}u'\\v'\end{smallmatrix}\bigr)$ are linearly \emph{independent} over $L$ then $\ulD=(D,\Phi,N,\Fq)$ is weakly admissible, because then $\Fq'\subset t^{-1}D'\otimes_L\BdRplus{L}$ for any subspace $D'=\bigl(\begin{smallmatrix}a\\b\end{smallmatrix}\bigr)\cdot L\subset D$, whence $t_N(\ulD')=v_L(p)$, $t_H(\ulD')\le1$ and $\slope(\ulD')\ge1$. Indeed, $\bigl(\begin{smallmatrix}a\\b\end{smallmatrix}\bigr)\cdot t^{-2}\in\Fq'$ would imply that $\bigl(\begin{smallmatrix}a\\b\end{smallmatrix}\bigr)\cdot t^{-2}\equiv\bigl(\begin{smallmatrix}u+tu'\\v+tv'\end{smallmatrix}\bigr)\cdot t^{-2}\cdot(c+tc')\equiv c\bigl(\begin{smallmatrix}u\\v\end{smallmatrix}\bigr)\cdot t^{-2}+(c'\bigl(\begin{smallmatrix}u\\v\end{smallmatrix}\bigr)+c\bigl(\begin{smallmatrix}u'\\v'\end{smallmatrix}\bigr))t^{-1}\mod\Fp$ for $c,c'\in L$. This implies $\bigl(\begin{smallmatrix}a\\b\end{smallmatrix}\bigr)=c\bigl(\begin{smallmatrix}u\\v\end{smallmatrix}\bigr)$ and $c'\bigl(\begin{smallmatrix}u\\v\end{smallmatrix}\bigr)+c\bigl(\begin{smallmatrix}u'\\v'\end{smallmatrix}\bigr)=0$ contradicting the linear independence.

Thus the weakly admissible locus $\sH_{\phi,\mu}^{\ad,\rm wa}$ in the fiber of $\sH_{\phi,N,\mu}$ over the point $(\Phi,N)=(p\Id_2,0)$ in $P_{\BQ_p,2}$ equals the complement of the zero section, while this fiber in $\sD_{\phi,N,\mu}^{\ad,\rm wa}$ is empty; see also Lemma~\ref{waundersection}.
\end{example}

We end this section by remarking that the weakly admissible locus is determined by the rigid analytic points, i.e.\ those points of an adic space whose residue field is a finite extension of $\Q_p$. 
\begin{lem}\label{maxmod}
Let $X$ be an adic space locally of finite type over $E_\mu$ and let $f: X\rightarrow \sH_{\phi,N,\preceq\mu}^{\rm ad}$ be a morphism defined by a $(\phi,N)$-module with Hodge-Pink lattice $\ulD$. Then $f$ factors over $\sH_{\phi,N,\preceq\mu}^{\rm ad,wa}$ if and only if $\ulD\otimes \kappa(x)$ is weakly admissible for all rigid analytic points $x\in X$. 
\end{lem}
\begin{proof}
One implication is obvious and the other one is an easy application of the maximum modulus principle. It is proven along the same lines as \cite[Proposition 4.3]{families}.
\end{proof}
\begin{rem}
The analogous statements for the stacks $\sH_{\phi,\preceq\mu}^{\rm ad, wa}\subset \sH_{\phi,\preceq\mu}^{\rm ad}$, resp.\ $\sH_{\phi,N,\mu}^{\rm ad, wa}\subset \sH_{\phi,N,\mu}^{\rm ad}$, resp.\ $\sH_{\phi,\mu}^{\rm ad, wa}\subset \sH_{\phi,\mu}^{\rm ad}$, resp.\ $\sD_{\phi,N,\mu}^{\rm ad,wa}\subset \sD_{\phi,N,\mu}^{\rm ad}$, resp.\ $\sD_{\phi,\mu}^{\rm ad,wa}\subset \sD_{\phi,\mu}^{\rm ad}$ are also true and are a direct consequence of their construction. 
\end{rem}

\section{The \'etale locus}
Let us denote by $\boldB_{[r,s]}$ the closed annulus over $K_0$ of inner radius $r$ and outer radius $s$ for some $r,s\in [0,1)\cap p^\Q$. For an adic space $X\in \Ad^{\rm lft}_{\Q_p}$ we write
\begin{align*}
\sA_X^{[0,1)}&=\pr_{X,\ast}\Ocal_{X\times \Ubb}^+\subset \sB_X^{[0,1)}=\pr_{X,\ast}\Ocal_{X\times\Ubb}\\
\sA_X^{[r,s]}&=\pr_{X,\ast}\Ocal_{X\times \boldB_{[r,s]}}^+\subset \sB_X^{[r,s]}=\pr_{X,\ast}\Ocal_{X\times\boldB_{[r,s]}}
\end{align*}
The Frobenius $\phi$ on $\sB_X^{[0,1)}$ restricts to a ring homomorphism $\phi$ on $\sA_X^{[0,1)}$. For this section we adapt the notation from \cite{families} and write $r_i=r^{1/p^i}$. Then $\phi$ restricts to a homomorphism 
\[\phi:\sB_X^{[r,s]}\longrightarrow \sB_X^{[r_1,s_1]}.\]

\begin{defn}\label{Def4.2}
A \emph{$\phi$-module of finite height} over $\sA_X^{[0,1)}$ is an $\sA_X^{[0,1)}$-module $\Mfrak$ which is locally on $X$ free of finite rank over $\sA_X^{[0,1)}$ together with an injective morphism $\Phi:\phi^\ast\Mfrak\rightarrow \Mfrak$ of $\sA_X^{[0,1)}$-modules such that $\coker \Phi$ is killed by some power of $E(u)\in W\dbl u\dbr\subset \sA_X^{[0,1)}$.
\end{defn}

\newcommand{\ddd}{n}
Inspired by Example~\ref{ExCyclot2} we define the $(\phi,N_\nabla)$-module $\sB_X^{[0,1)}(1)$ over $X$ to be $(\sB_X^{[0,1)},\Phi_\CM=\tfrac{pE(u)}{E(0)},N_\nabla^\CM)$ with $N_\nabla^\CM(f)=N_\nabla(f)+u\tfrac{d\lambda}{du}\,f$. For an integer $\ddd\in\Z$ we set $\sB_X^{[0,1)}(\ddd):=\sB_X^{[0,1)}(1)^{\otimes\ddd}=(\sB_X^{[0,1)},\bigl(\tfrac{pE(u)}{E(0)}\bigr)^\ddd,N_\nabla^\CM)$ with $N_\nabla^\CM(f)=N_\nabla(f)+\ddd u\tfrac{d\lambda}{du}\,f$. Given a $(\phi,N_\nabla)$-module $(\Mcal,\Phi_\Mcal)$ on $X$ we write $(\Mcal,\Phi_\Mcal)(\ddd)$ for the twist $\Mcal\otimes_{\sB_X^{[0,1)}}\sB_X^{[0,1)}(\ddd)$. Note that $\tfrac{p}{E(0)}\in W\mal$ since $E(u)$ is an Eisenstein polynomial. Thus for $\ddd\ge0$ we have an obvious integral model $\sA_X^{[0,1)}(\ddd)$ for $\sB^{[0,1)}(\ddd)$ which is a $\phi$-module of finite height over $\sA_X^{[0,1)}$ (by forgetting the $N_\nabla$-action). Further we write $\bA^{[0,1)}(\ddd)=\sA_{\Spa(\Q_p,\Z_p)}(\ddd)$ for the $W\dbl u \dbr$-module of rank $1$ with basis $e$ on which $\Phi$ acts via $\Phi(e)=\bigl(\tfrac{E(u)}{pE(0)}\bigr)^\ddd e$. 

\begin{defn}\label{DefEtale}
Let $(\Mcal,\Phi_\Mcal,N_\nabla^\Mcal)$ be a $(\phi,N_\nabla)$-module over an adic space $X\in \Ad_{\Q_p}^{\rm lft}$. \\
\noindent (i) The module $\Mcal$ is called \emph{\'etale} if there exists an fpqc-covering $(U_i\rightarrow X)$, an integer $\ddd\geq 0$ and $\phi$-modules $(\Mfrak_i,\Phi_{\Mfrak_i})$ of finite height over $\sA_{U_i}^{[0,1)}$ such that 
\[(\Mcal,\Phi_\Mcal)(\ddd)|_{U_i}=(\Mfrak_i,\Phi_{\Mfrak_i})\otimes_{\sA_{U_i}^{[0,1)}}\sB_{U_i}^{[0,1)}.\]
\noindent (ii) Let $x\in X$, then $\Mcal$ is called \emph{\'etale at $x$} if there exists an integer $\ddd\geq 0$ and a $(\kappa(x)^+\otimes_{\Z_p}W)\dbl u\dbr$-lattice $\Mfrak\subset \Mcal(\ddd)\otimes \kappa(x)$ such that $E(u)^h\Mfrak\subset \Phi_{\Mcal(\ddd)}(\phi^\ast\Mfrak)\subset \Mfrak$ for some integer $h\geq 0$.
\end{defn}

\begin{theo}\label{etale}
Let $X$ be an adic space locally of finite type over $\Q_p$ and let $(\Mcal,\Phi)$ be a $(\phi,N_\nabla)$-module. Then the subset 
\[X^{\rm int}=\{x\in X\mid \Mcal\ \text{is \'etale at}\ x\}\]
is open and the restriction $\Mcal|_{X^{\rm int}}$ is \'etale.
\end{theo}
This is similar to the proof of \cite[Theorem 7.6]{families}. However, we need to make a few generalizations as we cannot rely on a reduced universal case. 
Given an affinoid algebra $A$ and $r,s\in [0,1)\cap p^\Q$ we write  
\begin{align*}
\bB_A^{[r,s]}=\Gamma(\Spa(A,A^\circ),\sB_{\Spa(A,A^\circ)}^{[r,s]})=A\widehat{\otimes}_{\Q_p}\bB^{[r,s]}=A_W\langle T/s, r/T \rangle,\\
\bA_A^{[r,s]}=\Gamma(\Spa(A,A^\circ),\sA_{\Spa(A,A^\circ)}^{[r,s]})=A^\circ\widehat{\otimes}_{\Z_p}\bA^{[r,s]}=A_W^\circ\langle T/s, r/T \rangle.
\end{align*}
The following is the analogue of \cite[Theorem 6.9]{families} in the non-reduced case.

\begin{theo}\label{loket}
Let $X$ be an adic space locally of finite type over $\Q_p$ and let $\Ncal$ be a family of free $\phi$-modules of rank $d$ over $\sB_X^{[r,r_2]}$. Assume that there exists $x\in X$ and an $\sA_X^{[r,r_2]}\otimes \kappa(x)^\circ$-lattice $N_x\subset \Ncal\otimes \kappa(x)$ such that $\Phi$ induces an isomorphism
\begin{equation}\label{etonannulus1}
\Phi:\phi^\ast(N_x\otimes_{\sA_X^{[r,r_2]}}\sA_X^{[r,r_1]})\isoto N_x\otimes_{\sA_X^{[r,r_2]}}\sA_X^{[r_1,r_2]}.
\end{equation}
Then there exists an open neighborhood $U\subset X$ of $x$ and a locally free $\sA_U^{[r,r_2]}$-submodule $N\subset \Ncal$ of rank $d$ such that 
\begin{align*}
N\otimes \kappa(x)^\circ &= N_x\\ 
\Phi(\phi^\ast N|_{U\times \boldB_{[r,r_1]}})&=N|_{U\times \boldB_{[r_1,r_2]}}\\
N\otimes_{\sA_U^{[r,r_2]}}\sB_U^{[r,r_2]}&=\Ncal|_U.
\end{align*}
\end{theo}
\begin{proof}
We may assume that $X=\Spa(A,A^\circ)$ is affinoid and we may choose a Banach norm $||\cdot||$ and a $\Z_p$-subalgebra $A^+=\{x\in A\mid ||x||\leq 1\}\subset A^\circ$ such that $A=A^+[1/p]$ and $X=\Spa(A,A^+)=\Spa(A,A^\circ)$.

Choose a basis $\underline{e}_x$\ of $N_x$ and denote by $D_0\in \GL_d(\sA_X^{[r,r_2]}\otimes \kappa(x)^\circ)$ the matrix of $\Phi$ in this basis. After shrinking $X$ if necessary we may lift the matrix $D_0$ to a matrix $D$ with coefficients in $\Gamma(X,\sA_X^{[r,r_2]})$. Localizing further we may assume that $D$ is invertible over $\Gamma(X,\sA_X^{[r,r_2]})$, as we only need to ensure that the inverse of its determinant has coefficients $a_i\in A^+$, i.e.\ $||a_i||\leq 1$ for some Banach norm $||\cdot||$ corresponding to $A^+$. 
Let us write $f\in A_w\langle T/s, r/T\rangle$ for this determinant and write $f=f^++f^-$ with
\begin{align*}
f^+&=\sum_{i\geq 0} \alpha_i (\tfrac{T}{s})^i\in A_w\langle \tfrac{T}{s}\rangle\\
f^-&=\sum_{i\geq 0} \beta_i (\tfrac{r}{T})^{i}\in A_w\langle \tfrac{r}{T}\rangle.
\end{align*}
We claim that $\alpha_i,\beta_i\in A^\circ_W$ for all $i$. But as $\alpha_i,\beta_i\xrightarrow{i\rightarrow \infty}0$ this is clear for all but finitely many $i$. Moreover for all $i\geq 0$ we have $\alpha_i(x),\beta_i(x)\in k(x)^\circ_W$. Hence after localization on $X$ we may assume that all coefficients are integral.

Fixing a basis $\underline{b}$ of $\Ncal$ we denote by $S\in \GL_d(\bB_A^{[r,r]})$ the matrix of $\Phi$ in this basis. Further we denote by $V$ a lift of the change of basis matrix from the basis $\underline{e}_x$ to the basis $\underline{b}\mod x$. From now on the proof is the same as the proof of \cite[Theorem 6.9]{families}.
\end{proof}

\begin{prop}\label{intmodel}
Let $X=\Spa\,(A,A^+)$ be an affinoid adic space of finite type over $\Q_p$.  
Let $r>|\pi|$ with $r\in p^\Q$ and set $r_i=r^{1/p^i}$. Let $\Mcal_r$ be a free vector bundle on $X\times\boldB_{[0,r_2]}$ together with an injection 
\[\Phi:\phi^\ast (\Mcal_r|_{X\times \boldB_{[0,r_1]}})\longrightarrow \Mcal_r\]
with cokernel supported at the point defined by $E(u)$.
Assume that there exists a free $\bA_A^{[r,r_2]}=A^+\langle T/r_2,r/T\rangle$ submodule 
\[N_r\subset \Ncal_r:=\Mcal_r\otimes_{\bB_A^{[0,r_2]}}\bB_A^{[r,r_2]}\]
of rank $d$, containing a basis of $\Ncal_r$ such that 
\[\Phi(\phi^\ast(N_r\otimes_{\bA_A^{[r,r_2]}} \bA_A^{[r,r_1]}))=N_r\otimes_{\bA_A^{[r,r_2]}} \bA_A^{[r_1,r_2]}. \]
Then fpqc-locally on $X$ there exists a free $\bA_A^{[0,r_2]}$-submodule $M_r\subset \Mcal_r$ of rank $d$, containing a basis of $\Mcal_r$ such that
\begin{equation}\label{phicond}
\Phi:\phi^\ast(M_r\otimes_{\bA_A^{[0,r_2]}}\bA_A^{[0,r_1]})\longrightarrow M_r
\end{equation}
is injective with cokernel killed by some power of $E(u)$.
\end{prop}
\begin{proof}
This is the generalization of \cite[Proposition 7.7]{families} to our context.
We also write $\Mcal_r$ for the global sections of the vector bundle. Write $M'_r=\Mcal_r\cap N_r\subset \Ncal_r$. This is an $A^+\langle T/r_2\rangle$-module. Further we set
\[M_r={\rm Im}(M'_r\widehat{\otimes}_{\bA_A^{[0,r_2]}}\bA_A^{[r,r_2]}\rightarrow \Ncal_r)\cap M'_r[\tfrac{1}{p}]\subset \Ncal_r.\]
Then $M_r$ is a finitely generated $A^+\langle T/r_2\rangle$-module as the ring is noetherian. 
First we need to make some modification in order to assure that $M_r$ is flat. Let $\Ycal=\Spf W\langle T/r_2\rangle$ denote the formal model of $\boldB_{[0,r_2]}$ and let $\Ycal'=\Spf W\langle T/r_2, r/T\rangle$ denote the formal model of $\boldB_{[r,r_2]}$. Note that $M_r[1/p]=\Mcal_r$ and hence $M_r$ is rig-flat.
By \cite[Theorem 4.1]{BoschLuetke} there exists a blow-up $\wt\Xcal$ of $\Spf(A)$ such that the strict transform $\wtM_r$ of $M_r$ in $\wt\Xcal\times \Ycal$ is flat over $\wt\Xcal$.
We write $\Mcal_{r,\wt\Xcal}$ (resp.\ $N_{r,\wt\Xcal}$) for the pullback of $\Mcal_r$ (resp.\ $N_r$) to the generic fiber of $\wt\Xcal\times\Ycal$ (resp.\ to $\wt\Xcal\times\Ycal'$). If we set $\wtM_r'=\Mcal_{r,\wt\Xcal}\cap N_{r,\wt\Xcal}$ then one easily finds
\[\wtM_r=(\wtM'_r\otimes_{\sA_{\wt\Xcal}^{[0,r_2]}}\sA_{\wt\Xcal}^{[r,r_2]})\cap \wtM'_r[\tfrac{1}{p}].\]
If follows that $\wtM_r$ is stable under $\Phi$. Further, as $\wtM_r$ is flat, it has no $p$-power torsion and hence we find that the formation $(\Mcal_{r,\wt\Xcal}, N_{r,\wt\Xcal})\mapsto \wtM_r$ commutes with base change $\Spf \Ocal\hookrightarrow \wt\Xcal$ for any finite flat $\Z_p$-algebra $\Ocal$, compare the proof of  \cite[Proposition 7.7]{families}. Especially this pullback is free over $\Ocal\otimes_{\Z_p}W\langle T/r_2\rangle$ and the cokernel of $\Phi$ is annihilated by $E(u)^{k_\mu}$ for some $k_\mu\gg 0$ depending only on the Hodge polygon $\mu$ (for an arbitrary finite flat $\Z_p$-algebra $\Ocal$ this follows by forgetting the $\Ocal$-structure and only considering the $\Z_p$-structure).

It follows that the restriction of $\wtM_r$ to the reduced special fiber $\wt\Xcal_0$ of $\wt\Xcal$ is locally free over $\wt\Xcal_0\times \Abb^1$ and hence, as in the proof of \cite[Proposition 7.7]{families} we may locally lift a basis and find that $\wtM_r$ is locally on $\wt\Xcal$ free over $\wt\Xcal\times\Ycal$.

It is only left to show that $E(u)^{k_\mu}\coker \Phi=0$ over $\wt\Xcal$. To do so we may localize and assume that $\wt\Xcal$ is affine. By abuse of language we denote it again by $\Spf A^+$ and write $N=E(u)^{k_\mu}\coker\Phi$. If $I$ denotes the ideal of nilpotent elements in $A^+$, we need to show that the multiplication $I\otimes_{A^+} N\rightarrow N$ is the zero map. Indeed, if $IN=0$, then $N$ does not change if we pull back the situation to the reduced ring $A^+/I$. However, for $A/I$ Nakayama's lemma implies that $E(u)^{k_\mu}\coker \Phi$ vanishes if it vanishes after all possible pullbacks $\Spf\,\Ocal\hookrightarrow \tilde\Xcal$. We already remarked above that in the case $A^+=\Ocal$ a finite flat $\Z_p$-algebra the cokernel is killed by $E(u)^{k_\mu}$.

We now show that $IN=0$. For some $k\gg 0$ we know that $I^k\otimes_{A^+} N\rightarrow N$ is the zero map, as $I$ is nilpotent. Then $N=N/I^k$ and the multiplication map $I^{k-1}\otimes_{A^+}N\rightarrow N$ factors over $I^{k-1}/I^k\otimes_{A^+} N\rightarrow N$ and this is a map of finitely generated $A^+/I^k$-modules which vanishes after pulling back to a quotient $A^+/I^k\rightarrow \Ocal_L$ onto the ring of integers in some finite extension $L$ of $\Q_p$. This can be seen as follows. The map on this pull back is induced by the pullback of the multiplication to a quotient of $A^+$ which is finite flat over $\Z_p$, and where $N$ is known to vanish by the above.
It follows that $I^{k-1}\otimes_{A^+}N\rightarrow N$ is the zero map and by descending induction we find that $I$ acts trivially on $N$.
\end{proof}

\begin{proof}[Proof of Theorem \ref{etale}]
Fix some $r>|\pi|$ and re-define 
\[X^{\rm int}=\left\{x\in X\left |
\begin{array}{*{20}c}
\Mcal|_{X\times\boldB_{[r,r_2]}}\otimes \kappa(x)\ \text{contains an}\ \Acal_X^{[r,r_2]}\otimes \kappa(x)^\circ \ \text{lattice}\ N_x\\ \text{such that}\ \Phi\ \text{induces an isomorphism}\\ \phi^\ast(N_x\otimes_{\sA_X^{[r,r_2]}}\sA_X^{[r,r_1]})\isoto N_x\otimes_{\sA_X^{[r,r_2]}}\sA_X^{[r_1,r_2]}
\end{array}\right.\right\}.\]
By Theorem \ref{loket} this subset is open and we need to show that the restriction $\Mcal|_{X^{\rm int}}$ is \'etale. Then it follows directly that $X^{\rm int}$ coincides with the characterization in the theorem, as the notion of being \'etale at points may be checked fpqc-locally by \cite[Proposition 6.14]{families}.

However Proposition \ref{intmodel} provides (locally on $X^{\rm int}$) an integral model $\Mfrak_{[0,r_2]}$ over $X\times \boldB_{[0,r_2]}$. Now we can glue $\Mfrak_{[0,r_2]}$ and $\phi^\ast \Mfrak_{[0,r_2]}$ over $X\times\boldB_{[r_2,r_3]}$ along the isomorphism $\Phi$. Hence we can extend $\Mfrak_{[0,r_2]}$ to a model $\Mfrak_{[0,r_3]}$ over $X\times \boldB_{[0,r_3]}$. Proceeding by induction we get a model $\Mfrak$ on $X\times \Ubb$ and \cite[Proposition 6.5]{families} guarantees that $\Mfrak$ is locally in $X$ free over $\sA_X^{[0,1)}$ (loc.~cit.~assumes that $\Ncal$ is free. However, its proof only uses the fact that the restriction of $\Ncal$ to an annulus $X\times \boldB_{[r, r^{1/p^2}]}$ is free. This is always true after localizing on $X$, see \cite{Luetkebohmert}). Hence it is the desired \'etale model. 
\end{proof}

\begin{cor}\label{Cor4.7}
Let $\mu$ be a cocharacter as in $(\ref{mu})$ with reflex field $E_\mu$. Then there is an open substack $\sH_{\phi,N,\preceq\mu}^{\rm ad, int}\subset \sH_{\phi,N,\preceq\mu}^{\rm ad}$ such that $f:X\rightarrow \sH_{\phi,N,\preceq\mu}^{\rm ad}$ factors over $\sH_{\phi,N,\preceq\mu}^{\rm ad, int}$ if and only if the family $(\Mcal,\Phi_{\Mcal},N_\nabla^\Mcal)$ defined by $f$ and $\underline{\Mcal}$ is \'etale. 
\end{cor}

\begin{proof}
Let $\ul\CM(\ulD)$ be the universal $(\phi,N_\nabla)$-module over $\sH_{\phi,N,\preceq\mu}^{\rm ad}$. By Theorem~\ref{etale} the set $\sH_{\phi,N,\preceq\mu}^{\rm ad, int}:=\{x\in \sH_{\phi,N,\preceq\mu}^{\rm ad}:\ul\CM(\ulD) \text{ is \'etale at }x\}$ is open and above it $\ul\CM(\ulD)$ is \'etale. If $f$ factors over $\sH_{\phi,N,\preceq\mu}^{\rm ad, int}$ then $(\Mcal,\Phi_{\Mcal},N_\nabla^\Mcal)$ is the pullback of the universal $\ul\CM(\ulD)$ and hence is \'etale. Conversely if $(\Mcal,\Phi_{\Mcal},N_\nabla^\Mcal)$ is \'etale, then it is \'etale at all points and $f$ factors over $\sH_{\phi,N,\preceq\mu}^{\rm ad, int}$, because the notion of being \'etale at points may be checked fpqc-locally by \cite[Proposition 6.14]{families}.
\end{proof}

\begin{prop}\label{PropIntWA}
Let $L$ be a finite extension of $E_{\mu}$, then $\sH_{\phi,N,\preceq\mu}^{\rm ad, int}(L)=\sH_{\phi,N,\preceq\mu}^{\rm ad, wa}(L)$ and hence $\sH_{\phi,N,\preceq\mu}^{\rm ad, int}\subset \sH_{\phi,N,\preceq\mu}^{\rm ad, wa}$.
\end{prop}
\begin{proof}
We show that being weakly admissible translates into being pure of slope zero over the Robba ring (in the sense of \cite{Kedlaya}) under the equivalence of categories from Theorem~\ref{ThmEquivDandM}.
However, the proof is the same as in \cite[Theorem 1.3.8]{crysrep}. One easily verifies that the functor $\underline{\Mcal}$ preserves the slope and that the slope filtration on the base change of $\underline\Mcal(D,\Phi,N,\Fq)$ to the Robba ring extends to all of $\underline\Mcal(D,\Phi,N,\Fq)$. Compare \cite[Proposition 1.3.7]{crysrep}.

As in \cite[Theorem 7.6 (ii)]{families} the second part is now a consequence of the fact that $\sH_{\phi,N,\preceq\mu}^{\rm ad, wa}\subset \sH_{\phi,N,\preceq\mu}^{\rm ad}$ is the maximal open subspace whose rigid analytic points are exactly the weakly admissible ones, see Lemma \ref{maxmod}.
\end{proof}

Pappas and Rapoport \cite[5.b]{phimod} define a \emph{period morphism} from a stack of integral data to a stack of filtered $\phi$-modules as follows. 
Let $d>0$ and let $\mu:\Gbb_{m, \ol\Q_p}\to \wt T_{\ol\Q_p}$ be a cocharacter as in \eqref{mu}. 
Pappas and Rapoport \cite[3.d]{phimod} define an fpqc-stack $\hat\Ccal_{\mu,K}$ on the category ${\rm Nil}_{\CO_{E_\mu}}$ of schemes over the ring of integers $\CO_{E_\mu}$ of $E_\mu$ on which $p$ is locally nilpotent. If $R$ is an $\CO_{E_\mu}$-algebra, we set $R_W=R\otimes_{\Z_p}W$ and denote by $\phi:R_W\dpl u\dpr\rightarrow R_W\dpl u\dpr$ the ring homomorphism that is the identity on $R$, the $p$-Frobenius on $W$ and that maps $u$ to $u^p$. Now the $R$-valued points of the stack $\hat\Ccal_{\mu,K}$ are given by a subset
\[\hat\Ccal_{\mu,K}(R)\subset\{\Mfrak,\;\Phi\colon\phi^\ast \Mfrak[1/u]\isoto\Mfrak[1/u]\}\]
where $\Mfrak$ is an $R_W\dbl u\dbr=(R\otimes_{\Z_p}W)\dbl u\dbr$-module that is fpqc-locally on $\Spec R$ free as an $R_W\dbl u\dbr$-module of rank $d$.  This subset is cut out by a condition prescribing the relative position of $\Phi(\phi^\ast\Mfrak)$ with respect to $\Mfrak$ at the locus $E(u)=0$ in terms of the cocharacter $\mu$, see \cite[3.c,d]{phimod} for the precise definition.

If $\mu$ is minuscule they define a \emph{period map}
\[ \Pi(\Xcal):\hat\Ccal_{\mu,K}(\Xcal)\longrightarrow \sD^{\rm ad}_{\phi,\mu}(\Xcal^{\rm rig}),
\]
see \cite[(5.37)]{phimod}. Note that $\hat\Ccal_{\mu,K}$ is a substack of $\hat\Ccal_{d,K}$ of loc.cit.~if and only if $\mu$ is minuscule. Moreover, the period morphism of loc. cit.~maps the closed substack $\hat\Ccal_{\mu,K}$ to the corresponding closed substack $\sD^{\rm ad}_{\phi,\mu}$ of their target $\sD_{d,K}$.

If $\mu$ is not miniscule we can not hope for a period map with target $ \sD^{\rm ad}_{\phi,\mu}$. 
However, if we replace the target by  $\sH^{\rm ad}_{\phi,\preceq\mu}$, then we can again define a period map as follows (note that $ \sD^{\rm ad}_{\phi,\mu}= \sH^{\rm ad}_{\phi,\preceq\mu}$ if $\mu$ is miniscule). Let $R$ be a $p$-adically complete $\CO_{E_\mu}$-algebra topologically of finite type over $\CO_{E_\mu}$ and let $(\Mfrak,\Phi)\in\hat\Ccal_{\mu,K}(\Spf R)$.  
The construction of section \ref{VBonopenunitdisc} associates to
 \begin{equation}\label{vbonU}
 (\Mcal,\Phi_\Mcal)=(\Mfrak,\Phi)\otimes_{R_W\dbl u\dbr}\sB^{[0,1)}_{\Spa(R[1/p],R)}
 \end{equation}
 a $\phi$-module with Hodge-Pink lattice over $\Spa(R[1/p],R)$. Given a formal scheme $\Xcal$ locally topologically of finite type over $\CO_{E_\mu}$, this yields a period functor
 \begin{equation}\label{periodmap}
 \Pi(\Xcal):\hat\Ccal_{\mu,K}(\Xcal)\longrightarrow \sH^{\rm ad}_{\phi,\preceq\mu}(\Xcal^{\rm rig}) , 
 \end{equation}
 where $\Xcal^{\rm rig}$ denotes the generic fiber of the formal scheme $\Xcal$ in the sense of rigid geometry (or in the sense of adic spaces).
We point out that we can not define a period map mapping to $\sD^{\ad}_{\phi,\mu}$ if $\mu$ is not miniscule, as the family of vector bundles on the open unit disc defined by $(\ref{vbonU})$ is not necessarily associated to a filtered $\phi$-module: the monodromy operator $N^\Mcal_\nabla$ is not necessarily holomorphic. 
When $\CX=\Spf \CO_L$ for a finite field extension $L$ of $E_\mu$, it was shown by Genestier and Lafforgue \cite[Th\'eor\`eme~0.6]{GL12} that $\Pi(\Spf \CO_L)\otimes_{\BZ_p}\BQ_p$ is fully faithful, and surjective onto $\sH^{\rm ad,wa}_{\phi,\preceq\mu}(L)=\sH_{\phi,\preceq\mu}^{\rm ad, int}(L)$.

  \begin{rem}\label{periodandN}
 From the point of view of Galois representations it is not surprising that we can not define a general period morphism using filtered $\phi$-modules. If $R$ is finite over $\CO_{E_\mu}$, then the points of $\hat\Ccal_{\mu,K}(R)$ correspond to $\sG_{K_\infty}$-representations rather than to $\sG_K$-representations. This also explains why the target of the period map is $\mathscr{H}^\ad_{\phi,\preceq\mu}$ instead of $\mathscr{H}^{\ad}_{\phi,N,\preceq\mu}$: the $\sG_{K_\infty}$-representation does not see the monodromy. 
 \end{rem}

If we want to take the monodromy into account we have to consider a stack $\hat \Ccal_{ \mu, N, K}$ whose $\Xcal$-valued points are given by $(\Mfrak,\Phi,N)$ with $(\Mfrak,\Phi)\in \hat\Ccal_{ \mu,K}(\Xcal)$ and $N:\Mfrak/u\Mfrak\rightarrow \Mfrak/u\Mfrak$ satisfying 
 \begin{equation}\label{NPhi=pPhiN}
 N\circ \ol\Phi(\ddd)=p\cdot\ol\Phi(\ddd)\circ N.
 \end{equation}
 Here $(\Mfrak(n),\Phi(n))=(\Mfrak,\Phi)\otimes_{W\dbl u\dbr} \bA^{[0,1)}(n)$ is the twist of $(\Mfrak,\Phi)$ with the object $\bA^{[0,1)}(n)$ defined before Definition~\ref{DefEtale} 
and $n\gg 0$ is some integer such that $\Phi(\ddd)(\phi^\ast\Mfrak)\subset \Mfrak$ and $\ol\Phi$ denotes the reduction of $\Phi$ modulo $u$.
  Note that given $\mu$ we may choose an $\ddd$ like that for all $(\Mfrak,\Phi)\in \hat\Ccal_{\mu,K}(\Xcal)$ and the map $\ol\Phi$ (and hence the equation $(\ref{NPhi=pPhiN})$) makes sense after this twist. Further the condition defined by $(\ref{NPhi=pPhiN})$ is independent of the chosen $\ddd$. 
 \begin{rem}
 \noindent (i) Using \eqref{EqDefHodgeWtsB} we observe that if $\mu_{\psi,d}\ge 0$ for all $\psi$, and if $L$ is a finite extension of $E_\mu$, a $\Spf \Ocal_L$-valued point of the stack $ \hat\Ccal_{\mu,N,K}$ gives rise to an object of the category ${\rm Mod}_{/\mathfrak{S}}^{\phi,N}$ in the sense of Kisin \cite[(1.3.12)]{crysrep}. We only use the twist in order to define the stack in the general case (i.e. if $\Phi(\phi^\ast\Mfrak)$ is not contained in $\Mfrak$). Kisin's definition takes place in the generic fiber. However, we can not use this as a good definition as our stack is defined for $p$-power torsion objects. 
 
\smallskip \noindent 
(ii) Note that we do not know much about the stack $\hat\Ccal_{\mu,N,K}$ and its definition is rather ad hoc. Especially we doubt that it is flat over $\Spf \Z_p$. 
 This means that  there is no reason to expect that we can reconstruct Kisin's semi-stable deformation rings \cite{Kisindeform} by using a similar construction as in \cite[\S\,4]{phimod}.
  \end{rem}
 In this general case described above we obtain a similar period morphism
 \begin{equation}\label{periodmorphwithN}
 \hat\Ccal_{\mu,N,K}(\Xcal)\longrightarrow \sH^{\rm ad}_{\phi, N ,\preceq\mu}(\Xcal^{\rm rig}).
 \end{equation}

As in \cite[Theorem 7.8]{families} the above allows us to determine the image of the period morphism. 
Recall that a valued field $(L,v_L)$ over $\Q_p$ is called \emph{of $p$-adic type} if it is complete, topologically finitely generated over $\Q_p$ and if for all $f_1,\dots, f_m\in L$ the closure of $\Q_p[f_1\dots, f_m]$ inside $L$ is a Tate algebra, i.e.\ the quotient of some $\Q_p\langle T_1,\dots, T_{m'}\rangle$.

\begin{cor}\label{CorImagePeriodMap}
The substack $\sH^{\rm ad, int}_{\phi,N,\preceq\mu}$ is the image of the period morphism $(\ref{periodmorphwithN})$ in the following sense:\\
\noindent {\rm (i)} If $\Xcal$ is a $p$-adic formal scheme and $(\Mfrak,\Phi,N)\in \hat\Ccal_{\mu,N,K}(X)$, then $\Pi(\Xcal)(\Mfrak,\Phi,N)\in \sH^{\rm ad, int}_{\phi,N,\preceq\mu}(\Xcal^{\rm rig})$.\\
\noindent {\rm (ii)} Let $L$ be a field of $p$-adic type over $E_\mu$ and $(D,\Phi,N,\q)\in \sH_{\phi,N,\preceq\mu}(L)$. Then there exists $(\Mfrak,\Phi,N)\in \hat\Ccal_{\mu,N,K}(\Spf\, L^+)$ such that $\Pi(\Spf\,L^+)(\Mfrak,\Phi,N)=(D,\Phi,N,\q)$ if and only if
\[\underline{\Mcal}(D)=\Mfrak\otimes_{L^+_W\dbl u\dbr}\sB_L^{[0,1)}. \]
is \'etale, if and only if  $\Spa(L,L^+)\rightarrow \sH^{\rm ad}_{\phi,N,\preceq\mu}$ factors over $\sH_{\phi,N,\preceq\mu}^{\rm ad, int}$.\\
\noindent {\rm (iii)} Let $X\in \Ad^{\rm lft}_{E_\mu}$ and let $f:X\rightarrow \sH_{\phi,N,\preceq\mu}^{\ad}$ be a morphism defined by $(D,\Phi,N,\q)$. Then $f$ factors over $\sH^{\rm ad, int}_{\phi,N,\preceq\mu}$ if and only if there exists a fpqc-covering $(U_i\rightarrow X)_{i\in I}$  and formal models $\Ucal_i$ of $U_i$ together with $(\Mfrak_i,\Phi_i,N)\in \hat\Ccal_{\mu,N,K}(\Ucal_i)$ such that $\Pi(\Ucal_i)(\Mfrak_i,\Phi_i,N)=(D,\Phi,N,\q)|_{U_i}$. 
\end{cor}
\begin{rem}
If we consider the period morphism without monodromy, then we obtain a similar characterization of the stack $\sH^{\rm ad, int}_{\phi,\preceq \mu}\subset \sH^{\rm ad}_{\phi,\preceq \mu}$ as the image of the period morphism $(\ref{periodmap})$.
\end{rem}
\section{Sheaves of period rings and the admissible locus}
We recall the definition of some sheafified period rings from \cite{families}. In doing so we will also correct mistakes in loc.~cit.~(in particular the proofs of Corollary 8.8, the definition of a family of crystalline representations, and the proof of Proposition 8.24 in \cite{families}).
 
Let $R=\lim\limits_{\longleftarrow}\,\Ocal_{\C_p}/p\Ocal_{\C_p}$ be the inverse limit with transition maps given by the $p$-th power. Given a 
reduced $p$-adically complete $\Z_p$-algebra $A^+$ topologically of finite type, we define 
\[A^+\wh\otimes_{\Z_p} W(R)=\lim_{\longleftarrow\,i} A^+\wh\otimes_{\Z_p}W_i(R),\]
where the completed tensor product on the right hand side means completion with respect to the canonical topology on the truncated Witt vectors $W_i(R)$ and the discrete topology on $A^+/p^iA^+$.

If $X$ is a reduced adic space locally of finite type over $\Q_p$, then there are sheaves $\Ocal_X^+\wh\otimes W(R)$ and $\Ocal_X\wh\otimes W(R)$ whose sections over an affinoid open $U=\Spa(A,A^+)\subset X$ are given by
\begin{align*}
\Gamma(U,\Ocal_X^+\wh\otimes W(R))&=A^+\wh\otimes_{\Z_p} W(R)\\
\Gamma(U,\Ocal_X\wh\otimes W(R))&=\big (A^+\wh\otimes_{\Z_p} W(R)\big)[\tfrac{1}{p}].
\end{align*} 
In the same fashion we can define sheaves of topological rings $\Ocal_X^+\wh\otimes W(\Frac R)$ and $\Ocal_X\wh\otimes W(\Frac R)$.

Let $\bA^{[0,1)}=W\dbl u\dbr$ and let $\bA$ denote the $p$-adic completion of $W\dpl u\dpr$.  Further let $\bB=\bA[1/p]$. We fix an element $\pi^\flat=(\pi_n)_n\in R$ with $\pi_0=\pi$. Depending on this element there are embeddings of $\bA^{[0,1)}$, $\bA$ and $\bB$ into $W({\rm Frac}\,R)[1/p]$ sending $u$ to the Teichm\"uller representative $[\pi^\flat]\in W(R)$ of $\pi^\flat$.
We write $\wt\bA$ for the ring of integers in the completion $\wt\bB$ of the maximal unramified extension of $\bB$ inside $W({\rm Frac}\,R)[1/p]$. Finally we set $\wt\bA^{[0,1)}=\wt\bA\cap W(R)\subset W({\rm Frac}\, R)$.
All these rings come along with a Frobenius endomorphism $\phi$ which is induced by the canonical Frobenius on $W({\rm Frac}\,R)$. 
Note that all these rings have a canonical topology induced from the one on $W({\rm Frac}\,R)$. 

\begin{rem}
We warn (and apologize to) the reader that the notations used in this paragraph do often not agree with the notations that are nowadays standard in $p$-adic Hodge theory. However, we often refer to \cite{families} and it seems to cause less confusion using the notations used there.
\end{rem}

We define sheafified versions of these rings as follows, compare \cite[8.1]{families}. Let $X$ be a reduced adic space locally of finite type over $\Q_p$.  We define the sheaves $\sA_X$, resp.\ $\wtsA_X$, resp.\ $\sA_X^{[0,1)}$, resp.\ $\wtsA_X^{[0,1)}$ by specifying their sections on open affinoids $U=\Spa(A,A^+)\subset X$: we define $\Gamma(U,\sA_X)$, resp.\ $\Gamma(U,\wtsA_X)$, resp.\ $\Gamma(U,\sA_X^{[0,1)})$, resp.\ $\Gamma(U,\wtsA_X^{[0,1)})$ to be the closure (with respect to the natural, i.e.~$(p, [{\pi}^\flat])$-adic, topology) of $A^+\otimes_{\Z_p}\bA$, resp.\ $A^+\otimes_{\Z_p}\wt\bA$, resp.\ $A^+\otimes_{\Z_p}\bA^{[0,1)}$, resp.\ $A^+\otimes_{\Z_p}\wt\bA^{[0,1)}$ in $\Gamma(U,\Ocal_X^+\wh\otimes W({\rm Frac}\,R))$.

Further we consider the rational analogues $\sB_X$, resp.\ $\wt\sB_X$, resp.\ $\sB_X^{[0,1)}$, resp.\ $\wt\sB_X^{[0,1)}$ of these sheaves given by inverting $p$ in $\sA_X$, resp.\ $\wtsA_X$, resp.\ $\sA_X^{[0,1)}$, resp.\ $\wtsA_X^{[0,1)}$. 

Finally we recall the construction of the sheaf $\Ocal_X\wh\otimes B_{\rm cris}$ from \cite[8.1]{families}. For a reduced adic space $X$ the map $\theta:W(R)\rightarrow \Ocal_{\mathbb{C}_p}$ given by $[(x,x^{1/p},x^{1/p^2},\dots)]\mapsto x$ extends to an $\Ocal_X^+$-linear map 
\[\theta_X:\Ocal_X^+\wh\otimes W(R)\rightarrow \Ocal_X^+\wh\otimes \Ocal_{\mathbb{C}_p},\]
where the completed tensor product denotes the $p$-adic completion. We define $\Ocal_X^+\wh\otimes A_{\rm cris}$ to be the $p$-adic completion of the divided power envelope of $\Ocal_X^+\wh\otimes W(R)$ with respect to the kernel of $\theta_X$. We claim that $\Ocal_X^+\wh\otimes A_{\rm cris}$ equals the $p$-adic completion of the tensor product $\Ocal_X^+\otimes_{\BZ_p} A_{\rm cris}$. Namely, the kernel of $\theta_X$ is generated by the kernel of $\theta$. The latter in turn is generated by $p-[p^\flat]$, where $p^\flat=(x_n)_n\in R$ with $x_0=p$ and $[\,.\,]$ denotes the Teichm\"uller lift. Therefore, the divided power envelope is constructed by adjoining $(p-[p^\flat])^n/n!$ for all $n\in\BN$, and this proves our claim. Finally we set
\begin{align*}
\Ocal_X\wh\otimes B^+_{\rm cris}&=\big(\Ocal_X^+\wh\otimes A_{\rm cris}\big)[1/p],
\\
\Ocal_X\wh\otimes B_{\rm cris}&=\big(\Ocal_X\wh\otimes B^+_{\rm cris}\big)[1/t], 
\\
\Ocal_X\wh\otimes B_{\rm st}^+&=\big(\Ocal_X\wh\otimes B_{\rm cris}^+\big)[\ell_u],
\\
\Ocal_X\wh\otimes B_{\rm st}&=\big(\Ocal_X\wh\otimes B_{\rm cris}\big)[\ell_u].
\end{align*}
Here $t={\rm log}[(1,\epsilon_1,\epsilon_2,\dots )]\in B_{\rm cris}$ is the period of the cyclotomic character (where $(\epsilon_i)$ is a compatible system of $p^i$-th roots of unity) and $\ell_u$ is an indeterminate thought of as a formal logarithm of $[\pi^\flat]$.
\begin{rem} 
The indeterminate $\ell_u$ considered here is the same indeterminate as in section 2.2.(b) and we identify both indeterminates. That is, the inclusion $\mathbf{B}^{[0,1)}\subset B^+_{\rm cris}$ given by $u\mapsto [\pi^\flat]$ will be extended to $\mathbf{B}^{[0,1)}[\ell_u]\hookrightarrow B^+_{\rm st}$ by means of $\ell_u\mapsto \ell_u$ and similarly for the sheafified versions. 
\end{rem}

\begin{lem}\label{finitebc}
Let $Y=\Spa(B,B^+)$ be an reduced adic space that is finite over $X=\Spa(A,A^+)$. Then we have canonical isomorphisms 
\begin{align*}
 \wt\sB_Y &\cong  \wt\sB_X\otimes_{\Ocal_X}\Ocal_Y & \sB_Y^{[0,1)}&\cong \sB_X^{[0,1)}\otimes_{\Ocal_X}\Ocal_Y \\
\Ocal_Y\wh\otimes B^+_{\rm cris}&\cong (\Ocal_X\wh\otimes B^+_{\rm cris})\otimes_{\Ocal_X}\Ocal_Y  & \Ocal_Y\wh\otimes B_{\rm cris}&\cong (\Ocal_X\wh\otimes B_{\rm cris})\otimes_{\Ocal_X}\Ocal_Y \\
 \Ocal_Y\wh\otimes B^+_{\rm st}&\cong (\Ocal_X\wh\otimes B^+_{\rm st})\otimes_{\Ocal_X}\Ocal_Y  & \Ocal_Y\wh\otimes B_{\rm st}&\cong (\Ocal_X\wh\otimes B_{\rm st})\otimes_{\Ocal_X}\Ocal_Y
\end{align*} 
\end{lem}
\begin{proof}
This is a direct consequence of the construction (and the fact that we do not have to complete tensor products with finitely generated modules). 
\end{proof}

We can consider these sheaves also on non-reduced spaces by locally embedding $X$ into a reduced space $Y$ and restricting the corresponding sheaves from $Y$ to $X$, i.e.~by applying $-\otimes_{\Ocal_Y}\Ocal_X$. Thanks to the above lemma, the sheaves like $\Ocal_X\wh\otimes B_{\rm cris}^+$ then do not depend on the choice of an embedding.
With this definition the claim of Lemma \ref{finitebc} also holds true for non-reduced adic spaces. 
\begin{rem}
For non-reduced spaces we make this slightly involved definition for the following reason: the construction of rings like $A_{\rm cris}$ involves a $p$-adic completion. But the rings of integral elements $A^+$ for an adic space $\Spa(A,A^+)$ (i.e.~the power bounded elements in $A$) are not $p$-adically complete: their $p$-adic completion would kill the nilpotent elements!
\end{rem}

On $\Ocal_X\wh\otimes B_{\rm cris}$ there is a canonical Frobenius $\phi$ induced by the Frobenius on $\Ocal_X^+\wh\otimes W(R)$. 
This endomorphism extends to a morphism 
\[\phi:\Ocal_X\wh\otimes B_{\rm st}\longrightarrow \Ocal_X\wh\otimes B_{\rm st},\]
where $\phi(\ell_u)=p\ell_u$. Further $N=\tfrac{d}{d\ell_u}$ defines an endomorphism of $\Ocal_X\wh\otimes B_{\rm st}$ which satisfies $N\phi=p\phi N$. 

Finally the continuous $\sG_K$-action on $\Ocal_X^+\wh\otimes W(R)$ extends to $\Ocal_X\wh\otimes B_{\rm cris}$ and we further extend this action to $\Ocal_X\wh\otimes B_{\rm st}$ by means of $\gamma\cdot \ell_u= \ell_u+c(\gamma)t$, where $c:\sG_K\rightarrow \Z_p$ is defined by $\gamma(\pi_n)=\pi_n\cdot(\epsilon_n)^{c(\gamma)}$ for all $n\ge0$.

\begin{lemma}\label{LemmaPowerSeriesExp}
Let $Y=\Spa(A,A^+)$ be an adic space locally of finite type over $\BQ_p$. 
\begin{enumerate}
\item \label{LemmaPowerSeriesExp_B}
Let $g\in \Gamma(Y,\Ocal_Y\wh\otimes B_{\rm cris}^+)$. Then $g\in \Gamma(Y,\CO_Y)\subset\Gamma(Y,\Ocal_Y\wh\otimes B_{\rm cris}^+)$ if and only if for every quotient $A\twoheadrightarrow A'$ onto a finite dimensional $\BQ_p$-algebra $A'$ the element 
\[
g\otimes 1 \in \Gamma(Y,\Ocal_Y\wh\otimes B_{\rm cris}^+)\otimes_A A'\cong A'\otimes B_{\rm cris}^+
\]
actually lies in $A'\subset A'\otimes_{\Q_p}B_{\rm cris}^+$.
\item \label{LemmaPowerSeriesExp_A}
Let $g\in \Gamma(Y,\wt\sB_Y)$. Then $g\in \Gamma(Y,\CO_Y)\subset\Gamma(Y,\wt\sB_Y)$ if and only if for every quotient $A\twoheadrightarrow A'$ onto a finite dimensional $\BQ_p$-algebra $A'$ the element
\[g\otimes 1\in \Gamma(\Spa(A',A'^+),\wt\sB_Y\otimes_A A')\cong A'\otimes_{\Q_p}\wt\bB\]
actually lies in $A'\subset A'\otimes_{\Q_p}\wt\bB$.
\item \label{LemmaPowerSeriesExp_C}
Assume that $A$ is reduced. Let $g\in \Gamma(Y,\wtsA^{[0,1)}_Y)$. Then $g\in \Gamma(Y,\sA^{[0,1)}_Y)\subset\Gamma(Y,\wtsA^{[0,1)}_Y)$ if and only if for every rigid analytic point $y\in Y$ the element $g(y):=g\otimes_{A^+}\kappa(y)^+\in \kappa(y)^+\wh\otimes_{\Z_p}\wt\bA^{[0,1)}$ actually lies in $\kappa(y)^+\wh\otimes_{\Z_p}\bA^{[0,1)}\subset \kappa(y)^+\wh\otimes_{\Z_p}\wt\bA^{[0,1)}$.
\end{enumerate}
\end{lemma}
Note that the identifications
\[
\Gamma(\Spa(A',A'^+),(\wt\sB_Y)\otimes_A A')\cong A'\otimes_{\Q_p}\wt\bB\qquad \text{and}\qquad \Gamma(Y,\Ocal_Y\wh\otimes B_{\rm cris}^+)\otimes_A A'\cong A'\otimes B_{\rm cris}^+
\] used in the formulation of the lemma are a direct consequence of Lemma \ref{finitebc} and the remark following it.

\begin{proof}
\ref{LemmaPowerSeriesExp_B} 
Clearly the condition is necessary. We now show that it is sufficient. Let us choose a closed immersion $Y=\Spa(A,A^+)\into X=\Spa(C,C^+)$ with $C$ a reduced Tate ring topologically of finite type over $\Q_p$. Then $C\onto A$ and our definitions imply that $A\wh\otimes B_{\rm cris}^+$ is the quotient of $C\wh\otimes B_{\rm cris}^+$ by the kernel of $C\to A$. We choose elements $b_i\in A_{\rm cris}$ with $b_0=1$ whose images $\bar b_i$ in $A_{\rm cris}/pA_{\rm cris}$ form an $\BF_p$-basis of $A_{\rm cris}/pA_{\rm cris}$. Recall that we remarked after the definition of $\Gamma(X,\Ocal_X\wh\otimes A_{\rm cris})$ that it equals the $p$-adic completion $C^+\wh\otimes_{\BZ_p}A_{\rm cris}$ of $C^+\otimes_{\BZ_p}A_{\rm cris}$. We start with the following

\medskip\noindent
{\itshape Claim.}
For every element $c\in C^+\wh\otimes_{\BZ_p}A_{\rm cris}$ there are uniquely determined elements $a_i\in C^+$ for $i\in I$ such that for every $n\in\BN$ the set $\{\,i\in I\colon a_i\notin p^n C^+\,\}$ is finite and $c=\sum_{i\in I}a_i\otimes b_i$ in $C^+\wh\otimes_{\BZ_p}A_{\rm cris}$.

\medskip\noindent
To establish the claim one proves by induction that for every $n$ there are elements $a_{i,n}\in C^+$ for all $i\in I$, only finitely many of which are non-zero, such that $c-\sum_{i\in I}a_{i,n}\otimes b_i\in p^n C^+\wh\otimes_{\BZ_p}A_{\rm cris}$ and such that $a_{i,n}-a_{i,n-1}\in p^{n-1} C^+$. Namely, for $n=0$ one can take $a_{i,0}=0$ for all $i\in I$. In the induction step from $n$ to $n+1$ one considers an element $c'\in C^+\wh\otimes_{\BZ_p}A_{\rm cris}$ with $c-\sum_{i\in I}a_{i,n}\otimes b_i=p^n c'$. Then the image of $c'$ in $C^+\wh\otimes_{\BZ_p}A_{\rm cris}/(p)=C^+/p C^+\otimes_{\BF_p} A_{\rm cris}/p A_{\rm cris}$ can be written as $\sum_i\bar\alpha_i\otimes \bar b_i$ with uniquely determined elements $\bar\alpha_i\in C^+/p C^+$ which are zero for all but finitely many $i$. After choosing lifts $\alpha_i\in C^+$, the elements $a_{i,n+1}:=a_{i,n}+p^n\alpha_i$ satisfy the assertion. Now taking $a_i$ as the limit of $a_{i,n}$ for $n\to\infty$ establishes the existence of the $a_i\in C^+$.

To prove the uniqueness, we must show that $\sum_{i\in I}a_i\otimes b_i=0$ implies $a_i=0$ for all $i$. It suffices to show that $a_i\in p^n C^+$ for all $n$ and $i$. This follows by induction on $n$, trivially starting with $n=0$. If it holds for some $n$, we can write $a_i=p^na'_i$ for $a'_i\in C^+$. Then $p^n\cdot\sum_ia'_i\otimes b_i=\sum_i a_i\otimes b_i=0$, and hence $\sum_ia'_i\otimes b_i=0$, because $C^+\wh\otimes_{\BZ_p}A_{\rm cris}$ has no $p$-torsion by \cite[Chapitre~III, \S\,5, no.~2, Th\'eor\`eme~1(v)]{BourbakiAlgCom} as $C^+$ and $A_{\rm cris}$ are flat over $\BZ_p$. Considering the images $\bar a'_i$ of $a'_i$ in $C^+/p C^+$, the equation $\sum_i\bar a'_i\otimes\bar b_i=0$ in $C^+\wh\otimes_{\BZ_p}A_{\rm cris}/(p)=C^+/p C^+\otimes_{\BF_p} A_{\rm cris}/p A_{\rm cris}$ implies that $\bar a'_i=0$ in $C^+/p C^+$, whence $a'_i\in p C^+$ and $a_i\in p^{n+1}C^+$ as desired. This establishes our claim.

Furthermore we note that this claim (and in particular the uniqueness part) also applies if we replace $C^+$ by a finite free $\Z_p$-algebra (that is not necessarily reduced). 
\smallskip

We lift $g$ to an element $\tilde g\in C\wh\otimes_{\BQ_p}A_{\rm cris}[1/p]$. After multiplying with a power of $p$ we can assume that $\tilde g\in C^+\wh\otimes_{\BZ_p}A_{\rm cris}$. By the claim we obtain uniquely determined elements $a_i\in C^+$ for all $i\in I$ with $\tilde g=\sum_i a_i\otimes b_i$ in $C^+\wh\otimes_{\BZ_p}A_{\rm cris}$. We show that $a_i\in {\rm ker}(C\rightarrow A)$ for all $i\ne0$ which obviously implies $g\in \Gamma(Y,\CO_Y^+)=A^+$.
\smallskip

As $C$ is noetherian the latter may be checked at completions $\hat C_{\mathfrak{m}}$ of $C$ with respect to maximal ideals $\mathfrak{m}$ of $C$. If the point defined by $\mathfrak{m}$ is not in $\Spa(A,A^+)\subset \Spa(C,C^+)$ this claim is obvious. Otherwise we consider the surjections $C\twoheadrightarrow A\twoheadrightarrow A/\mathfrak{m}^nA=A'$ onto the finite dimensional $\Q_p$-algebra $A'$, and let $A'^+$ denote the image of $C^+$ in $A'$. Then $A'^+$ is a finite $\Z_p$-algebra and we write $\bar a_i\in A'^+$ for the image of $a_i$. By what we noted above the expansion $\bar g=\sum \bar a_i\otimes b_i\in A'^+\hat\otimes_{\Z_p}A_{\rm cris}=A'^+\otimes_{\Z_p}A_{\rm cris}$ is unique and by assumption lies in $A'^+\subset A'^+\otimes_{\Z_p}A_{\rm cris}$. It follows that $\bar a_i=0$ for all $i\neq 0$.  
We have shown that $a_i$ for $i\neq 0$ vanishes in $A'=A/\mathfrak{m}^n$ for all $n$ and the $a_i$ for $i\neq 0$ vanish in $A_\mathfrak{m}$. 

\bigskip\noindent
\ref{LemmaPowerSeriesExp_A}, \ref{LemmaPowerSeriesExp_C} 
We denote the residue field of $W$ by $k$ and let $k'$ be either $k$ for proving \ref{LemmaPowerSeriesExp_C} or $\BF_p$ for proving \ref{LemmaPowerSeriesExp_A}. We view the residue field $k\dpl u\dpr^\sep=\wt\bA/p\wt\bA$ of $\wt\bA$ as a $k'\dpl u\dpr$-vector space. We denote the integral closure of $k'\dbl u\dbr$ in $k\dpl u\dpr^\sep$ by $k\dbl u\dbr^\sep$. It is a free $k'\dbl u\dbr$-module: we can write $k\dpl u\dpr^\sep$ as union of finite extensions $E_i$ of $k'\dpl u\dpr$, where $E_i\subset E_{i+1}$, then $k\dbl u\dbr^\sep$ is the increasing union of the rings of integers $\mathcal{O}_{E_i}$ which are free, and $\Ocal_{E_i}$ is a direct summand of $\Ocal_{E_{i+1}}$. Choosing the basis successively yields a basis for $k\dbl u\dbr^\sep$. 

We choose a $k'\dbl u\dbr$-basis $(\bar g_i)_{i\in I}$ of $k\dbl u\dbr^\sep$ with $\bar g_0=1$ and we lift the $\bar g_i$ to elements $g_i\in\wt\bA^{[0,1)}$ with $g_0=1$. 

We first prove \ref{LemmaPowerSeriesExp_C} and use $k'=k$. The image of $g$ in $\Gamma(Y,\wtsA^{[0,1)}_Y)/(p)=(A^+/pA^+\otimes_{\BF_p}k)\otimes_{k}k\dbl u\dbr^\sep$ can be written as $\sum_i\sum_{j=0}^\infty\bar\alpha_{i,j,0}\otimes u^j\bar g_i$ with uniquely determined elements $\bar\alpha_{i,j,0}\in A^+/pA^+\otimes_{\BF_p}k$ which are non-zero only for finitely many $i$ but possibly for all $j\ge0$. After choosing lifts $\alpha_{i,j,0}\in A^+\otimes_{\BZ_p}W$, the image of $\tfrac{1}{p}\cdot(g-\sum_{i,j}\alpha_{i,j,0}\otimes u^j g_i)$ in $\Gamma(Y,\wtsA^{[0,1)}_Y)/(p)$ can likewise be written as $\sum_{i,j}\bar\alpha_{i,j,1}\otimes u^j\bar g_i$ with uniquely determined elements $\bar\alpha_{i,j,1}\in A^+/pA^+\otimes_{\BF_p}k$. Note for this that $\Gamma(Y,\wtsA^{[0,1)}_Y)$ has no $p$-torsion by \cite[Chapitre~III, \S\,5, no.~2, Th\'eor\`eme~1(v)]{BourbakiAlgCom}, because $\wt\bA^{[0,1)}$ and $A^+$ are flat over $\BZ_p$. Continuing in this way, we obtain elements $\alpha_{i,j}:=\sum_{k=0}^\infty\alpha_{i,j,k}\,p^k\in A^+\otimes_{\BZ_p}W$ such that for every $n\ge1$ the equality $g=\sum_{i,j}\alpha_{i,j}\otimes u^j g_i$ holds in $\Gamma(Y,\wtsA^{[0,1)}_Y)/(p^n)$, although the sum does in general not converge in $\Gamma(Y,\wtsA^{[0,1)}_Y)$. 

The elements $\alpha_{i,j}$ are uniquely determined by $g$ because the equality $g=\sum_{i,j}\alpha_{i,j}\otimes u^j g_i$ in $\Gamma(Y,\wtsA^{[0,1)}_Y)/(p^n)$ shows that the images of $\alpha_{i,j}$ in $A^+\otimes_{\BZ_p}W/(p^n)$ are uniquely determined for every $n$. The uniqueness of the $\alpha_{i,j}$ then follows from the fact that $A^+\otimes_{\BZ_p}W$ is $p$-adically separated. We conclude that the element $g$ lies in $\Gamma(Y,\sA^{[0,1)}_Y)$ if and only if $\alpha_{i,j}=0$ whenever $i\ne0$ or $j<0$.

Now $g\otimes 1\in \kappa(y){}^+\otimes_{\Z_p}\bA^{[0,1)}$ implies that $\alpha_{i,j}\otimes 1=0$ in $\kappa(y){}^+\otimes_{\BZ_p}W$ whenever $i\ne0$ or $j<0$. If this holds for every rigid analytic point $y$, then $\alpha_{i,j}=0$ whenever $i\ne0$ or $j<0$, because $A^+\otimes_{\BZ_p}W$ is reduced. This implies $g\in \Gamma(Y,\sA^{[0,1)}_Y)$. 

\medskip\noindent
\ref{LemmaPowerSeriesExp_A}
Again the condition is necessary and we show that it is sufficient. Let us choose a closed immersion $Y=\Spa(A,A^+)\into X=\Spa(C,C^+)$ with $C$ a reduced Tate ring topologically of finite type over $\Q_p$. Then again our definitions imply that $\Gamma(Y,\wt\sB_Y)$ is the quotient of $\Gamma(X,\wt\sB_X)$ by the kernel of the epimorphism $C\onto A$.
We lift $g$ to an element $\tilde g\in C\wh\otimes_{\BQ_p}\wt\bB$. After multiplying with a power of $p$ we can assume that $\tilde g\in C^+\wh\otimes_{\BZ_p}\wt\bA$, where the complete tensor product denotes completion with respect to the $(p,u)$-adic topology. 

We use the elements $g_i\in \wt\bA^{[0,1)}\subset\wt\bA$ with $g_0=1$ from the proof of \ref{LemmaPowerSeriesExp_C} above (with $k'=\BF_p$), whose residues $\bar g_i\in k\dbl u\dbr^\sep\subset k\dpl u\dpr^\sep$ modulo $p$ form an $\BF_p\dbl u\dbr$-basis of $k\dbl u\dbr^\sep$, and hence also an $\BF_p\dpl u\dpr$-basis of $k\dpl u\dpr^\sep$. Then the image of $\tilde g$ in $C^+\wh\otimes_{\BZ_p}\wt\bA/(p)=C^+/pC^+\otimes_{\BF_p}k\dbl u\dbr^\sep$ can be written as $\sum_{i,j}\bar\alpha_{i,j,0}\otimes u^j\bar g_i$ with uniquely determined elements $\bar\alpha_{i,j,0}\in C^+/pC^+$ which are zero for all but finitely many $i$ and for $j\ll0$. After choosing lifts $\alpha_{i,j,0}\in C^+$, the image of $\tfrac{1}{p}\cdot(g-\sum_{i,j}\alpha_{i,j,0}\otimes u^j g_i)$ in $C^+\wh\otimes_{\BZ_p}\wt\bA/(p)$ can likewise be written as $\sum_{i,j}\bar\alpha_{i,j,1}\otimes u^j\bar g_i$ with uniquely determined elements $\bar\alpha_{i,j,1}\in C^+/pC^+$. Note for this that $C^+\wh\otimes_{\BZ_p}\wt\bA$ has no $p$-torsion by \cite[Chapitre~III, \S\,5, no.~2, Th\'eor\`eme~1(v)]{BourbakiAlgCom}, because $\wt\bA$ and $C^+$ are flat over $\BZ_p$. Continuing in this way, we obtain elements $\alpha_{i,j}:=\sum_{k=0}^\infty\alpha_{i,j,k}\,p^k\in C^+$ such that for every $n\ge1$ the equality $g=\sum_{i,j}\alpha_{i,j}\otimes u^j g_i$ holds in $C^+\wh\otimes_{\BZ_p}\wt\bA/(p^n)$, although the sum does in general not converge in $C^+\wh\otimes_{\BZ_p}\wt\bA$. The elements $\alpha_{i,j}$ are uniquely determined by $g$ by reasoning like in \ref{LemmaPowerSeriesExp_C} above. We conclude that the element $g$ lies in $C^+$ if and only if $\alpha_{i,j}=0$ whenever $(i,j)\ne(0,0)$.

As $C$ is noetherian the latter may be checked at completions $\hat C_{\mathfrak{m}}$ of $C$ with respect to maximal ideals $\mathfrak{m}$ of $C$. If the point defined by $\mathfrak{m}$ is not in $\Spa(A,A^+)\subset \Spa(C,C^+)$ this claim is obvious. Otherwise we consider the surjections $C\twoheadrightarrow A\twoheadrightarrow A/\mathfrak{m}^nA=A'$ onto the finite dimensional $\Q_p$-algebra $A'$. Then our assumptions imply that the image of $a_{i,j}$ in $A'$ vanishes for $(i,j)\ne(0,0)$ by a similar reasoning as above for $A'{}^+$ in place of $C^+$. We have shown that the image of $a_{i,j}$ in $\hat C_\mfrak$ lie in the kernel of $\hat C_\mfrak\rightarrow \hat A_\mfrak$ for all maximal ideals of $C$ and all $(i,j)\ne(0,0)$. The claim follows from this.
\end{proof}
\begin{rem}
Assume that in the situation of Lemma \ref{LemmaPowerSeriesExp} the ring $A$ is reduced. We remark that it is then enough to check the conditions for surjections $A\twoheadrightarrow \kappa(y)$ for all rigid analytic points $y\in Y$. 
We only need to argue (in the situation of the proof above) that $g(y)\in\kappa(y)^+\subset\kappa(y)^+\wh\otimes_{\BZ_p}B$ implies that $ a_i(y)=0$ in $\kappa(y)^+$ for all $i\ne0$ for $B=A_{\rm cris}$, respectively $B=\wt\bA$. Here we write $g(y)=1\otimes g\in \kappa(y)^+\wh\otimes_{\BZ_p}B$ and so on. If this holds for every rigid analytic point $y\in Y$, then $ a_i=0$ for all $i\ne0$, because $Y$ is reduced. This implies $g\in \Gamma(Y,\CO_Y^+)$. 
\end{rem}

\begin{rem}
It is also possible to define $\Z$-filtrations $\Fil^i(\Ocal_X\wh\otimes B_{\rm cris})$ resp.\ $\Fil^i(\Ocal_X\wh\otimes B_{\rm st})$ on $\Ocal_X\wh\otimes B_{\rm cris}$ resp.\ $\Ocal_X\wh\otimes B_{\rm st}$. The most natural procedure seems to be the following:
given $i\in\Z$ and an adic space $X=\Spa(A,A^+)$, a section $f\in \Gamma(X,\Ocal_X\wh\otimes B_{\rm cris})$ lies in $\Gamma(X,\Fil^i(\Ocal_X\wh\otimes B_{\rm cris}))$, if $f\otimes 1\in \Fil^i{B}_{\rm cris}\otimes_{\Q_p}B$ for all surjections $A\twoheadrightarrow B$ of A onto finite dimensional $\Q_p$-algebras $B$. Here $\Fil^i B_{\rm cris}$ is the usual filtration on $B_{\rm cris}$ induced by restricting the  $t$-adic filtration on Fontaine's ring ${B}_{\rm dR}$ to $B_{\rm cris}$.
This construction obviously globalizes and defines a filtration of the sheaf $\Ocal_X\wh\otimes B_{\rm cris}$.
 A similar construction also applies to the filtration on $\Ocal_X\wh\otimes B_{\rm st}$. 
 However some issues with this filtration seem to be a bit involved, in particular dealing with families. One main reason is, that $\Ocal_X\wh\otimes B_{\rm cris}^+$ is much better behaved than $\Ocal_X\wh\otimes B_{\rm cris}$, but $\Fil^0 B_{\rm cris}$ does not give back $B^+_{\rm cris}$.
 Hence we will not consider this filtration on $\Ocal_X\wh\otimes B_{\rm cris}$ explicitly. 
 \end{rem}

\smallskip

\begin{prop} \label{periodsheavesatpoints}
Let $X$ be an adic space locally of finite type over $\Q_p$.
The canonical inclusions induce equalities 
\begin{align*}
\wt\sB_X^{\Phi={\rm id}}&=\Ocal_X\\
\big(\Ocal_X\wh\otimes B_{\rm cris}^+\big)^{ \Phi={\rm id}}&=\Ocal_X \\
\big(\Ocal_X\wh\otimes B_{\rm st}^+\big)^{\Phi={\rm id}, N=0}&=\Ocal_X.
 \end{align*}
Moreover one has
\begin{align*}
(\Ocal_X\hat\otimes B_{\rm cris}^+)^{\mathscr{G}_K}=(\Ocal_X\hat\otimes B_{\rm st}^+)^{\mathscr{G}_K}=\Ocal_X\otimes_{\Q_p}K_0
\end{align*} 
\end{prop}
\begin{proof}
It is clear that in all cases $\Ocal_X$ injects onto the sheaves of invariants, resp.~that $\Ocal_X\otimes_{\Q_p}K_0$ injects into $(\Ocal_X\hat\otimes B_{\rm cris})^{\mathscr{G}_K}$. Let us prove the converse. 
Let $U=\Spa(A,A^+)\subset X$ be an affinoid open and let $f\in \Gamma(U,\wt\sB_X)$ be a section that is invariant under $\Phi$. Then for each quotient $A\twoheadrightarrow A'$ with $A'$ a finite dimensional $\Q_p$ algebra the element $f\otimes 1\in \Gamma(U,\wt\sB_X)\otimes_AA'=A'\otimes_{\Q_p} \wt{\mathbf{B}}$ is invariant under $\Phi$ and hence $f\otimes 1\in A'\subset A'\otimes_{\Q_p} \wt{\mathbf{B}}$. Now Lemma \ref{LemmaPowerSeriesExp} implies $f\in \Gamma(U,\Ocal_X)$.
The other claims are proven using the same argument.
\end{proof}

\begin{defn}
Let $\sG$ denote a compact topological group. A \emph{family of $\sG$-representations} on an adic space $X$ consists of a vector bundle $\Ecal$ on $X$ together with an $\Ocal_X$-linear action of the group $\sG$ on $\Ecal$ which is continuous for the topologies on the sections $\Gamma(-,\Ecal)$. 
This definition extends to the category of stacks on $\Ad_{\Q_p}^{\rm lft}$.
\end{defn}

\begin{defn}
Let $X$ be an adic space locally of finite type over $\Q_p$.\\
\noindent (i) A \emph{$\phi$-module over $\sA_X$} is an $\sA_X$-module $M$ which is locally on $X$ free of finite rank over $\sA_X$ together with an isomorphism $\Phi:\phi^\ast M\isoto M$.\\
\noindent (ii)  A \emph{$\phi$-module over $\sB_X$} is an $\sB_X$-module $M$ which is locally on $X$ free of finite rank over $\sB_X$ together with an isomorphism $\Phi:\phi^\ast M\isoto M$.\\
\noindent (iii) A $\phi$-module $M$ over $\sB_X$ is called \emph{\'etale} if it is locally on $X$ of the form $N\otimes_{\sA_X}\sB_X$ for a $\phi$-module $N$ over $\sA_X$.
\end{defn}

The following theorem summarizes results of \cite{families} which are needed in the sequel. 
\begin{theo}\label{summary}
Let $X$ be a reduced adic space locally of finite type over $\Q_p$ and let $(\Ncal,\Phi)$ be an \'etale $\phi$-module of rank $d$ over $\sB_X$.\\
\noindent {\rm (i)} The set
\[X^{\rm adm}=\{x\in X\mid \dim_{\kappa(x)}((\Ncal\otimes_{\sB_X}\wt\sB_X)\otimes \kappa(x))^{\Phi={\rm id}}=d\}\subset X\]
is an open subspace and 
\[\mathcal{V}=(\Ncal\otimes\wt\sB_X)^{\Phi={\rm id}}\]
is a family of $\sG_{K_{\infty}}$-representations on $X^{\rm adm}$.\\
\noindent {\rm (ii)} If $f:Y\rightarrow X$ is a morphism in ${\rm Ad}^{\rm lft}$ and if $(\Ncal_Y,\Phi_Y)$ denotes the pullback of $(\Ncal,\Phi)$ along $f$, then $Y^{\rm adm}=f^{-1}(X^{\rm adm})$ and 
\[(\Ncal_Y\otimes \wt\sB_Y)^{\rm \Phi={\rm id}} = (f|_{Y^{\rm adm}})^\ast \mathcal{V}\]
as families of $\sG_{K_{\infty}}$-representations on $Y^{\rm adm}$. \\
\noindent {\rm (iii)} If $(\Mfrak,\Phi)$ is a $\phi$-module of finite height over $\sA_X^{\rm [0,1)}$ as in Definition~\ref{Def4.2} and $(\Ncal,\Phi)=(\Mfrak,\Phi)\otimes_{\sA_X^{\rm [0,1)}}\sB_X$, then 
\[U=X^{\rm adm}=\{x\in X\mid \rk_{\kappa(x)^+}\Hom_{\sA_X^{[0,1)}\otimes \kappa(x),\Phi}(\Mfrak\otimes \kappa(x),\wtsA_X^{[0,1)}\otimes \kappa(x)) =d\}\]
and 
\[\sH\!om_{\sA^{[0,1)}_U,\Phi}(\Mfrak|_U,\wtsA_U^{[0,1)})\otimes_{\Z_p}\Q_p=\sH\!om_{\sB_U,\Phi}(\Mfrak|_U\otimes_{\sA_U^{[0,1)}}\sB_U,\wt\sB_U)\]
as families of $\sG_{K_\infty}$-representations on $U=X^{\rm adm}$.
\end{theo}
\begin{proof}
This is a summary of \cite[Proposition 8.20, Corollary 8.21, Proposition 8.22 and Proposition 8.23]{families}. 
\end{proof}
Given a cocharacter $\mu$ as in $(\ref{mu})$, the stack $\sH_{\phi,N,\preceq \mu}$ is the stack quotient of $P_{K_0,d}\times_{\Spec \Q_p}Q_{K,d,\mu}$ by the action of the reductive group $(\Res_{K_0/\BQ_p}\GL_{d,K_0})_{E_\mu}$.
Let us denote by $\sH_{\phi,N,\preceq\mu }^{\rm red}$ the quotient of the reduced subscheme underlying $P_{K_0,d}\times_{\Spec \Q_p}Q_{K,d,\preceq\mu}$ by the induced action of $(\Res_{K_0/\BQ_p}\GL_{d,K_0})_{E_\mu}$.
Recall that $P_{K_0,d}\times_{\Spec\Q_p} Q_{K,d,\mu}$ is reduced, hence this modification will not be necessary if we restrict to the case where the Hodge type is fixed by $\mu$. 
\begin{cor}\label{universalGKinfty}
There is an open substack $\sH^{\rm red,ad,adm}_{\phi,N,\preceq\mu}\subset \sH_{\phi,N,\preceq\mu}^{\rm red, ad, int}$ and a family $\Ecal$ of $\sG_{K_\infty}$-represen\-ta\-tions on $\sH^{\rm red,ad,adm}_{\phi,N,\preceq\mu}$ such that 
\[\Ecal=(\underline{\Mcal}(D,\Phi,N,\q)\otimes_{\sB_X^{[0,1)}}\wt\sB_X)^{\rm \Phi=\id},\]
where $(D,\Phi,N,\q)$ denotes the restriction of the universal family on $\sH_{\phi,N,\preceq\mu}^{\rm red, ad}$.\\
This subspace is maximal in the following sense: If $X$ is a reduced adic space and if $\ulD'$ is a $(\phi,N)$-module with Hodge-Pink lattice over $X$ with Hodge polygon bounded by $\mu$, then the induced map $f:X\rightarrow \sH_{\phi,N,\preceq\mu}^{\rm ad}$ factors over $\sH^{\rm red, ad, adm}_{\phi,N,\preceq \mu}$ if and only if $X=X^{\rm adm}$ with respect to the family\[\underline{\Mcal}(\ulD')\otimes_{\sB_X^{\rm [0,1)}}\wt\sB.\]
In this case there is a canonical isomorphism of $\sG_{K_\infty}$-representations
\[f^\ast\Ecal=(\underline{\Mcal}(\ulD')\otimes_{\sB_X^{\rm [0,1)}}\wt\sB)^{\Phi=\id}.\]
If $L$ is a finite extension of $E_\mu$, then $\sH^{\rm red,ad,adm}_{\phi,N,\preceq\mu}(L)= \sH_{\phi,N,\preceq\mu}^{\rm red, ad, int}(L)$
\end{cor}
\begin{proof}
Let us write $X_{\preceq \mu}=(P_{K_0,d}\times_{\Spec \Q_p}Q_{K,d,\preceq\mu})^{\rm red, ad}$ for the moment. Further we denote the pullback of the universal family of vector bundles on the open unit disc to $X_{\preceq \mu}$ by $(\Mcal,\Phi,N_\nabla^\CM)=\underline{\Mcal}(D,\Phi,N,\q)$. Locally on $X_{\preceq \mu}^{\rm int}$ there exists a $\phi$-module of finite height $\Mfrak$ inside $(\Mcal,\Phi)$, at least after a Tate twist. It follows that $\Mfrak\otimes_{\sA_X^{[0,1)}}\sA_X$ is \'etale and we may apply the above theorem.
Then $X_{\preceq\mu}^{\rm adm}\subset X_{\preceq \mu}$ is invariant under the action of $(\Res_{K_0/\BQ_p}\GL_{d,K_0})_{E_\mu}$ and hence its quotient by this group is an open substack $\sH^{\rm red,ad,adm}_{\phi,N,\preceq\mu}\subset \sH_{\phi,N,\preceq\mu}^{\rm red, ad}$. Further 
\[(\Mcal\otimes_{\sB_{X_\mu^{\rm adm}}^{[0,1)}}\wt\sB_{X_\mu^{\rm adm}})^{\Phi=\id}\]
is a $(\Res_{K_0/\BQ_p}\GL_{d,K_0})_{E_\mu}$-equivariant vector bundle with $\sG_{K_\infty}$-action on $X_{\preceq\mu}^{\rm adm}$. Hence it defines a family of $\sG_{K_\infty}$-representations on $\sH^{\rm red,ad,adm}_{\phi,N,\preceq\mu}$.

The second statement is local on $X$ and hence, after locally choosing a basis of $D$, we can locally lift the morphism $f:X\rightarrow \sH_{\phi,N,\preceq\mu}^{\rm red, ad}$ to a morphism $f' :X\rightarrow X_{\preceq\mu}$ such that the pullback of $(D,\Phi,N,\q)$ on $X_{\preceq\mu}$ along $f'$ is isomorphic to $\ulD'$. Now the claim follows from Theorem \ref{summary} (ii).
\end{proof}

\section{The universal semi-stable representation}

In this section we want to construct a semi-stable $\sG_K$-representation out of the $\sG_{K_\infty}$-representation on $\sH^{\rm red,ad,adm}_{\phi,N,\preceq\mu}$ from Corollary~\ref{universalGKinfty}. This will be possible only on a part of $\sH^{\rm red,ad,adm}_{\phi,N,\preceq\mu}$. First of all we need to restrict to the open subspace where the Hodge polygon is constant. This can be seen as follows. Let $\CE$ be a family of $\sG_K$-representations on an adic space $X$. It follows from \cite[\S\,4.1]{BergerColmez} that the (generalized) Hodge-Tate weights vary continuously on $X$. Namely, they are the eigenvalues of Sen's operator $\Theta_{\rm Sen}$ constructed in \cite[before Remark~4.1.3]{BergerColmez}. The characteristic polynomial of $\Theta_{\rm Sen}$ has coefficients in $\CO_X\otimes_{\BQ_p}K$. However, with any reasonable definition of a semi-stable family $\Ecal$ the Hodge-Tate weights of $\Ecal\otimes \kappa(x)$ should be integers for all $x\in X$ and hence the Hodge-Tate weights and the Hodge polygon are locally constant on $X$. 

Secondly, Kisin~\cite[Theorem~0.1 and Corollary~1.3.15]{crysrep} showed that the universal \'etale $(\phi,N_\nabla)$-module $\ul\CM$ on $\sH_{\phi,N,\preceq\mu}^{\rm ad,int}$ from Corollary~\ref{Cor4.7} can come from a semi-stable $\sG_K$-representation only if the connection $\nabla$ has logarithmic singularities, which is equivalent to $N_\nabla^\CM$ being holomorphic; see Remark~\ref{RemConnection}(2). Therefore we have to restrict further to the closed subspace $\sH_{\phi,N,\mu}^\nabla\cap\sH^{\rm ad,adm}_{\phi,N,\mu}$ of $\sH^{\rm ad,adm}_{\phi,N,\mu}$ which is isomorphic to $\sD_{\phi,N,\mu}^{\rm ad,adm}$. 
Here $\sD_{\phi,N,\mu}^{\rm ad,adm}\subset \sD_{\phi,N,\mu}^{\rm ad}$ is the admissible locus with respect to the family defined in Remark \ref{familyonDadphiN}.

\begin{lemma}\label{Lemma8.1}
Let $\Ecal$ be a family of $\sG_{K_\infty}$-representations over a reduced adic space $X$ locally of finite type over $\Q_p$. Let $\Mfrak_1$ and $\Mfrak_2$ be two $\phi$-modules of finite height over $\sA_X^{[0,1)}$ such that $\phi$ and $\sG_{K_\infty}$-equivariant isomorphisms
\begin{equation}\label{EqLemma8.1}
\Mfrak_i\otimes_{\sA_X^{[0,1)}}\wt\sA_X^{[0,1)}[1/p]\cong \Ecal\otimes_{\CO_X}\wt\sA_X^{[0,1)}[1/p]
\end{equation}
exist for $i=1,2$. Then $\Mfrak_1[\tfrac{1}{p}]=\Mfrak_2[\tfrac{1}{p}]$ as $\sA_X^{[0,1)}[\tfrac{1}{p}]$-submodules of $\Ecal\otimes_{\CO_X}\wt\sA_X^{[0,1)}[1/p]$. In particular they are isomorphic as $\phi$-modules.
\end{lemma}

\begin{proof}
The case $X=\Spa\BQ_p$ was proven by Kisin~\cite[Proposition 2.1.12]{crysrep}. 

If $X=\Spa(A,A^+)$ for a finite free $\BZ_p$-algebra $A^+$ and $A=A^+[\tfrac{1}{p}]$, then this implies that $\Mfrak_1[1/p]$ and $\Mfrak_2[1/p]$ agree as $\sA_{\Spa(\Q_p,\Z_p)}^{[0,1)}[1/p]$-submodules (even without the $A$-action). 

For general $X$ we may work locally and assume that $\Mfrak_i\cong(\sA_X^{[0,1)})^n$. The isomorphisms \eqref{EqLemma8.1} yield a matrix $M\in\GL_n(\Gamma(X,\wt\sA_X^{[0,1)}[1/p]))$ and we must show that $M\in\GL_n(\Gamma(X,\sA_X^{[0,1)}[1/p]))$. It suffices to show that every entry $g$ of $M$ and $M^{-1}$ lies in $\Gamma(X,\sA_X^{[0,1)}[1/p])$. Multiplying the entry $g$ by a power of $p$ we can assume that it lies in $\Gamma(X,\wtsA_X^{[0,1)})$. By Lemma~\ref{LemmaPowerSeriesExp}\ref{LemmaPowerSeriesExp_C} we must check that $g(x)\in \kappa(x)^+\wh\otimes_{\Z_p}\bA^{[0,1)}$ for every rigid analytic point $x\in X$. Since $\kappa(x)$ is a finite dimensional $\BQ_p$-algebra, this was proved above.
\end{proof}

\begin{defn}\label{DefSemiStRepNeu}
Let $\Ecal$ be a family of $\sG_K$-representations of rank $d$ on an adic space $X$ locally of finite type over $\Q_p$. Denote by $\bar X$ the reduced subspace underlying $X$ and by $\bar \Ecal$ the restriction of $\Ecal$ to $\bar X$.\\
\noindent (i) The family $\Ecal$ is called \emph{crystalline with negative Hodge Tate weights} if fpqc-locally on $X$ there is a $\phi$-module $\bar \Mfrak$ of finite height over $\sA_{\bar X}^{[0,1)}$ and a $\phi$ and $\sG_{K_\infty}$-equivariant isomorphism
\begin{equation}\label{GKinftymodel}
\bar \Mfrak\otimes_{\sA_{\bar X}^{[0,1)}}\wtsA_{\bar X}^{[0,1)}[\tfrac{1}{p}]\cong \bar \Ecal\otimes_{\CO_{\bar X}}\wtsA_{\bar X}^{[0,1)}[\tfrac{1}{p}]
\end{equation}
and a $(\phi,N_\nabla)$-module $\Mcal$ over $\sB_X^{[0,1)}$ deforming $\bar\Mfrak\otimes_{\sA_{\bar X}^{[0,1)}}\sB_{\bar X}^{[0,1)}$ as a $\phi$-module such that $(\ref{GKinftymodel})$ extends to a $(\sG_K,\phi)$ equivariant isomorphism
\[\Mcal\otimes_{\sB_X^{[0,1)}}B_{\rm cris}^+\wh\otimes\Ocal_X\cong \Ecal\otimes_{\Ocal_X} B_{\rm cris}^+\wh\otimes\Ocal_X.\]
\noindent (ii) The family $\Ecal$ is called \emph{semi-stable with negative Hodge Tate weights} if fpqc-locally on $X$ there is a $\phi$-module $\bar \Mfrak$ of finite height over $\sA_{\bar X}^{[0,1)}$ and a $\phi$ and $\sG_{K_\infty}$-equivariant isomorphism
\begin{equation}\label{GKinftymodel2}
\bar \Mfrak\otimes_{\sA_{\bar X}^{[0,1)}}\wtsA_{\bar X}^{[0,1)}[\tfrac{1}{p}]\cong \bar \Ecal\otimes_{\CO_{\bar X}}\wtsA_{\bar X}^{[0,1)}[\tfrac{1}{p}]
\end{equation}
and a $(\phi,N_\nabla)$-module $\Mcal$ over $\sB_X^{[0,1)}$ deforming $\bar\Mfrak\otimes_{\sA_{\bar X}^{[0,1)}}\sB_{\bar X}^{[0,1)}$ as a $\phi$-module such that $(\ref{GKinftymodel2})$ extends to a $(\sG_K,\phi,N)$ equivariant isomorphism
\begin{equation}\label{comparison}
\Mcal\otimes_{\sB_X^{[0,1)}}B_{\rm st}^+\wh\otimes\Ocal_X\cong \Ecal\otimes_{\Ocal_X} B_{\rm st}^+\wh\otimes\Ocal_X.
\end{equation}
\noindent (iii) We say that $\Ecal$ is \emph{crystalline} (resp.~\emph{semi-stable}) if some twist of $\Ecal$ with a power of the cyclotomic character is crystalline with negative Hodge-Tate weights (resp.~semi-stable with negative Hodge-Tate weights). 
\end{defn}

\begin{rem}
The definition of being crystalline resp.~semi-stable is slightly involved. We did not define it in the usual way using only period ring $B^+_{\rm cris}\hat\otimes\Ocal_X$, as our method requires that we have a comparison isomorphism for the integral models on the open unit disc as in $(\ref{GKinftymodel})$. 
Working only with $B^+_{\rm cris}\hat\otimes\Ocal_X$ it is not clear to us whether this is automatically true.
\end{rem}
\begin{lem}
Let $X$ be an adic space locally of finite type over $\Q_p$ and let $\Ecal$ be a family of crystalline (resp. semi-stable) $\sG_K$-representations with negative Hodge-Tate weights on $X$. 
Assume that the objects $\bar\Mfrak$ and $\Mcal$ in the above Definition exist globally on $X$.
Then the following holds true:\\
\noindent (i) If $X=\Spa(A,A^+)$ for some finite dimensional $\Q_p$-algebra $A$, then $\Gamma(X,\Ecal)$ is crystalline (resp.~semi-stable) as a $\sG_K$-representation on a finite dimensional $\Q_p$-vector space and $\underline{D}(\Mcal)=D_{\rm cris}(\Gamma(X,\Ecal))$ (resp.~$=D_{\rm st}(\Gamma(X,\Ecal))$) as filtered $\phi$-modules (resp.~as filtered $(\phi,N)$-modules), compatible with the canonical $A$-action on both sides.  \\
\noindent (ii) If $Y\rightarrow X$ is any morphism of adic spaces locally of finite type, and if $\Ecal_Y$ denotes the $\sG_K$-representation on $Y$ obtained by base changing $\Ecal$, then $\Ecal_Y$ is crystalline (resp.~semi-stable).\\
\noindent (iii) 
Assume that $\Ecal$ is crystalline (resp.~semi-stable) with negative Hodge-Tate weights. The family $\Mcal$ is uniquely determined as a $(\phi,N_\nabla)$-module and in fact as a submodule of $\Ecal\otimes_{\Ocal_X}B_{\rm cris}^+\wh\otimes\Ocal_X$ (resp.~of $\Ecal\otimes_{\Ocal_X}B_{\rm st}^+\wh\otimes\Ocal_X$)
\end{lem}
\begin{proof}
\noindent (i) This follows from (the covariant formulation of) \cite[Proposition 2.1.5]{crysrep}. Note that the proof of loc.cit. implies that the morphisms in (2.1.6) of loc.cit.~are isomorphisms. The fact that the morphism is compatible with the $A$-action follows from functoriality. \\
\noindent (ii) This is obvious.\\
\noindent (iii) We only prove the crystalline case. The semi-stable case is proved along the same lines. Assume that $X=\Spa(A,A^+)$ is affinoid and that there are two $\sB_X^{[0,1)}$-modules $\Mcal_1$ and $\Mcal_2$ as in the definition.
We set $D_i=\underline{D}(\Mcal_i)$ and consider the morphisms
\[D_i\longrightarrow D_i\otimes_{\Ocal_X}B_{\rm cris}^+\wh\otimes\Ocal_X\longrightarrow \Mcal_i\otimes_{\sB_X^{[0,1)}}B_{\rm cris}^+\wh\otimes\Ocal_X=\Ecal\otimes_{\Ocal_X}B_{\rm cris}^+\wh\otimes\Ocal_X.\]
As these morphisms are compatible with the $\sG_K$-action (which is of course trivial on $D_i$) we obtain a morphism
\[\alpha_i:D_i\longrightarrow  (\Ecal\otimes_{\Ocal_X}B_{\rm cris}^+\wh\otimes\Ocal_X)^{\sG_K}.\]
Now both sides are locally on $X$ free as $\Ocal_X\otimes_{\Q_p}K_0$-modules. To see this on the right hand side use the equality \[\Mcal_i\otimes_{\sB_X^{[0,1)}}B_{\rm cris}^+\wh\otimes\Ocal_X=\Ecal\otimes_{\Ocal_X}B_{\rm cris}^+\wh\otimes\Ocal_X\]
and apply Proposition \ref{periodsheavesatpoints}.
Now the construction of this map is functorial and for each quotient $A\twoheadrightarrow A'$ onto a finite dimensional $\Q_p$-algebra $A'$  the induced map  
\[\alpha_{i,A'}:D_i\otimes_A A'\longrightarrow (\Ecal\otimes_{\Ocal_X}B_{\rm cris}^+\wh\otimes\Ocal_X)^{\sG_K}\otimes_{A}A'=((\Ecal\otimes_A A')\otimes_{\Ocal_X}B_{\rm cris}^+\wh\otimes\Ocal_X)^{\sG_K}.\]
is an isomorphism. It follows that $\alpha$ is an isomorphism for $i=1,2$ and hence $D=D_1=D_2$ is uniquely determined as a $\phi$-submodule of $\Ecal\otimes_{\Ocal_X}B_{\rm cris}^+\wh\otimes\Ocal_X$. 
In particular we have shown that $\Mcal_i[1/\lambda]$ is uniquely determined as a submodule of $\Ecal\otimes_{\Ocal_X}B_{\rm cris}\wh\otimes\Ocal_X$.

It remains to prove that the two filtrations on $D=D_1=D_2$ are the same. Assume this is not the case. Then there exists a surjection $A\twoheadrightarrow A'$ onto a finite dimensional $\Q_p$-algebra $A'$ such that the filtrations on $D\otimes_A A'$ induced by $D_1$ and $D_2$ do not agree.
Replacing $A$ by $A'$ we may assume that $A'$ is a finite dimensional $\Q_p$-algebra. However, in this case (i) implies that $\Mcal_1=\Mcal_2$ (as submodules of $\Ecal\otimes_{\Q_p}B_{\rm cris}^+$) and hence the filtrations on $D_1$ and $D_2 $ coincide.

\end{proof}
\begin{rem}
Let $\Ecal$ be a crystalline representation with negative Hodge-Tate weights. Then fpqc-locally on $X$ we have associated a $(\phi,N_\nabla)$-module $\Mcal$ over $\sB_X^{[0,1)}$ as in Definition \ref{DefSemiStRepNeu}. By the uniqueness result established in the previous lemma and fpqc descent this $(\phi,N_\nabla)$-module in fact descends to $X$.
The same remark applies to semi-stable representations as well.
\end{rem}
Using this remark we can make the following definition:
\begin{defn} Let $X$ be an adic space locally of finite type over $\Q_p$ and let $\Ecal$ be a family of $\sG_K$-representations on $X$.\\
(i) Assume that $\Ecal$ is crystalline with negative Hodge-Tate weights and let $\Mcal$ as in Definition \ref{DefSemiStRepNeu}. Then define $D_{\rm cris}(\Ecal)=\underline{D}(\Mcal)$.\\
(ii) Assume that $\Ecal$ is semi-stable with negative Hodge-Tate weights and let $\Mcal$ as in Definition \ref{DefSemiStRepNeu}. Then define $D_{\rm st}(\Ecal)=\underline{D}(\Mcal)$.\\
(iii) Assume that $\Ecal$ is crystalline and that its twist $\Ecal(i)$ is crystalline with negative Hodge-Tate weights for some $i\in \Z$. Then define $D_{\rm cris}(\Ecal)=D_{\rm cris}(\Ecal(i))(-i)$.\\
(iv) Assume that $\Ecal$ is semi-stable and that its twist $\Ecal(i)$ is semi-stable with negative Hodge-Tate weights for some $i\in \Z$. Then define $D_{\rm st}(\Ecal)=D_{\rm st}(\Ecal(i))(-i)$.
\end{defn}
\begin{rem}
Obviously the last two parts of the definition are independent of the choice of $i$ such that $\Ecal(i)$ has negative Hodge-Tate weights. 
\end{rem}
The above defines a functor from the category of crystalline representations on $X$ to the category of filtered $\phi$-modules over $X$. Moreover it is a direct consequence of the definition that for every morphism $f:Y\rightarrow X$ and any family of crystalline $\sG_K$-representations on $X$ we have 
\[D_{\rm cris}(f^\ast \Ecal)=f^\ast D_{\rm cris}(\Ecal).\]
The same remark applies to the semi-stable case as well.

\begin{defn}
Let $\mu$ be a cocharacter as in \eqref{mu}, let $E_\mu$ be its reflex field, and let $X$ be an adic space locally of finite type over $E_\mu$. We say that a crystalline, resp.\ semi-stable $\sG_K$-representation $\Ecal$ over $X$ has \emph{constant Hodge polygon equal to $\mu$} if the $K$-filtered $\phi$-module $D_{\rm cris}(\Ecal)$, resp.\ $D_{\rm st}(\Ecal)$ over $X$ has this property.
\end{defn}

It is obvious from the definition that $D_{\rm st}$ defines a functor from the category of semi-stable representations with constant Hodge polygon $\mu$ over an adic space $X$ to the category of $K$-filtered $(\phi,N)$-modules over $X$ with constant Hodge polygon $\mu$ and similarly for crystalline representations.  

\begin{rem}
Let $\Ecal$ be a crystalline (resp. semi-stable) representation over $X$ with negative Hodge-Tate weights and let $\Mcal$ be as in Definition \ref{DefSemiStRepNeu}. We write $D=D_{\rm cris}(\Ecal)=\underline{D}(\Mcal)$. 
Then $$D\otimes_{\Ocal_X\otimes_{\Q_p}K_0}\sB_X^{[0,1)}[1/\lambda]\cong \Mcal\otimes_{\sB_X^{[0,1)}}\sB_X^{[0,1)}[1/\lambda]$$ and hence, as $\lambda$ is invertible in $B_{\rm cris}$, we obtain a $(\sG_K,\phi)$-equivariant isomorphism
\[D_{\rm cris}(\Ecal)\otimes_{\Ocal_X\otimes_{\Q_p}K_0}\Ocal_X\wh\otimes B_{\rm cris}\cong \Ecal\otimes_{\Ocal_X}\Ocal_X\wh\otimes B_{\rm cris}.\]
Similarly, if $\Ecal$ is semi-stable, we obtain a $(\sG_K,\phi,N)$-equivariant isomorphism
\[D_{\rm st}(\Ecal)\otimes_{\Ocal_X\otimes_{\Q_p}K_0}\Ocal_X\wh\otimes B_{\rm st}\cong \Ecal\otimes_{\Ocal_X}\Ocal_X\wh\otimes B_{\rm st}.\]
Using twists by the cyclotomic character, we find that the same holds true also for crystalline (resp.~semi-stable) representations with arbitrary Hodge-Tate weights. 

Moreover, if we had defined (the correct) filtration on $\Ocal_X\wh\otimes B_{\rm st}$ this morphisms would also respect filtrations. However, as we will not explicitly make use of this, we did not carefully define the filtrations. 
\end{rem}

\begin{lem}\label{ssN=0cryst}
Let $\Ecal$ be a family of $\Gcal_K$-representations on an adic space $X$ locally of finite type over $\Q_p$. Then $\Ecal$ is crystalline if and only it is semi-stable and the monodromy $N$ on $D_{\rm st}(\Ecal)$ vanishes. In this case we have $D_{\rm st}(\Ecal)=D_{\rm cris}(\Ecal)$ as subobjects of $\Ecal\otimes_{\Ocal_X}(\Ocal_X\hat\otimes B_{\rm st})$.
\end{lem}
\begin{proof}
We may assume that $\Ecal$ has negative Hodge-Tate weights and that there exists some $\Mcal$ as in Definition \ref{DefSemiStRepNeu}.

Assume that $\Ecal$ is semi-stable with vanishing monodromy.  As the isomorphism $(\ref{comparison})$ is equivariant for the action of $N$ the claim follows after taking $N=0$ on both sides. 

Conversely, let us assume that $\Ecal$ is crystalline. Then obviously $\Ecal$ is semi-stable and using the definition of $D_{\rm st}$ we see immediately that $N=0$ on $D_{\rm st}(\Ecal)$. 
\end{proof}

\begin{rem}
In his paper \cite{Kisindeform}, Kisin has used the techniques from \cite{crysrep} to construct what is called \emph{(potentially) semi-stable deformation rings}. 
Fix a continuous representation $\bar\rho:\Gcal_K\rightarrow \GL_n(\Fbb)$ with $\Fbb$ a finite extension of $\Fbb_p$ as well as a set of labeled Hodge-Tate weights \[{\bf k}=\{k_{i,\tau}, i=1,\dots,n, \tau:K\hookrightarrow\bar\Q_p\}.\]
 Given these data, Kisin constructs a quotient $R_{\bar\rho}^{{\bf k}}$ of the universal framed deformation ring\footnote{Note that our notations here differ from Kisin's.} $R_{\bar\rho}$ of $\bar\rho$ such that a point
 \[\Spec L\longrightarrow \Spec R_{\bar\rho}\]
 with $L$ a finite extension of $\Q_p$, factors over $\Spec R_{\bar\rho}^{{\bf k}}$ if and only if the corresponding Galois representation is semi-stable\footnote{There is a similar version with \emph{crystalline} instead of \emph{semi-stable}} with labeled Hodge-Tate weights ${\bf k}$. 
 As the defining condition for $R_{\bar\rho}^{\bf k}$ is only formulated for points, the ring $R_{\bar\rho}^{\bf k}$ is reduced by definition (once it is known to exists). Kisin moreover shows that for every finite dimensional $\Q_p$-algebra $A$ a morphism \[\Spec A\longrightarrow \Spec R_{\bar\rho}\]
factors over $\Spec R_{\bar\rho}^{{\bf k}}$ if and only if the corresponding representation $\rho:\Gcal_K\rightarrow \GL_n(A)$ is semi-stable. 

Our construction differs from Kisin's strategy in the following way: Kisin starts with a family of Galois representations on \emph{integral} level and cuts out the locus in the generic fiber where the representations are semi-stable. In contrast to this we start with a family of $p$-adic Hodge structures in characteristic zero and cut out the locus where this family of $p$-adic Hodge structures comes from a Galois representation. 
 
On the other hand after having constructed a universal family in our case, we can compare the outcome of this construction to Kisin's deformation space again. This is done in Proposition $\ref{comparetoKisin}$ below.
\end{rem}

\begin{lem} \label{Dstff} Let $X$ be an adic space locally of finite type over $\Q_p$ and let $\Ecal$, $\Ecal_1$ and $\Ecal_2$ be families of semi-stable representations. \\
\noindent {\rm (i)}
 Assume that $\Ecal$ has negative Hodge-Tate weights. Then there is a canonical isomorphism
\[\Ecal\longrightarrow (\underline{\Mcal}(D_{\rm st}(\Ecal))\otimes_{\sB_X^{[0,1)}}(\Ocal_X\wh\otimes B_{\rm st}^+))^{\phi=\id, N=0}.\]
\noindent {\rm (ii)} One has $\Ecal_1\cong \Ecal_2$ if and only if $D_{\rm st}(\Ecal_1)\cong D_{\rm st}(\Ecal_2)$.
\end{lem}
\begin{proof}
(i) Let us write $\Mcal={\Mcal}(D_{\rm st}(\Ecal))$. Then by definition fpqc-locally on $X$ we obtain an isomorphism
\[\Mcal\otimes_{\sB_X^{[0,1)}}B_{\rm st}^+\wh\otimes\Ocal_X\cong \Ecal\otimes_{\Ocal_X} B_{\rm st}^+\wh\otimes\Ocal_X.\]
Then locally on $X$ the claim follows by applying the invariants on both sides and using Proposition \ref{periodsheavesatpoints}.
The construction of this morphism is obviously compatible with the descent data and hence descends to $X$. \\
(ii) After twisting with powers of the cyclotomic character, we may assume that $\Ecal_1$ and $\Ecal_2$ have negative Hodge-Tate weights. Then second part is a direct consequence of the first.
\end{proof}
\begin{prop}\label{semistabfamily}
Let $X$ be a reduced adic space locally of finite type over $E_\mu$ and let $\ulD=(D,\Phi,N,\Fcal^\bullet)\in\sD_{\phi,N,\mu}^{\rm an,adm}(X)$. 
Then there is a family of semi-stable representations $\Ecal$ on $X$ such that $D_{\rm st}(\Ecal)=(D,\Phi,N,\Fcal^\bullet)$. 
Moreover, $\Ecal$ is canonically identified with the sub-representation 
\[(\underline{\Mcal}(D_{\rm st}(\Ecal))\otimes_{\sB_X^{[0,1)}}(\Ocal_X\wh\otimes B_{\rm st}^+))^{\phi=\id, N=0}\]
of $\underline{\Mcal}(D_{\rm st}(\Ecal))\otimes_{\sB_X^{[0,1)}}(\Ocal_X\wh\otimes B_{\rm st}^+)$.
\end{prop}
\begin{proof}
After twisting with powers of the cyclotomic character and after changing $\mu$ accordingly, we may assume that the Hodge-Tate weights defined by $\mu$ are negative. 
Let us write $\Mfrak$ for the choice of a $\phi$-module of finite height over $\sA_X^{[0,1)}$ and a family $\Ecal$ of $\sG_{K_\infty}$-representations such that there is a $(\phi,\sG_{K_\infty})$-equivariant isomorphism
\[\Ecal\otimes_{\Ocal_X}\wt\sA_X^{[0,1)}[1/p]\cong \Mfrak\otimes_{\sA_X^{[0,1)}}\wt\sA_X^{[0,1)}.\]
Such a module exists fpqc locally on $X$ by definition of the admissible locus and Theorem \ref{summary} (iii).

This isomorphism extends to an isomorphism 
\[\Ecal\otimes_{\Ocal_X}(\Ocal_X\wh\otimes B_{\rm st}^+)\cong \Mcal\otimes_{\Bcal_X^{[0,1)}}(\Ocal_X\wh\otimes B_{\rm st}^+)\]
that is still equivariant for the actions of $\phi$ and $\sG_{K_\infty}$. According to the definition of a semi-stable representation we have to prove that the $\sG_{K_\infty}$ action on $\Ecal$ extends to an action of $\sG_K$ and that the above isomorphism is equivariant for $\sG_K$. 
As $\Ecal$ embeds into the left hand side, it is enough to show that it is stabilized by the $\sG_K$-action on the right hand side. 
After localization we may assume that $X=\Spa(A,A^+)$ is affinoid and that $\Ecal$ is the trivial vector bundle on $X$. After choosing a basis of $\Ecal$ let $g\in \sG_K$ and denote by $M\in {\rm Mat}_{n\times n} (\Gamma(X,\Ocal_X\wh\otimes B_{\rm st}^+))$ the matrix of the $g$-action with respect to this basis. We have to show that this matrix has entries in $A$. 
However, $M\otimes_A\kappa(x)$ has entries in $\kappa(x)$ for all classical points $x\in X$ by \cite[Proposition 2.1.5]{crysrep} (note that the proof of that proposition implies that the arrows in (2.1.6) of loc.~cit.~are isomorphisms). 
It now follows from Lemma~\ref{LemmaPowerSeriesExp}\ref{LemmaPowerSeriesExp_B} (and the remark following that lemma) that $M$ has entries in $A$. 

This proves the existence of $\Ecal$ fpqc-locally on $X$. In order to finish the proof, we just notice that our construction defines descend data on $\Ecal$ that are compatible with the isomorphisms 
\[D_{\rm st}(\Ecal)\longrightarrow (D,\Phi,N,\Fcal^\bullet)\]
and the descend data on the latter. Hence both $\Ecal$ as well as the isomorphism descent. 
\end{proof}

Recall that the stack $\sD_{\phi,N,\mu}^{\rm ad}$ is the quotient of the adic space $X_\mu$ associated to $P_{K_0,d}\times \Flag_{K,d,\mu}$ by the action of the group $(\Res_{K_0/\BQ_p}\GL_{d,K_0})_{E_\mu}$ and consider the open subspace $X_\mu^{\rm adm}\subset X_{\mu}^{\rm int}\subset X_\mu$. This subset is stable under the action of $(\Res_{K_0/\BQ_p}\GL_{d,K_0})_{E_\mu}$ and we write $\sD_{\phi,N,\mu}^{\rm ad, adm}$ for the quotient of $X_\mu^{\rm adm}$ by this action. 
\begin{prop}\label{factorsoverDadm}
Let $\mu$ be a cocharacter as in \eqref{mu} and let $E_\mu$ be its reflex field. Let $X$ be a reduced adic space locally of finite type over $E_\mu$ and let $\Ecal$ be a family of semi-stable $\sG_K$-representations on $X$ with constant Hodge polygon equal to $\mu$.  Then the morphism 
\[X\longrightarrow \sD_{\phi,N,\mu}^{\rm ad }\]
induced by the $K$-filtered $(\phi,N)$-module $D_{\rm st}(\Ecal)$ factors over $\sD_{\phi,N,\mu}^{\rm ad, adm}$.
\end{prop}
\begin{proof}
By Definition \ref{DefSemiStRepNeu}, there is (locally on $X$) an $\sA_X^{[0,1)}$-module $\Mfrak$ such that $\Mfrak\otimes_{\sA_X^{[0,1)}}\sB_X^{[0,1)}=\underline{\Mcal}(D)$, where $D=D_{\rm st}(\Ecal)$ is the filtered $(\phi,N)$-module on $X$ defining the morphism $X\longrightarrow \sD_{\phi,N,\mu}^{\rm ad }$. Moreover by definition 
\[\Mfrak\otimes_{\sA_X^{[0,1)}}\wt\Acal_X^{[0,1)}[1/p]\cong \Ecal\otimes_{\Ocal_X}\wt\Acal_X^{[0,1)}[1/p]\]
equivariant for the action of $\phi$ and $\sG_{K_\infty}$. 
In particular this implies that $f$ factors over the admissible locus. 
\end{proof}

\begin{theo}\label{ThmUnivFamily}
There is a family $\Ecal^{\rm univ}$ of semi-stable $\sG_K$-representations on $\sD_{\phi,N,\mu}^{\rm ad,adm}$ such that $D_{\rm st}(\Ecal)=(D,\Phi,N,\Fcal^\bullet)$ is the universal family of filtered $(\phi,N)$-modules on $\sD_{\phi,N,\mu}^{\rm ad,adm}$.
This family is universal in the following sense: Let $X$ be an adic space locally of finite type over $E_\mu$ and let $\Ecal'$ be a family of semi-stable $\mathscr{G}_K$-representations on $X$ with constant Hodge polygon equal to $\mu$. Then there is a unique morphism $f:X\rightarrow \sD_{\phi,N,\mu}^{\rm ad,adm}$ such that $\Ecal'\cong f^\ast\Ecal$ as families of $\sG_K$-representations.
\end{theo}
\begin{proof}
The existence of the family $\Ecal$ follows by applying Proposition \ref{semistabfamily} to the family $(\Mfrak,\Phi)$ of $\phi$-modules of finite height over $\sA^{[0,1)}$ on
\[Y=(P_{K_0,d}\times {\rm Flag}_{K,d,\mu})^{\rm ad, adm}.\] 
As the construction is obviously functorial, this vector bundle is equivariant for the action of the group $(\Res_{K_0/\BQ_p}\GL_{d,K_0})_{E_\mu}$ and hence defines the desired family of semi-stable $\sG_K$-representations on $\sD_{\phi,N,\mu}^{\rm ad,adm}$. Further the isomorphism $D_{\rm st}(\Ecal)\cong (D,\Phi,N,\Fcal^\bullet)$ on $Y$ is by construction equivariant under the action of $(\Res_{K_0/\BQ_p}\GL_{d,K_0})_{E_\mu}$ and hence descends to $\sD_{\phi,N,\mu}^{\rm ad,adm}$.

Now let $X$ be as above. The $K$-filtered $(\phi,N)$-module $D_{\rm st}(\Ecal')$ defines a morphism $f:X\rightarrow \sD_{\phi,N,\mu}^{\rm ad}$. This map factors over $\sD_{\phi,N,\mu}^{\rm ad, adm}$ by Proposition \ref{factorsoverDadm} as factoring over an open subspace may be check on the reduced space underlying $X$. Further we have isomorphisms $D_{\rm st}(\Ecal')\cong f^\ast D_{\rm st}(\Ecal)\cong D_{\rm st}(f^\ast\Ecal)$.
Now the claim follows from Lemma \ref{Dstff}.
\end{proof}

\begin{cor}\label{CorUnivCrystFamily}
There is a family $\Ecal$ of crystalline $\sG_K$-representations on $\sD_{\phi,\mu}^{\rm ad,adm}$ such that $D_{\rm cris}(\Ecal)=(D,\Phi,\Fcal^\bullet)$ is the universal family of filtered $\phi$-modules on $\sD_{\phi,\mu}^{\rm ad,adm}$.
This family is universal in the following sense: Let $X$ be an adic space locally of finite type over $E_\mu$ and let $\Ecal'$ be a family of crystalline $\mathscr{G}_K$-representations on $X$ with constant Hodge polygon $\mu$. Then there is a unique morphism $f:X\rightarrow \sD_{\phi,\mu}^{\rm ad,adm}$ such that $\Ecal'\cong f^\ast\Ecal$ as families of $\sG_K$-representations.
\end{cor}

\begin{proof}
This is a direct consequence of the discussion of the semi-stable case in the Theorem above and Lemma $\ref{ssN=0cryst}$.
\end{proof}

Let us compare this result to the construction of the universal semi-stable deformation rings as in \cite{Kisindeform}. Fix a continuous representation 
\[\bar\rho:\Gcal_K\longrightarrow \GL_n(\Fbb)\]
with $\Fbb$ a finite extension of $\Q_p$ and write $R_{\bar\rho}$ for the universal framed deformation ring of $\bar\rho$. Further we write $R_{\bar\rho}^{\bf k}$ for the quotient of $R_{\bar\rho}$ constructed in \cite[Theorem 2.5.5]{Kisindeform}. 
Moreover let us write $\wt\sD_{\phi,N,\mu}^{\rm ad,adm}$ for the stack over $\sD_{\phi,N,\mu}^{\rm ad,adm}$ parametrizing trivializations of the universal semi-stable representation constructed in Theorem $\ref{ThmUnivFamily}$, i.e. for the stack that assigns to $f:S\rightarrow \sD_{\phi,N,\mu}^{\rm ad,adm}$ the set of isomorphisms $\Ocal_S^n\cong f^\ast \Ecal^{\rm univ}$. 
Note that $\wt\sD_{\phi,N,\mu}^{\rm ad,adm}$ actually is representable by an adic space locally of finite type over $\Q_p$ (resp.~over the reflex field of $\mu$). We write $\wt\sD_{\phi,N,\mu}^{\rm ad,adm,+}$ for the open subspace where the canonical representation
 \[\rho^{\rm univ}:\Gcal_K\longrightarrow  \GL(\Gamma(\wt\sD_{\phi,N,\mu}^{\rm ad,adm},\Ecal))\longrightarrow \GL_n(\Gamma(\wt\sD_{\phi,N,\mu}^{\rm ad,adm},\Ocal))\] 
 factors over \[ \GL_n(\Gamma(\wt\sD_{\phi,N,\mu}^{\rm ad,adm},\Ocal^+))\subset  \GL_n(\Gamma(\wt\sD_{\phi,N,\mu}^{\rm ad,adm},\Ocal)).\]
 Note that this really defines an open subspace as the group $\Gcal_K$ is topologically finitely generated (and hence we only need to check for finitely many elements of $\sG_K$ whether the corresponding matrix has bounded entries). 
 
 Having fixed $\bar\rho$ we can cut out an open subspace $\wt\sD_{\phi,N,\mu}^{\rm ad,adm,+}(\bar\rho)$ by demanding that the composition
 \[\Gcal_K\longrightarrow \GL_n(\Gamma(\wt\sD_{\phi,N,\mu}^{\rm ad,adm},\Ocal^+))\longrightarrow \GL_n(\Gamma(\wt\sD_{\phi,N,\mu}^{\rm ad,adm},\Ocal^+/\Ocal^{++})) \]
 is equal to $\bar\rho$. Here $\Ocal^{++}\subset \Ocal^+$ denotes the ideal of topologically nilpotent elements. More precisely, given any affinoid open subset $U=\Spa(A,A^+)\subset \wt\sD_{\phi,N,\mu}^{\rm ad,adm,+}$ we have a canonical family of $\Gcal_K$-representations on the reduced special fiber $\Spec A^+/A^{++}$ of $\Spf A$ (where $A^{++}\subset A^+$ is the ideal of topologically nilpotent elements), namely 
\[\Gcal_K\longrightarrow \GL_n(A^+)\longrightarrow \GL_n(A^+/A^{++}).\]
Then we let $U(\bar\rho)\subset U$ denote the tube over the Zariski closed subset of $\Spec A^+/A^{++}$ where this composition is equal to $\bar\rho$ (resp.~the base change of $\bar\rho$ to $A^+/A^{++}$). This construction is  obviously compatible with localization on the generic fiber (i.e.~with replacing $\Spf A^+$ by an affine open subset of an admissible blow up) and hence the pieces $U(\bar\rho)$ glue together to give $\wt\sD_{\phi,N,\mu}^{\rm ad,adm,+}(\bar\rho)$.

 Moreover the restriction of $\rho^{\rm univ}$ to $\wt\sD_{\phi,N,\mu}^{\rm ad,adm,+}(\bar\rho)$ induces by construction a morphism to $(\Spf R_{\bar\rho})^{\rm ad}$. 
\begin{prop}\label{comparetoKisin}
The canonical morphism
\begin{equation}\label{comparewKisin}
\wt \sD_{\phi,N,\mu}^{\rm ad,adm,+}(\bar\rho)\longrightarrow (\Spf R_{\bar\rho})^{\rm ad}
\end{equation}
induces an isomorphism
\[\wt\sD_{\phi,N,\mu}^{\rm ad,adm,+}(\bar\rho)\cong (\Spf R_{\bar\rho}^{\bf k})^{\rm ad}\]
\end{prop}
\begin{proof}
It follows from \cite[Theorem 2.5.5]{Kisindeform} and the reducedness of the source that the morphism factors over $(\Spf R_{\bar\rho}^{\bf k})^{\rm ad}$ and is a bijection on $L$-valued points. Now \cite[Theorem 2.5.5]{Kisindeform} again and the functorial description of the left hand side show that the morphism is an isomorphism on $A$-valued point for all finite dimensional $\Q_p$-algebras $A$. As the left hand side is known to be representable we find that the morphism  
\[f:\wt\sD_{\phi,N,\mu}^{\rm ad,adm,+}(\bar\rho)\longrightarrow  (\Spf R_{\bar\rho}^{\bf k})^{\rm ad}\]
is a smooth and bijective map of adic spaces locally of finite type over $\Q_p$. Especially it is \'etale and hence locally given by the composition of an open embedding with a finite \'etale morphism. As $f$ is bijective on $L$-valued points, the finite \'etale morphism has to be of degree $1$, i.e.~an isomorphism. We deduce that $f$ is an open embedding. We conclude that $f$ is an isomorphism by constructing a continuous section to $f$.

Indeed, Kisin's construction in \cite[(2.5)]{Kisindeform} consists of two steps: first he constructs a quotient $A$ of $R_{\bar\rho}$ where the restriction of the universal $\sG_K$ representation to $\sG_{K_\infty}$ is defined by a $\phi$-module $\Mfrak$ over $A_W\dbl u \dbr$, that is by a $\Acal_{(\Spf A)^{\rm ad}}^{[0,1)}$-module of finite height. Let us write $X=(\Spf A)^{\rm ad}$ and $\Ecal$ for the restriction of the universal $\sG_K$-representation to $X$. 
Then $R_{\bar\rho}^{\bf k}$ is the quotient of $A$, defined by the condition that the isomorphism  
\[\Mfrak\otimes_{\sA_{X}^{[0,1)}}\wtsA_{X}^{[0,1)}[\tfrac{1}{p}]\cong \Ecal\otimes_{\CO_{X}}\wtsA_{X}^{[0,1)}[\tfrac{1}{p}]\]
extends to a $(\sG_K,\phi,N)$-equivariant isomorphism
\[\Mfrak\otimes_{\sA_{X}^{[0,1)}}(B\otimes_{\Q_p} B^+_{\rm st})\cong \Ecal\otimes_{\CO_{X}}B\otimes_{\Q_p} B^+_{\rm st}\]
for every map $R_{\bar\rho}^{\bf k}\rightarrow B$ to a finite dimensional $\Q_p$-algebra. 
Using Lemma \ref{LemmaPowerSeriesExp}\ref{LemmaPowerSeriesExp_B} and the matrices of the action of $g\in \sG_K$ (resp.~of $\phi$ and $N$) in some chosen basis, we deduce that hence the induced isomorphism 
\[\Mfrak\otimes_{\sA_{X}^{[0,1)}}(\Ocal_X\wh\otimes B^+_{\rm st})\cong \Ecal\otimes_{\CO_{X}}\Ocal_X\wh\otimes B^+_{\rm st}\]
is equivariant for the actions of $(\sG_K,\phi,N)$. In particular the family of Galois representations on $(\Spf R_{\bar\rho}^{\bf k}) ^{\rm ad}$ is semi-stable according to our definition. 

Hence we obtain a canonical morphism \[(\Spf R_{\bar\rho}^{\bf k})^{\rm ad}\longrightarrow \sD_{\phi,N,\mu}^{\rm ad, adm}.\] As $\Ecal$ comes with a trivialization of an $\Gcal_K$-stable $\Ocal^+$-lattice inside $\Ecal$ this morphism canonically lifts to $\wt\sD_{\phi,N,\mu}^{\rm ad,adm,+}(\bar\rho)$ and defines a morphism that is continuous set-theoretically and a section to $f$. As $f$ already is known to be an open embedding it is enough to conclude. 
\end{proof}
\begin{rem}
We note that Kisin's description of the semi-stable deformation rings is a priori quite different: a family of Galois representations over some affinoid algebra $A$ is crystalline (resp.~semistable) in Kisin's sense if it is crystalline (resp.~semistable) after the base change to each quotient of $A$ that is finite dimensional as a $\Q_p$-vector space. On the other hand we have aimed at giving a definition of a family of crystalline representations in the spirit of Fontaine (though we did not do this using filtered $\phi$-modules, but rather $\phi$-modules on the open unit disc).
As it is not so obvious how these definitions directly relate to each other we construct the morphism $(\ref{comparewKisin})$ in a slightly complicated manner. 

Note that our construction has the advantage that we no longer need to fix a framing of an integral structure inside the Galois representation. 
After this paper was written, Wang-Erickson~\cite{Erickson} extended the results of Kisin in a direct way to families that do not longer fix a framing. For such families a similar comparison with our construction should hold true. However, is seems that one can not recover the main result of \cite{Erickson} from our construction that takes place purely in the generic fiber. 
\end{rem}

We finally comment on the relation of our construction with the work of Berger and Colmez \cite{BergerColmez}. In loc.~cit.~the authors study families of $p$-adic representations parametrized by $p$-adic Banach algebras. They prove for example that in such a family the locus of point-wise crystalline (resp.~semi-stable) representations of fixed Hodge-Tate weight is a closed subspace, and there exist a family of filtered $\phi$-modules (resp.~filtered $(\phi,N)$-modules) that specializes to the filtered $(\phi,N)$-modules at each point. We deduce from the comparison with Kisin's construction that our families have the same property.  
\begin{cor}
Let $X$ be a reduced adic space locally of finite type over $\Q_p$ and let $\Ecal$ be a family of $\sG_K$-representations on $X$. We assume that (fpqc-locally on $X$) there exists a $\sG_K$-stable $\Ocal_X^+$-lattice in $\Ecal$. 
Then $\Ecal$ is a semi-stable family if and only if $\Ecal\otimes \kappa(x)$  is semi-stable for all $x\in X$. 
\end{cor}
\begin{proof}
The subspace $(\Spf R_{\bar\rho}^{\bf k})^{\rm ad}\subset (\Spf R_{\bar\rho})^{\rm ad}$ is the Zariski-closure of all classical points at which the universal Galois representation on $(\Spf R_{\bar\rho})^{\rm ad}$ is semi-stable with Hodge-Tate weight ${\bf k}$. The result hence follows from the fact that by assumption we may (locally in the fpqc-topology) construct a morphism $X\rightarrow (\Spf R_{\bar\rho})^{\rm ad}$ such that the pullback of the universal representation on $(\Spf R_{\bar\rho})^{\rm ad}$ agrees with the $\sG_K$-stable $\Ocal_X^+$ lattice in $\Ecal$. 
\end{proof}
We remark that the existence of an integral lattice is always assumed in \cite{BergerColmez}. In fact we do not know whether it automatically exists or whether this is a true condition. 
As our definition (in particular the definition of the completed sheaves of period rings) differs from the one in \cite{BergerColmez}, the relation of our construction with theirs is less clear in the non-reduced case. However, the universal case is reduced.

\section{The morphism to the adjoint quotient}\label{SectAdjQuot}

\newcommand{\bbb}{b}

As in \cite[\S\,4]{Hellmann} we consider the adjoint quotient $A/\FS_d$ where $A\subset\GL_{d,\BQ_p}$ is the diagonal torus and $\FS_d$ is the finite Weyl group of $\GL_d$. Under the morphism $c:A\to\BA_{\BQ_p}^{d-1}\times_{\BQ_p}\BG_{m,\BQ_p}$ which maps an element $g$ of $A$ to the coefficients $c_1,\ldots,c_d$ of its characteristic polynomial $\chi_g=X^d+c_1X^{d-1}+\ldots+c_d$, the adjoint quotient $A/\FS_d$ is isomorphic to $\BA_{\BQ_p}^{d-1}\times_{\BQ_p}\BG_{m,\BQ_p}=\Spec\BQ_p[c_1,\ldots,c_d,c_d^{-1}]$. Recall from \cite[\S\,4]{Hellmann} that there is a morphism
\begin{equation}\label{EqAdjQuot}
\Res_{K_0/\BQ_p}\GL_{d,K_0}\longto A/\FS_d
\end{equation}
which is invariant under $\phi$-conjugation on the source. It is defined on $R$-valued points by sending $\bbb\in(\Res_{K_0/\BQ_p}\GL_{d,K_0})(R)=\GL_d(R\otimes_{\BQ_p}K_0)$ to the characteristic polynomial of $(\bbb\cdot\phi)^f=\bbb\cdot\phi(\bbb)\cdot\ldots\cdot\phi^{(f-1)}(\bbb)$ where $f=[K_0:\BQ_p]$. This characteristic polynomial actually has coefficients in $R$, because it is invariant under $\phi$, as can be seen from the formula $\phi(\bbb\cdot\phi)^f=\bbb^{-1}\cdot(\bbb\cdot\phi)^f\cdot\phi^f(\bbb)=\bbb^{-1}\cdot(\bbb\cdot\phi)^f\cdot\bbb$. Since $\Res_{K_0/\BQ_p}\GL_{d,K_0}$ acts on itself by $\phi$-conjugation via $(g,\bbb)\mapsto g^{-1}\bbb\,\phi(g)$ and $(g^{-1}\bbb\,\phi(g)\,\phi)^f=g^{-1}\cdot(\bbb\cdot\phi)^f\cdot g$ the map \eqref{EqAdjQuot} is invariant under $\phi$-conjugation. 

Let $\mu$ be a cocharacter as in \eqref{mu}, let $E_\mu$ be its reflex field, and set $(A/\FS_d)_{E_\mu}:=A/\FS_d\times_{\BQ_p}E_\mu$. By projecting to $\Res_{K_0/\BQ_p}\GL_{d,K_0}$ we may extend $\beta$ to morphisms
\[
\xymatrix @C+1pc @R-0.5pc {
P_{K_0,d}\,\times_{\BQ_p}\,Q_{K,d,\preceq\mu} \ar[r]^{\qquad\TS\wt\alpha} \ar[d] & (A/\FS_d)_{E_\mu} \,~\ar@{=}[d] \\
\sH_{\phi,N,\preceq\mu} \ar[r]^{\TS\alpha} & (A/\FS_d)_{E_\mu} \;.
}
\]
We further obtain morphisms to $(A/\FS_d)_{E_\mu}$ from the locally closed substacks $\sH_{\phi,\preceq\mu}$, $\sH_{\phi,N,\mu}$, $\sH_{\phi,\mu}$, $\sD_{\phi,N,\mu}$, and $\sD_{\phi,\mu}$, which we likewise denote by $\alpha$. Here we view $\sD_{\phi,N,\mu}$ and $\sD_{\phi,\mu}$ as substacks of $\sH_{\phi,N,\mu}$ via the zero section from Remark~\ref{Rem2.4}(3). We also consider the adification of these morphisms.

\begin{theorem}\label{Thm7.1}
Let $\mu$ be a cocharacter as in \eqref{mu}, let $E_\mu$ be its reflex field and let $x\in(A/\FS_d)_{E_\mu}^\ad$. Then there exists an open subscheme $X$ of $\wt\alpha^{-1}(x)$ such that the weakly admissible locus in the fiber over $x$ is given by
\[
\wt\alpha^{-1}(x)^{\rm wa}\;=\;X^\ad\,.
\]
\end{theorem}

\begin{proof}
This is similar to the proof of \cite[Theorem~4.1]{Hellmann}. Let $x=(c_1,\dots,c_d)\in \kappa(x)^{d-1}\times \kappa(x)^\times$ and let $v_x$ denote the (multiplicative) valuation on $\kappa(x)$.
First note that 
\[c_d=\det\nolimits_{\kappa(x)\otimes_{\Q_p}K_0}(\bbb\cdot\phi)^f=\det\nolimits_{\kappa(x)}((\bbb\cdot\phi)^f)^{1/f}\]
and hence $\widetilde{\alpha}^{-1}(x)^{\rm wa}=\emptyset$ unless 
\[ 
v_x(c_d)^{-1/f}\cdot v_x(p)^{\tfrac{1}{ef}\sum_{\psi,j}\mu_{\psi,j}}\;=\;\lambda(\ulD)\;=\;1.
\]
In the following we will assume that this condition is satisfied. We now revert to the notation of the proof of Theorem~\ref{ThmWAOpen}. In particular we consider the projective $P_{K_0,d}$-schemes $Z_i$, the global sections $f_i\in\Gamma(Z_i,\CO_{Z_i})$, the functions $h_i$, the closed subsets 
\[Y_{i,m}=\{y\in Z_i^\ad\times Q^\ad_{K,d,\preceq\mu}\mid h_i(y)\geq m\}\]
and the proper projections $\pr_{i,m}:Y_{i,m}\to P_{K_0,d} \times_{\BQ_p} Q_{K,d,\preceq \mu}$. This time
\[S_{i,m}=\{y=(g_y,N_y,U_y,\Fq_y)\in Y_{i,m}\times_{(P_{K_0,d}\times Q_{K,d,\preceq\mu})}\widetilde{\alpha}^{-1}(x)\mid v_y(f_i(g_y,U_y))>v_y(p)^{f^2m}\}\]
is a union of connected components of $Y_{i,m}\times_{(P_{K_0,d}\times Q_{K,d,\preceq\mu})}\widetilde{\alpha}^{-1}(x)$, hence a closed subscheme and not just a closed adic subspace. This can be seen as follows: Let $\lambda_1,\dots,\lambda_d$ denote the zeros of the polynomial
\[X^d+c_1X^{d-1}+\dots+c_{d-1}X+c_d.\]
Then every possible value of the $f_i$ is a product of some of the $\lambda_i$ and hence $f_i$ can take only finitely many values. As in the proof of Theorem~\ref{ThmWAOpen}
\[
\wt\alpha^{-1}(x)^{\rm wa}\;=\;\wt\alpha^{-1}(x)\,\setminus\, \bigcup_{i,m} \pr_{i,m}(S_{i,m}),
\]
where the union runs over $1\leq i\leq d-1$ and $m\in \Z$. So $\wt\alpha^{-1}(x)^{\rm wa}$ is an open subscheme of $\wt\alpha^{-1}(x)$.
\end{proof}

\begin{cor}\label{Cor9.2}
Let $x\in (A/\FS_d)_E^{\rm ad}$ and consider the $2$-fiber product
\[\begin{xy}\xymatrix{
\alpha^{-1}(x)^{\rm wa} \ar[r]\ar[d] & \sH_{\phi,N,\preceq\mu}^{\rm ad,wa}\ar[d]^{\TS\alpha}\\
x \ar[r] & (A/W)_E^{\rm ad}.
}
\end{xy}\]
Then there exists an Artin stack in schemes $\mathfrak{A}$ over the field $\kappa(x)$ which is an open substack of $\alpha^{-1}(x)$, such that $\alpha^{-1}(x)^{\rm wa}=\mathfrak{A}^{\ad}$. The same is true for $\sH_{\phi,\preceq\mu}$, $\sH_{\phi,N,\mu}$, $\sH_{\phi,\mu}$, $\sD_{\phi,N,\mu}$, and $\sD_{\phi,\mu}$.
\end{cor}
\begin{proof}
This is an immediate consequence of Theorem~\ref{Thm7.1} and the proof of Corollary~\ref{Cor3.7}.
\end{proof}

We also determine the image of the weakly admissible locus in the adjoint quotient.

\begin{theorem}\label{ThmImageWA}
The image of $\sH_{\phi,N,\preceq\mu}^{\rm ad,wa}$ (and $\sH_{\phi,\preceq\mu}^{\rm ad,wa}$, $\sH_{\phi,N,\mu}^{\rm ad,wa}$, $\sH_{\phi,\mu}^{\rm ad,wa}$, $\sD_{\phi,N,\mu}^{\rm ad,wa}$, and $\sD_{\phi,\mu}^{\rm ad,wa}$) under the morphism(s) $\alpha$ is equal to the affinoid subdomain 
\begin{equation}\label{eqNewtonstratum}
\Bigl\{\,c=(c_1,\ldots,c_d)\in (A/\FS_d)_{E_\mu}^\ad\,\Big|\es v_c(c_i)\le v_c(p)^{^{\TS\tfrac{1}{e}\sum_\psi(\mu_{\psi,d}+\ldots+\mu_{\psi,d+1-i})}}\text{ with equality for }i=d\,\Bigr\}\,,
\end{equation}
where $v_c$ is the (multiplicative) valuation of the adic point $c$ with $v_c(p)<1$.
\end{theorem}

\begin{remark}
(1) The subset described in $(\ref{eqNewtonstratum})$ is really an affinoid subdomain. Indeed the adjoint quotient $(A/\FS_d)_{E_\mu}^\ad$ is (admissibly) covered by the (admissible) open affinoid rigid spaces (resp.\ adic spaces) $X_M=\{c=(c_1,\dots,c_d)\in (A/\FS_d)_{E_\mu}^\ad\mid v_c(c_i)\leq p^M,\ \text{for all}\ i\ \text{and}\ v_c(v_d)\geq -p^M\}$ and the subspace $(\ref{eqNewtonstratum})$ is easily seen to be a Laurent subdomain of each of these $X_M$ for $M\gg 0$.  

\medskip\noindent
(2) The morphisms $\alpha$ forget the Hodge-Pink lattice $\Fq$ (or the $K$-filtration $\CF^\bullet$) and in general their fibers contain infinitely many weakly admissible points.

\medskip\noindent
(3) Like in \cite[Proposition~5.2]{Hellmann} the affinoid subdomain of Theorem~\ref{ThmImageWA} can be described as the \emph{closed Newton stratum} of the coweight $(-\tfrac{1}{e}\sum_\psi\mu_{\psi,d}\ge\ldots\ge-\tfrac{1}{e}\sum_\psi\mu_{\psi,1})$ of $A$. By this we mean that the $\ol\BQ_p$-valued points (i.e.~the rigid analytic points) of \ref{eqNewtonstratum} coincide with the points of the corresponding Newton stratum in the sense of Kottwitz~\cite{Kottwitz06}. In \cite{Hellmann} the claim is made for all points of the corresponding Berkovich space. In the set up of adic spaces we can not rely on Kottwitz's definition of a Newton stratum for all points of the adic space, as the valuations are not necessarily rank one valuations, i.e.~the value group is not necessarily a subgroup of the real numbers. Especially the Newton strata do not cover the adic space $(A/\mathfrak{S}_d)^{\rm ad}$.

\medskip\noindent
(4) For $\sD_{\phi,\mu}^{\rm ad,wa}$ the description of the image in our Theorem~\ref{ThmImageWA} has previously been obtained by Fontaine and Rapoport \cite[Th\'eor\`em~1]{FontaineRapoport} and Breuil and Schneider~\cite[Proposition~3.2]{BreuilSchneider} on the level of $L$-valued points where in \cite{FontaineRapoport} $L$ is a complete discretely valued extension of $E_\mu$ with algebraically closed residue field. In \cite{BreuilSchneider} $L$ is a finite extension of $E_\mu$ and in addition all Hodge-Tate weights are assumed to be pairwise different. Moreover, our affinoid subdomain~\eqref{eqNewtonstratum} equals $\FS_d\backslash\mathbf{T}'_\xi$ from \cite[Corollary~2.5]{BreuilSchneider}, where $\xi$ is associated with the cocharacter $\tilde\xi:=\bigl(-\mu-(0,1,\ldots,d-1)\bigr)_\dom\in X_*(\wt T)$. Actually, both \cite{FontaineRapoport,BreuilSchneider} even prove that over an $L$-valued point $c$ in the image there is an $L$-valued point in $\wt\alpha^{-1}(c)^{\rm wa}$. This also follows from our Theorem~\ref{Thm7.1}, which shows that $\wt\alpha^{-1}(c)^{\rm wa}$ is Zariski-open in a scheme covered by affine spaces, see \eqref{sDpresentation}, because the $L$-valued points (for any infinite field $L$) lie dense in such schemes. In this way our theorem provides a new proof for \cite[Th\'eor\`em~1]{FontaineRapoport} and generalizes \cite[Proposition~3.2]{BreuilSchneider}; see Chapter~\ref{SectApplications} for more details.
\end{remark}

Before we prove the theorem we note the following 

\begin{lemma}\label{LemmaELLi}
Set $l_i:=\tfrac{1}{ef}\sum_\psi(\mu_{\psi,d}+\ldots+\mu_{\psi,d+1-i})$. Then $l_i$ equals the number $l_i$ defined in \cite[Formula~(5.2) on p.~988]{Hellmann}. If $\ulD=(D,\Phi,N,\Fq)$ is a $(\phi,N)$-module with Hodge-Pink lattice over a field $L\supset E_\mu$ whose Hodge polygon is bounded by $\mu$ and if $\ulD'=\bigl(D',\Phi|_{\phi^*D'},N|_{D'},\Fq\cap D'\otimes_{L\otimes K_0}\BdR{L}\bigr)\subset\ulD$ for a free $L\otimes_{\BQ_p}K_0$-submodule $D'\subset D$ of rank $i$ which is stable under $\Phi$ and $N$, then $t_H(\ulD')\ge l_i$.
\end{lemma}

\begin{proof}
The number $l_i$ in \cite[(5.2)]{Hellmann} was defined as follows. Write $\{\mu_{\psi,1},\ldots,\mu_{\psi,d}\}=\{x_{\psi,1},\ldots,x_{\psi,r}\}$ with $x_{\psi,j}>x_{\psi,j+1}$. Let $n_{\psi,j}:=\max\{k:\mu_{\psi,k}\ge x_{\psi,j}\}$. In particular $n_{\psi,r}=d$ and $\mu_{\psi,n_{\psi,j}}\ge x_{\psi,j}$. For $0\le i\le d$ let $m_{\psi,j}(i):=\max\{0,n_{\psi,j}+i-d\}$. So $m_{\psi,j}(0)=0$ for all $j$ and $m_{\psi,r}(i)=i$. It follows that $n_{\psi,j}\ge d-i$ if and only if $\mu_{\psi,d-i}\ge x_{\psi,j}$. Now $l_i$ was defined in \cite[(5.2)]{Hellmann} as
\[
l_i=\tfrac{1}{ef}\sum_{\psi}\Biggl(\sum _{j=1}^{r-1}(x_{\psi,j}-x_{\psi,j+1})m_{\psi,j}(i)+x_{\psi,r} m_{\psi,r}(i)\Biggr).
\]
We compute $l_{i+1}-l_i=\tfrac{1}{ef}\sum_\psi\left(\sum_{j=1}^{r-1}(x_{\psi,j}-x_{\psi,j+1})\bigl(m_{\psi,j}(i+1)-m_{\psi,j}(i)\bigr)+x_{\psi,r}\right)$. The difference $m_{\psi,j}(i+1)-m_{\psi,j}(i)$ is $1$ if $n_{\psi,j}+i-d\ge0$, that is if $x_{\psi,j}\le\mu_{\psi,d-i}$. Else $m_{\psi,j}(i+1)-m_{\psi,j}(i)$ is $0$. Therefore $l_{i+1}-l_i=\tfrac{1}{ef}\sum_\psi\mu_{\psi,d-i}$ and $l_0=0$ implies that $l_i=\tfrac{1}{ef}\sum_\psi(\mu_{\psi,d}+\ldots+\mu_{\psi,d+1-i})$.

To prove the second assertion let $s\in\Spec L\otimes_{E_\mu}\norm{K}$ be a point and let $\mu'=\mu_\ulD(s)$  be the Hodge polygon of $\ulD$ at $s$. Then $\mu_{\psi,d}+\ldots+\mu_{\psi,d+1-i}\le\mu'_{\psi,d}+\ldots+\mu'_{\psi,d+1-i}$ for all $\psi$ and all $i$ by Proposition~\ref{PropHWts}\ref{PropHWts_B}. We let $\Fp_\psi$ be the $\psi$-component of $s^*\Fp:=s^*D\otimes_{\kappa(s)\otimes K_0}\BdRplus{\kappa(s)}$ and $\Fp'_\psi$ be the $\psi$-component of $s^*\Fp':=s^*D'\otimes_{\kappa(s)\otimes K_0}\BdRplus{\kappa(s)}$. By definition of the Hodge polygon, see Construction~\ref{ConstrHodgeWts}, we can choose a $\kappa(s)\dbl t\dbr$-basis $(v_{\psi,1},\ldots,v_{\psi,d})$ of $\Fp_\psi$ such that $(t^{-\mu'_{\psi,1}}\,v_{\psi,1},\ldots,t^{-\mu'_{\psi,d}}\,v_{\psi,d})$ is a $\kappa(s)\dbl t\dbr$-basis of the $\psi$-component $\Fq_\psi$ of $s^*\Fq$. Since $\dim_{\kappa(s)\dpl t\dpr}\bigl(\Fp'_\psi[\tfrac{1}{t}]\cap\langle v_{\psi,1},\ldots,v_{\psi,n}\rangle_{\kappa(s)\dpl t\dpr}\bigr)\ge n+i-d$ for all $n$, we can find a $\kappa(s)\dbl t\dbr$-basis $(v'_{\psi,1},\ldots,v'_{\psi,i})$ of $\Fp'_\psi$ with $v'_{\psi,j}\in\langle v_{\psi,1},\ldots,v_{\psi,d+j-i}\rangle_{\kappa(s)\dbl t\dbr}$. Namely, for each $j$ we let $\bar v'_j$ be an element of $\bigl(\Fp'_\psi\cap\langle v_{\psi,1},\ldots,v_{\psi,d+j-i}\rangle_{\kappa(s)\dbl t\dbr}\bigr)\big/\langle v'_{\psi,1},\ldots,v'_{\psi,j-1}\rangle_{\kappa(s)\dbl t\dbr}$ which generates a non-zero saturated $\kappa(s)\dbl t\dbr$-submodule, and we let $v'_{\psi,j}\in\Fp'_\psi\cap\langle v_{\psi,1},\ldots,v_{\psi,d+j-i}\rangle_{\kappa(s)\dbl t\dbr}$ be a lift of $\bar v'_{\psi,j}$. Then $(v'_{\psi,1},\ldots,v'_{\psi,j})$ is linearly independent over $\kappa(s)\dpl t\dpr$ and generates a saturated $\kappa(s)\dbl t\dbr$-submodule of $\Fp'_\psi$. Using this basis we see that $t^{-\mu'_{\psi,d+j-i}}\cdot v'_{\psi,j}\in\Fq_\psi\cap\Fp'_\psi[\tfrac{1}{t}]$. This implies that $t_H(\ulD')\ge\tfrac{1}{ef}\sum_\psi\mu'_{\psi,d}+\ldots+\mu'_{\psi,d+1-i}\ge l_i$.
\end{proof}

\begin{proof}[Proof of Theorem~\ref{ThmImageWA}]
We consider the embedding of $\sD_{\phi,\mu}$ into $\sH_{\phi,\mu}$ via the zero section. Under this section $\sD_{\phi,\mu}^{\rm ad,wa}$ is contained in $\sH_{\phi,N,\preceq\mu}^{\rm ad,wa}$, $\sH_{\phi,\preceq\mu}^{\rm ad,wa}$, $\sH_{\phi,N,\mu}^{\rm ad,wa}$, $\sH_{\phi,\mu}^{\rm ad,wa}$, and $\sD_{\phi,N,\mu}^{\rm ad,wa}$ by Lemma~\ref{waundersection} or Remark~\ref{Rem3.8}. Conversely they are all contained in $\sH_{\phi,N,\preceq\mu}^{\rm ad,wa}$. Moreover, these inclusions are compatible with the morphisms $\alpha$ to $(A/\FS_d)_{E_\mu}$. 

We first claim that the affinoid subdomain is contained in the image of the weakly admissible locus for all these stacks. By the above it suffices to prove the claim for $\sD_{\phi,\mu}^{\rm ad,wa}$. In this case the claim follows from \cite[Theorem~5.5 and Proposition~5.2]{Hellmann} using Lemma~\ref{LemmaELLi}. Note that in loc.\ cit.\ only Berkovich's analytic points are treated, but the given argument works verbatim also for adic points.

Conversely let $c=(c_1,\ldots,c_d)$ be an $L$-valued point of $(A/\FS_d)_{E_\mu}^\ad$ which lies in the image of the weakly admissible locus of one of these stacks. By the above it lies in the image of $\sH_{\phi,N,\preceq\mu}^{\rm ad,wa}$. So let $\ulD\in\sH_{\phi,N,\preceq\mu}^{\rm ad,wa}(L')$ for a field extension $L'/L$, such that $\ulD$ maps to $c$. By extending the field $L'$ further we may assume that $K_0\subset L'$ and that $X^d+c_1X^{d-1}+\ldots+c_d=\prod_{j=1}^d(X-\lambda_j)$ splits into linear factors with $\lambda_j\in L'$. We claim that $v_{L'}(\prod_{j\in I}\lambda_j)\le v_{L'}(p)^{fl_i}$ for all subsets $I\subset\{1,\ldots,d\}$ of cardinality $i$. By Lemma~\ref{LemmaELLi} this implies that $c$ lies in our affinoid subdomain.

To prove the claim we use Remark~\ref{RemDecompOfD}. Then $X^d+c_1X^{d-1}+\ldots+c_d$ is the characteristic polynomial of the $L'$-endomorphism $(\Phi^f)_0$ of $D_0$ and $t_N(\ulD)=v_{L'}(\det_{\SSC L'}(\Phi^f)_0)^{1/f}$. We write $(\Phi^f)_0$ in Jordan canonical form and observe that $N_0$ maps the generalized eigenspace of $(\Phi^f)_0$ with eigenvalue $\lambda_j$ into the one with eigenvalue $p^{-f}\lambda_j$. If $I\subset\{1,\ldots,d\}$ is a subset with cardinality $i$ this allows us to find an $i$-dimensional $L'$-subspace $D'_0\subset D_0$ which is stable under $(\Phi^f)_0$ and $N_0$ such that the eigenvalues of $(\Phi^f)_0$ on $D'_0$ are of the form $(p^{-n_j}\lambda_j:j\in I)$ for suitable $n_j\in\BZ_{\ge0}$. We let $D'\subset D$ be the $(\phi,N)$-submodule corresponding to $D'_0\subset D_0$ under Remark~\ref{RemDecompOfD}. Then
\[
v_{L'}\bigl(\,\prod_{j\in I}\lambda_j\bigr)\;\le\;v_{L'}\bigl(\,\prod_{j\in I}p^{-n_j}\lambda_j\bigr)\;=\;v_{L'}\bigl(\det\nolimits_{\SSC L'}(\Phi^f)_0|_{D'_0}\bigr)\;=\;t_N(\ulD')^f\;\le\; v_{L'}(p)^{f\,t_H(\ulD')}\;\le\; v_{L'}(p)^{fl_i}
\]
by the weak admissibility of $\ulD$ and by Lemma~\ref{LemmaELLi}. This proves the theorem.
\end{proof}

\section{Applications}\label{SectApplications}
Let us mention two conjectural applications of our constructions to the $p$-adic local Langlands program. 

\smallskip\noindent
{\bf Breuil's conjecture on the locally analytic socle}\\
In \cite{Br1} and \cite{Br2} Breuil formulates a conjecture on the locally analytic principal series representations that embed into the $\rho$-isotypical part of completed cohomology (or some $p$-adically completed space of automorphic forms) for some fixed global Galois representation $\rho$ which is associated to an automorphic representation. 
The automorphic representation to which $\rho$ is associated defines a locally algebraic representation inside completed cohomology, i.e.~a representation that appears in the conjecture of Breuil and Schneider; see below. 
The conjectured existence of more locally analytic principal series representations is the representation-theoretic formulation of the existence of \emph{companion points} on eigenvarieties, i.e.~the existence of (overconvergent) $p$-adic automorphic forms (of finite slope) such that the associated Galois representation is in fact automorphic. 

These additional locally analytic representations that should conjecturally embed into completed cohomology are described by combinatorial data: the relative position of the de Rham filtration and a flag of $\phi$-stable subspaces inside $D_{\rm st}(\rho)$, i.e.~they are described completely by local data. 
In fact one can formulate a conjecture for all (potentially) semi-stable local Galois representations (not just the restrictions of global Galois representations) by replacing the completed cohomology by the candidate for the $p$-adic local Langlands correspondence as in \cite{CEGGPS}.

In joint work with Breuil and Schraen the second author establishes a link between the existence of these locally analytic principal series representations and the degenerations of certain structures from $p$-adic Hodge theory (resp.~the theory of $(\phi,\Gamma)$-modules) in rigid analytic families. The degenerations predicted by Breuil's conjecture can be constructed using precisely the universal families of semi-stable representations defined in the present article.

\bigskip\noindent
{\bfseries The Breuil-Schneider conjecture}\\
This second application is rather a speculation than a true application. As mentioned in the introduction the $p$-adic local Langlands program wants to relate on the one hand certain continuous representations of $\mathscr{G}_K$ on $n$-dimensional $L$-vector spaces for another $p$-adic field $L$, and on the other hand topologically irreducible admissible representations of ${\rm GL}_n(K)$ on finite dimensional $L$-Banach spaces. We want to explain in which sense both kinds of representations vary in families.

On the side of $\GL_n(K)$-representations, when all Hodge-Tate weights are pairwise different, Breuil, Schneider and Teitelbaum~\cite{BreuilSchneider,SchneiderTeitelbaum06} constructed a Banach-Hecke algebra $\CB$ which is the completion of the usual spherical Hecke algebra for a certain norm. This Banach-Hecke algebra is an affinoid algebra over the Galois closure $\norm{K}$ of $K/\BQ_p$, whose associated affinoid space $\Spa\CB$ is contained in a split $n$-dimensional torus $A$. Moreover, the algebra $\CB$ acts on a universal infinite dimensional locally algebraic Banach representation of $\GL_n(K)$. Breuil and Schneider also conjectured that the specialization of the universal Banach representation at any $L$-valued point of $\Spa\CB$ admits an (in general many) invariant norm(s) and proved this in some cases. Further cases were established by Sorensen~\cite{Sorensen13} and more recently many new cases were proved by Caraiani, Emerton, Gee, Geraghty, Pa{\v{s}}k{\=u}nas and Shin \cite{CEGGPS}. One might hope that the completions with respect to these norms produce the searched for irreducible admissible finite dimensional $L$-Banach representations. 

If on the Galois side one restricts to semi-stable or crystalline representations of $\mathscr{G}_K$ then we provide in this article the moduli spaces $\sD_{\phi,N,\mu}^{\rm ad,adm}$ for those. Sending a semi-stable $\mathscr{G}_K$-representation to the characteristic polynomial of its associated Frobenius defines a morphism $\alpha$ from $\sD_{\phi,N,\mu}^{\rm ad,adm}$ to the adjoint quotient $(A/\FS_n)^\ad$ which contains (an image of) the affinoid domain $\Spa\mathcal{B}$ of Breuil and Schneider. In Chapter~\ref{SectAdjQuot} we proved that the fibers of this morphism $\alpha$ are Artin stacks in schemes (Corollary~\ref{Cor9.2}) and we determined the image of $\alpha$. If all Hodge-Tate weights are pairwise different, Breuil and Schneider \cite[Proposition~3.2]{BreuilSchneider} proved that the image equals $\Spa\mathcal{B}$. Our Theorem~\ref{ThmImageWA} generalizes this to arbitrary Hodge-Tate weights. So one may now ask whether there is a relation between the fiber of the morphism $\alpha$ over an $L$-valued point of $\Spa\CB$ and the set of invariant norms on the specialization of the universal Banach representation at this point.

The reader should note that by the condition of \cite{families} that the Hodge-Tate weights lie in $\{0,1\}$ together with the condition of \cite{BreuilSchneider,SchneiderTeitelbaum06} that they are pairwise different, one was limited to $\GL_2$ for which the $p$-adic local Langlands program is established when $K=\BQ_p$; see \cite{Colmez10, Paskunas13, CDP13}. So for the application to $\GL_n$ when $n>2$ our generalization in the present article is essential.

\vspace{1cm}

\noindent
\parbox[t]{0.5\textwidth}{ 
Urs Hartl  \\ 
Universit\"at M\"unster\\
Mathematisches Institut \\
Einsteinstr.~62\\
D -- 48149 M\"unster
\\ Germany
\\[1mm]
\href{https://www.uni-muenster.de/Arithm/hartl/}{https:/\!/www.uni-muenster.de/Arithm/hartl/}
} 
\parbox[t]{0.5\textwidth}{ 
Eugen Hellmann\\
Universit\"at M\"unster\\
Mathematisches Institut \\
Einsteinstr.~62\\
D -- 48149 M\"unster\\
Germany
\\[1mm]
E-mail: \href{mailto:e.hellmann@uni-muenster.de}{e.hellmann@uni-muenster.de}
}
\end{document}